\documentclass[aap,preprint]{imsart}
\usepackage{newcommands}
\usepackage{enumitem}
\begin{document}
	
	\begin{frontmatter}
		\title{Central Limit Theorem for Gram-Schmidt Random Walk Design}
		\runtitle{CLT for GSW}
		\runauthor{Chatterjee, Dey, and Goswami}
		
		\begin{aug}
			\author[A]{\fnms{Sabyasachi}~\snm{Chatterjee}\ead[label=e1]{sc1706@illinois.edu}}, %\footnote{{Supported by NSF Grant DMS-1916375}}
			\author[B]{\fnms{Partha S. }\snm{Dey}\ead[label=e2]{psdey@illinois.edu}}%\footnote{Supported partially by Campus Research Board Grant RB23016.}
			\and
			\author[C]{\fnms{Subhajit}~\snm{Goswami}\ead[label=e3]{goswami@math.tifr.res.in}} %\footnote{{Supported partially by a grant from the Infosys Foundation}}
			\runauthor{Chatterjee, Dey, and Goswami}
			%\affiliation{University of Illinois Urbana-Champaign, University of Illinois Urbana-Champaign and Tata Institute of Fundamental Research}
			\address[A]{Department of Statistics, University of Illinois Urbana--Champaign,\printead[presep={,\ }]{e1}}
			\address[B]{Department of Mathematics, University of Illinois Urbana--Champaign,\printead[presep={,\ }]{e2}}
			\address[C]{School of Mathematics, Tata Institute of Fundamental Research,\printead[presep={,\ }]{e3}}
		\end{aug}
		\date{\today}
		
		\begin{abstract}
			We prove a central limit theorem for the Horvitz--Thompson estimator based on the Gram--Schmidt Walk (GSW) design, recently developed in~\cite{harshaw2019balancing}. In particular, we consider the version of the GSW design, which uses {\em randomized pivot order}, thereby  answering an  open question raised in the same article. We deduce this under minimal and global assumptions involving {\em only} the problem parameters, such as the (sum) potential outcome vector and the covariate matrix. As an interesting consequence of our analysis, we also obtain the precise limiting variance of the estimator in terms of these parameters, which is {\em smaller} than the previously known upper bound. The main ingredients are a simplified {\em skeletal} process approximating the GSW design and concentration phenomena for random matrices obtained from random sampling using Stein's method for exchangeable pairs.
		\end{abstract}
		
		
		\maketitle
	\end{frontmatter}
	\tableofcontents%\setcounter{tocdepth}{1}

	\section{Introduction}\label{sec:intro}
	We are interested in the statistical problem of estimating the \textit{average treatment effect} (ATE).
	This is one of the canonical problems in causal inference, and provides valuable insights into the effectiveness or impact of a particular treatment or intervention. The setup is the following.
	
	Suppose there are $n$ units or individuals with each individual $i$ having two \textit{potential outcomes} $a_i$ and $b_i$ corresponding to two possible treatments. We can think of $a_i$ and
	$b_i$ as the responses of the individual $i$ that we would have observed if we had administered treatment $A$ or $B$ (respectively) to that individual. We will denote the ATE by $\tau$, defined as
	\begin{equation}\label{eq:ht}
		\tau = \frac{1}{n} \sum_{i = 1}^{n} (a_i - b_i).
	\end{equation}
	Let us set up some notations to be used throughout. Let $\mva,\mvb \in \R^n$ denote the two 
	potential outcome vectors and $\mvmu \coloneqq \mva + \mvb$ be the {\em sum potential outcome 
		vector}.

	\subsection{Horvitz--Thompson Estimator}
	One of the classical methods to estimate the ATE is the {\em Horvitz--Thompson estimator}. 
	Let $z_i = \pm1$ denote whether treatment $A$ or $B$ is administered. The vector $\mvz = (z_1,\dots,z_n) \in \{\pm 1\}^{n}$ is called the {\em design vector}. It then just estimates the ATE by 
	taking the difference of the empirical means of the observed potential outcomes, namely,
	\begin{equation*}
		\hat{\tau} = \frac{1}{n} \left(\sum_{i: z_i = +1} \frac{a_i}{\P[z_i = 1]} -  \sum_{i: z_i = -1} \frac{b_i}{\P[z_i = -1]}\right).
	\end{equation*}
	It is well known (see~\cite{imbens2015causal}) that 
	$\hat \tau$ is unbiased
	if $\P[z_i = 1] \in (0, 1)$ for all units $i \in [n]$. In this paper, we consider designs where each unit is equally likely to receive either treatment, \ie, $\P[z_i = 1] = \P[z_i = - 1] = \tfrac12$. The 
	distribution of $\hat{\tau}$ clearly depends on the design $\mvz$. One choice for $\mvz$ is the i.i.d. 
	design where the variables $z_i$'s i.i.d.~Rademacher variables. It is not hard to check (see \eg~\cite[Lemma~1.2]{harshaw2019balancing}) that the Horvitz--Thompson estimator $\hat{\tau}_{{\rm 
			Rad}}$ based on the i.i.d. design has variance given by
	\begin{equation}\label{eq:htvar}
		\E [\hat{\tau}_{{\rm Rad}} - \tau]^2  = \var [\hat{\tau}_{{\rm Rad}}] = \frac{1}{n^2} \|\mvmu\|^2.
	\end{equation}

	The setting that we are interested in here is when in addition to our observation of one of the potential 
	outcomes $a_i$ or $b_i$ for the $i$-th unit, we also observe a covariate vector $\mvx_i \in \R^d$, for 
	each unit $i \in [n].$ A natural question that arises at this point is 
	\begin{center}
		\textit{whether it is possible to use the information in the covariates to improve the Horvitz--Thompson Estimator?} 
	\end{center}

	\subsection{Gram--Schmidt Walk Design}
	Recently,~\cite{harshaw2019balancing} proposed a solution to the above-posed question by 
	using the Gram--Schmidt Walk algorithm from~\cite{bansal2018gram} to change the sampling 
	strategy of the design vector $\mvz$. This algorithm is reviewed in Section~\ref{sec:algo}. The authors obtain a random vector $\mvz \in \{\pm 1\}^n$ (final output of the so-called Gram--Schmidt walk), which no longer consists of i.i.d.~Rademacher entries. However, they consider the same exact form of the Horvitz--Thompson estimator as in~\eqref{eq:ht}. They show the following mean squared error bound for their estimator $\hat{\tau}_{{\gsw}}$.

	\begin{theorem}\rm{\cite[Theorem~4.1]{harshaw2019balancing}}\label{thm:msebnd}
		Let $X$ be an $n \times d$ covariate matrix with rows $\{\mvx_i\}_{i = 1}^{n}$. Let $\xi := \max_{i \in 
			[n]} \norm{\mvx_i}$ and $\phi \in (0,1)$ be an algorithm parameter, called the robustness parameter, fixed 
		beforehand. Then, for the Horvitz--Thompson estimator based on the GSW design with parameter 
		$\phi$, we have the following upper bound on the mean squared error:
		\begin{equation}\label{eq:gs}
			n\E [\hat{\tau}_{{\gsw}} - \tau]^2 \leq \inf_{\mvgb \in \R^d} \left(\frac{1}{\phi n} \|\mvmu - X \mvgb\|^2 + 
			\frac{\xi^2 \|\mvgb\|^2}{(1 - \phi) n} \right).
		\end{equation}
		\begin{comment}
			In particular, we have the bound
			\[
			\E (\hat{\tau}_{{\gsw}} - \tau)^2 \leq \frac1{n^2}\mvmu^{\t} (\vB^{\t}\vB)^{-1}\mvmu.
			\]
			where  $\vB= [\sqrt{\phi}\cdot I_n, \xi^{-1}\sqrt{1-\phi}\cdot X]^{\t}$.
			\todo{check the statement?}\end{comment}
	\end{theorem}
	Suppose we choose $\mvgb =\mvgb_\ls$ such that $X \mvgb_\ls  =  \proj_{{\rm 
			ColSp}(X)}(\mvmu)$, where $\proj_{\colsp(X)}$ is the (orthogonal) projector 
	onto the column space of $X$. The first term on the right-hand side above will then scale like 
	$$\frac{1}{n}\|\mvmu  - \proj_{\colsp(X)}(\mvmu)\|^2$$ which is typically $\Theta(1)$ unless %$\mvmu$ lies ``too close'' to $\colsp(X)$. %Typically 
	the (sum potential) outcome vector $\mv\mu$ lies ``too close'' to $\colsp(X)$. 
	Also we can expect $\xi^2$ and $\|\mvgb \|^2$ to be $O(d)$ as they are norms of $d$-dimensional 
	vectors. Therefore, the second-term scales 
	like $O({d^2}/{n})$ which, if $d = o(\sqrt{n})$, is a lower order term than the first term. Therefore, in 
	such a regime, the above theorem ensures that the Horvitz--Thompson estimator 
	$\hat{\tau}_{{\gsw}}$ based on the Gram--Schmidt Walk design (setting $\phi$ very close to $1$) 
	satisfies the two qualitative properties when compared to $\hat \tau_{{\rm Rad}}$ 
	(see~\eqref{eq:htvar} above).
	\begin{enumerate}
		\item If the covariates are predictive of the outcome vector $\mvmu $, \ie  $\|\mvmu  - \proj_{\colsp(X)}(\mvmu)\|^2$ is significantly smaller than $\| \mvmu \|^2$ then 
		$\hat{\tau}_{{\gsw}}$ has significantly smaller mean squared error than $\hat{\tau}_{{\rm Rad}}$.
		
		\item However, even if the covariates are not predictive of $\mvmu $, since $\|\mvmu  - 
		\proj_{\colsp(X)}(\mvmu)\|^2$ is bounded by $\|\mv\mu\|^2$ and the second term in the 
		right-hand side of~\eqref{eq:gs} is negligible compared to the first term, mean squared error of $\hat{\tau}_{{\gsw}}$ is never too much greater than the mean squared error of $\hat{\tau}_{{\rm 
				Rad}}.$
	\end{enumerate}
	
	The first property has been termed as \textit{Covariate Balance} and the second property  has been 
	termed as \textit{Robustness} in~\cite{harshaw2019balancing}. In this way, the Horvitz--Thompson 
	estimator using the Gram--Schmidt Walk (GSW) design achieves both covariate balance and robustness. It 
	has also been argued in~\cite{harshaw2019balancing} that the GSW design enjoys 
	certain advantages over other existing design approaches in the causal inference literature such as 
	rerandomization and designs based on matching pairs. Overall, it is probably fair to say that the GSW 
	design, although a recent entrant to the causal inference design toolbox, has already become one of 
	its prominent tools. In the next section, we review the GSW design algorithm; 
	see~\cite{bansal2018gram},~\cite{harshaw2019balancing} for more detailed discussions on the 
	algorithm. 
	%\todo{more explanations on importance of GSW design! Sabya da will think.}

	\subsection{Review of Gram--Schmidt walk algorithm}\label{sec:algo}
	We briefly describe the Gram--Schmidt random walk design in this section using the so-called {\em randomized pivot ordering}. See Section~3 and Section~A1.1 in \cite{harshaw2019balancing} for 
	further details. Define
	\begin{equation}\label{def:B_intro}
		\vB := \mat{\sqrt{\phi}I_n}{\xi^{-1}\sqrt{1 - \phi}X^{\t}}
	\end{equation}
	where $\xi \coloneqq \max_{i \in [n]}\|\mvx_i\|$. We start with $\cA_1^{{\gsw}} \coloneqq [n]$, $\mv z^{{\gsw}}_{0} = \mv 0 \in \R^n$. The algorithm at round $t \in [n]$. %maintains an active set $\A_t$
	\begin{enumerate}
		\item If $p^{{\gsw}}_{t - 1}  \notin \cA_t$, choose a pivot $p_t^{{\gsw}}$ uniformly at random from $\cA_t^{{\gsw}}$, otherwise set $p_t^{{\gsw}} = p_{t-1}^{{\gsw}}$.
		
		\item Compute a {\em step direction} $\mv u_t^{{\gsw}} \in \R^n$ as
		\begin{equation*}
			\begin{split}
				\mv u_t^{{\gsw}} \leftarrow  \argmin_{\mvu} \: \: \: \: \: \: \:  &\|\vB \mvu\|^2\\
				\mbox{subject to } \: &\mvu[i] = 0 \mbox{ for all }i \notin \cA_t^{{\gsw}}\\
				& \mvu[p_t^{{\gsw}}] = 1.
			\end{split}
		\end{equation*}
		%  	$\mv u_t[[n]\setminus\cA_t] = \mv0$, $\mv u_t[p(t)] = 1$ and $$B[\,\,  : \cA_t 
		%  	\setminus \{p_t\}]\mvu[\cA_t \setminus \{p_t\}] = - \proj_{\,\colsp(B_{\cA_t - 
				%  	\{p_t\}})}\, \left( \mvb_{p(t)}\right).$$

		\item Setting $$\Delta \coloneqq \{\delta \in \R: \mv z^{{\gsw}}_{t - 1} + \delta \mv u_t^{{\gsw}} \in [-1,1]^n\},$$
		let $\delta^{+} \coloneqq |\sup \Delta|$ and $\delta^{-} \coloneqq |\inf \Delta|$. Next define $\delta_t^{{\gsw}}$  as follows:
		\begin{equation*}
			\delta_t^{{\gsw}} =
			\begin{cases}
				\delta^{+} \:\:\text{with probability} \:\:\frac{\delta^{-}}{\delta^{+} + \delta^{-}} \\
				\delta^{-} \:\:\text{with probability} \:\:\frac{\delta^{+}}{\delta^{+} + \delta^{-}}
			\end{cases}
		\end{equation*}
		
		\item  	Now update 
		\begin{equation}\label{update:ztgs}
			\mv z_t^{{\gsw}} = \mv z_{t - 1}^{{\gsw}} + \delta_t^{{\gsw}} \mv u_t^{{\gsw}}
		\end{equation}
		and $\cA_{t + 1}^{{\gsw}} = \cA_{t}^{{\gsw}}\setminus \{i \in [n] : |\mv 
		z_t^{{\gsw}}[i]| < 1 \}$.
		
		\item Increment the index $t \leftarrow t + 1$ and go to 1, unless $t = n$.
	\end{enumerate}

	%{\color{red} Should we give a section where we give a sketch of the whole proof? What makes us able to handle the random ordering and the limiting variance identification?}

	\subsection{Central Limit Theorem}
	Apart from estimating the ATE, it is also of interest to give confidence intervals for the ATE. Towards 
	this end,~\cite{harshaw2019balancing} also studied the distributional properties of $\hat{\tau}_{\gsw} 
	- \tau.$ In particular, they showed that $\hat{\tau}_{\gsw} - \tau$ is subgaussian-$\sigma$ random 
	variable where $\sigma$ is given by the right-hand side in~\eqref{eq:gs}. As shown in~\cite{harshaw2019balancing}, 
	this fact allows one to construct confidence intervals valid for any sample size. Moreover, the authors also derived a central limit theorem (CLT) for $\hat{\tau}_{\gsw}$ and showed how this theorem 
	can be used (in the large sample case) to construct confidence intervals that are narrower than 
	those based on concentration inequalities for a subgaussian random variable. The CLT 
	in~\cite{harshaw2019balancing} was proved for the version of the Gram--Schmidt Walk design which 
	uses a {\em fixed} ordering of the pivots. It is based on somewhat strong assumptions on the matrix 
	$X$ along with the requirement that $n^2\var[\hat\tau_{\gsw}]$, \ie  the (scaled) variance of the 
	estimator {\em itself}, be at least $cn$ (as in the i.i.d. case) which is {\em unknown} a priori in relation 
	to the {\em parameters} $\mv\mu$ and $X$. The authors there conjectured, see 
	\cite[Section~6.3]{harshaw2019balancing}, that the CLT should also hold for the version which uses 
	{\em randomized pivot ordering}. The main focus of our article is to prove this conjecture under 
	minimal and global assumptions involving {\em only} the design matrix $X$ and the outcome vector 
	$\mv\mu$.
	
	%See Section~\ref{subsec:discuss} below for a comparative discussion of these two results.

	\subsection{The main results}
	%Our main result is a central limit theorem for the estimator $\hat\tau_{\gsw}$. 
	Recall that the covariate matrix $X$ and the outcome vector $\mv\mu$ both naturally depend on the 
	number of units $n$. Similarly the number of covariates $d$ and the robustness parameter $\phi$ may 
	also depend on $n$. In the sequel, we let
	\begin{equation}\label{def:v}
		\mvmu = X \mvgb_{\ls} + \mvv \mbox{ with } \mvv^\t X = \mv0,
	\end{equation}
	\ie\ $\mvv$ is the orthogonal projection of $\mv\mu$ onto ${\rm 
		ColSp}(X)^{\perp}$. We are now ready to state the three %``simplified'' 
	regularity conditions %(cf.~condition~\ref{eq:formal_cond} in Theorem~\ref{thm:formal} below)
	which we will use to show the asymptotic normality. These are ``simplified'' versions of much more 
	{\em general} assumptions under which we can derive the CLT in Theorem~\ref{thm:formal} below.
	\begin{assumption}[Outcome regularity]\label{assumption:outcome-regularity}
		%$\frac{\|\mv v\|_{\infty}}{\|\mv v\|} \le \frac{C}{\log^{2 + \Cl[c]{c:out_regular}}(n)}$ 
		$\frac{\|\mv v\|_{\infty}}{\|\mv v\|} \le \frac{C}{\log^{3 + c}(n)}$ whereas $\|\mvgb_{{\rm ls}}\|^2 \le C \log n$ for some positive constants $c$ and $C$. %is bounded in 	$n$.
	\end{assumption}
	\begin{assumption}[Covariate regularity]\label{assumption:covariate-regularity}
		%$\sigma_{{\rm min}}(X) \ge \sqrt{\Cl[c]{c:covar_regular}n}$ 
		The smallest singular value of the covariate matrix $X$ satisfies 
		$\sigma_{{\rm min}}(X) \ge \sqrt{cn}$ for some positive constant $c$. Also the maximum row norm of the covariate matrix, $\xi = \max_{i \in [n]}\|\mvx_i\|$, 
		is bounded as $\xi^2 \le C d \log n$ for some positive contant $C$.
	\end{assumption}
	\begin{assumption}[Non-degeneracy]\label{assumption:nondegeneracy} %$\|\mvv\| \to \infty$ as $n \to \infty$.
		$\|\mvv\|^2 \ge \log^{2 + c}(n)$ for some positive constant $c$.
	\end{assumption}
	%{\color{red} why are these conditions natural?}

	A detailed discussion of the necessity and minimality of these assumptions %as well as their 
	in comparison to their counterparts in \cite{harshaw2019balancing} is provided in 
	the next subsection.
	\begin{theorem}\label{thm:main}
		Suppose that Assumptions~\ref{assumption:outcome-regularity}--\ref{assumption:nondegeneracy} hold. Also, suppose that $d$ is fixed and $\phi$ is bounded away from $0$ and $1$. Then the limiting distribution of the Horvitz--Thompson
		estimator $\hat\tau_{\gsw}$ under the Gram--Schmidt Walk design is normal:
		\begin{equation}
			\frac{\hat\tau_{\gsw} - \tau}{\sqrt{\var[\hat\tau_{\gsw}]}}  \xrightarrow[n 
			\to \infty]{{\rm law}} N(0, 1).
		\end{equation}
		Furthermore, in this case, we have the following asymptotic formula for the 
		variance:
		\begin{equation}\label{eq:variance_formula}
			\lim_{n \to \infty}\frac{\var[\hat\tau_{\gsw}]}{\|\mvv\|^2/n^2} = 1    
		\end{equation}
		where $\mvv = \proj_{\colsp(X)^{\perp}}(\mv\mu)$ is the orthogonal 
		projection of the outcome vector $\mv\mu$ onto $\colsp(X)^{\perp}$.
	\end{theorem}

	\begin{remark}\label{rem:mainres}
		In Theorem~\ref{thm:main}, we assume conditions~\ref{assumption:outcome-regularity}--\ref{assumption:nondegeneracy} to keep the exposition simpler. One can prove a CLT under other assumptions as long as the error term converges to zero. Also, one can apply a similar idea for the case when $\P[z_i = 1],\P[z_i = - 1] \in (\eps,1-\eps)$ for some $\eps\in(0,1)$ fixed.
	\end{remark}
	
	\subsection{Main contributions and comparisons with literature} \label{subsec:discuss}
	In this Section, we briefly discuss the major contributions of the article.
	\begin{enumerate}
		\item \textbf{Generality of assumptions on $\mv\mu$ and $X$.} We compare our assumptions vis-\`{a}-vis the assumptions made 
		in~\cite[Section~6.3]{harshaw2019balancing} facilitating the CLT.
		\begin{enumerate}
			\item {\em Outcome Regularity.} The analogous assumption in \cite{harshaw2019balancing}, namely 
			Assumption~6.4, is expressed in terms of the $\ell^\infty$-norm of the outcome vector $\mv\mu$:
			\[\|\mv\mu\|_\infty = O(\log^c(n)) .\]
			If we assume that $\|\mv\mu\|^2 \ge cn$ in the ``typical'' scenario, we can reinterpret this condition as
			\begin{equation}\label{bnd:mu}
				\frac{\|\mv \mu\|_{\infty}}{\|\mv\mu\|} = O\left(\frac{\log^{c}(n)}{\sqrt{n}}\right).	
			\end{equation}
			On the other hand, we formulate our condition in terms of the ``residual'' outcome vector $\mvv = {\rm 
				Proj}_{\colsp(X)^{\perp}}(\mv\mu)$ which turns out to be the right object to look at based on our 
			analysis. Note however that while the covariate matrix $X$ is known to the experimenter prior to 
			the treatment assignment, $\mv\mu$ and hence $\mvv$ is unavailable to her. Compared to 
			\eqref{bnd:mu}, our assumption on $\mvv$ as a vector satisfying 
			\[\frac{\|\mv v\|_{\infty}}{\|\mv v\|} = O\left(\frac{1}{\log^{3 + c}(n)}\right)\] is significantly weaker. This 
			outcome regularity assumption essentially puts a density constraint on $\mvv$; \ie  $\mvv$ cannot be too sparse. The assumption in \cite{harshaw2019balancing} interpreted as 
			\eqref{bnd:mu} above basically says that $\mv\mu$ should have {\em effectively} 
			$\Omega\left(\tfrac{n}{\log^c(n)}\right)$ many non-zero entries of roughly equal size. In comparison, 
			our assumption allows $\mvv$ which is very sparse, in the sense that it can be supported on 
			$\log^c(n)$ many elements of roughly equal size.  It is not difficult to see that the effective support of 
			$\mv\mu$ or $\mvv$ should diverge to $\infty$ at some rate for any CLT to hold. Our assumption on 
			$\mv\beta_{\ls}$ is %new but 
			very mild and can in fact be easily deduced from conditions like $\|\mv\mu\|_{\infty}^2 = O(\log n)$ and 
			Assumption~\ref{assumption:covariate-regularity} on $X$.

			\item  {\em Covariate Regularity.} Assumption~6.5 in \cite{harshaw2019balancing} stipulates that 
			$\sigma_{\min}(X_m) \ge \sqrt{cm}$ for {all} $m \ge n^{1/2 - \epsilon}$ where $X_m$ is the 
			submatrix of dimensions $m \times d $ given by the first $m$ rows of $X$ after ordering them 
			according to the ordering of pivots. In our understanding, this assumption is necessary essentially because they work with a fixed pivot ordering. Since we work with the {\em randomized pivot 
				ordering}, this randomization actually allows us to only require a {\em global} condition on the full 
			covariate matrix instead of its many submatrices. We believe that %the assumptions required on the covariate matrix for 
			our Assumption~\ref{assumption:covariate-regularity} on the covariate matrix $X$ is both general and 
			minimal. Assumption~6.5 in \cite{harshaw2019balancing} also requires a logarithmic upper bound on 
			the so-called {\em incoherence} of the submatrix $X_m$ for all $m \ge n^{1/2 - \ep}$ which we {\em 
				do not need at all}. Finally, our assumption that the maximum squared $\ell^2$-norm of $\mvx_i$'s 
			($\in \R^d$) is at most $Cd\log n$ is also present in \cite[Assumption~6.5]{harshaw2019balancing} 
			and is motivated by similar reasons.

			\item {\em Non-degeneracy.} The non-degeneracy criterion presented 
			by~\cite[Assumption~6.6]{harshaw2019balancing} is the blanket assumption that $n^2\var[\hat{\tau}_{\gsw}]$ is at least $cn$ like in the case of i.i.d. design. In particular, it is not clear what this assumption means in terms of the problem parameters $\mv \mu$ and $X$. In comparison, 
			our analysis yields that $n^2\var[\hat{\tau}_{\gsw}]$ is {\em asymptotically equivalent} to $\|\mvv\|^2$ 
			\eqref{eq:variance_formula}. Consequently, our Assumption~\ref{assumption:nondegeneracy} 
			suggests a non-degeneracy condition involving solely the super-logarithmic divergence of 
			$\|\mvv\|^2$ which is a known function of the problem parameters $\mv\mu$ and $X$. Also since 
			$n^2\var[\hat{\tau}_{\gsw}]$ is asymptotically equivalent to $\|\mvv\|^2$, our condition is {\em significantly weaker} compared to \cite[Assumption~6.6]{harshaw2019balancing} as it replaces the 
			linear lower bound in the latter with only a poly-logarithmic bound.
		\end{enumerate}
		
		\smallskip
		
		\item \textbf{A precise and smaller asymptotic variance.}
		%As a , 
		A very important byproduct of our analysis is that we obtain an asymptotically {\em exact} variance control for the estimator $\hat\tau_{\gsw}$ as given by \eqref{eq:variance_formula}. 
		%Horvitz--Thompson estimator based on the GSW design. 
		It is not difficult to see that (see for instance the brief discussion following Theorem~\ref{thm:msebnd})
		the best possible upper bound on $n^2\var[\hat\tau_{\gsw}]$ allowed by Theorem~\ref{thm:msebnd} %(which is actually a consequence of Theorem~$3.5$ in \cite{harshaw2019balancing}) 
		is $\frac1{\phi}\|\mvv\|^2$. Since $\phi \in (0, 1)$, our formula \eqref{eq:variance_formula} shows that 
		the {\em true} variance is {\em smaller} by the factor $\phi$ and as such is {\em independent} of the 
		robustness parameter $\phi$. This is particularly beneficial for statisical applications. %We would like to mention here that a similar {\em upper bound} can be easily deduced from Theorem~\ref{thm:msebnd} provided we allow $\phi$ to be ``very close'' to 1 {\em depending} on the unknown problem parameters $\mvbeta_{\ls}$ and $\mvv$. This was already noted in \cite[Corollary~4.4]{harshaw2019balancing}.
		%Note that the asymptotic variance is smaller than ($\phi$ times) the upper bound on the variance 
		%obtained in~\cite{harshaw2019balancing} under our assumption on $\mvv$. It is easy to check that the 
		%RHS of~\eqref{eq:gs} is given by 
		%		\[
		%		\frac1{n^2}\mvmu^{\t}(\vB^{\t}\vB)^{-1}\mvmu = \frac{1}{\phi n^2}\norm{\mvv}^2 + \frac{\xi^2}{(1-\phi)n^2}\mvgb^{\t}_{\ls}D^{-1}\mvgb_{\ls}
		%		\]
		%		where $D$ is a diagonal matrix with entries $\left(1+\frac{\phi \xi^2}{(1-\phi)\gl_i}\right)_{i=1}^d$ with $(\gl_i)_{i=1}^d$'s being eigenvalues of $X^{\t} X$. Our analysis shows that (see~\eqref{formula:var}) this bound is not optimal.
		%		{\color{red}expand}
		
		\smallskip
		
		\item \textbf{Power of randomization.} As already mentioned, we prove our CLT for the GSW design 
		with randomized pivot ordering as opposed to the fixed pivot ordering in \cite{harshaw2019balancing} 
		where it was posed as a conjecture. This randomization plays a crucial role in our analysis and 
		enables us to derive our result under such minimal conditions as discussed above. We harness its 
		strength in this paper in three different (but related) ways. Firstly, randomization lies at the heart of our 
		definition of the {\em skeletal} process which carries most of our analysis and yields the correct 
		formula for the asymptotic variance in \eqref{eq:variance_formula} (see Section~\ref{sec:martingales} 
		below). Secondly, we use it to take advantage of the concentration phenomenon for the sum of a random 
		sample without replacement of size $k$ out of $n$ unit-rank matrices when $1\ll k\ll n$. We use the 
		Stein's method for exchangeable pairs for this purpose (see Proposition~\ref{prop:lambdaminconc}). 
		Finally, we use some classical moments formulae for the sample mean of a random sample without 
		replacement (see Proposition~\ref{thm:srswor_moment}) propelled by the need to exploit the 
		orthogonality between $\mvv$ and $\colsp(X)$. All of these point towards the fact that the 
		randomized pivot ordering is perhaps the ``natural'' setup to perform the GSW design.

		%\smallskip

		%\item \textbf{Non-uniform treatment probabilities.} Our approach to find and analyze the skeletal 
		%process is general. This can be used to analyze the Horvitz--Thompson estimator based on the GSW 
		%design in non-uniform probability assignment cases (see 
		%also~\cite[Section~A3.1]{harshaw2019balancing}).
	\end{enumerate}
	
	\subsection{Proof sketch}\label{subsec:sketch}
	We deduce Theorem~\ref{thm:main} as a special instance of a more general result. Below we let $Y$ 
	denote the matrix $\xi^{-1}\textstyle\frac{\sqrt{1 - \phi}}{\sqrt{\phi}}\, X$ (the motivation for choosing 
	this matrix is given below \eqref{def:B} in the next section).%We will make a purely asymptotic statement. %Problem Parameters: $n$, $d$, $\|v\|_{\infty}$, $\kappa$
	\begin{theorem}\label{thm:formal}
		% 	Recall 
		% 	\begin{equation*}
			% 		M_n^{{\gsw}} = \sum_{s \in [n]} \delta^{{\gsw}}_s \left \langle B \mv u^{{\gsw}}_s, \mat{\mv v}{0} \right\rangle.
			% 	\end{equation*}
		% 	
		% 	where {\color{red} define B and Y....make the assumption v is a unit vector and $v$ is orthogonal to the columns of $Y$...say in the intro why wlg we can make this assumption...leaving this for Subhajit..}
		% 	
		%We will consider a sequence of problems indexed by $n.$ 
		%Our result will be stated in terms of the problem parameters $d$, $\|v\|_{\infty}$, $\kappa$ where 
		Let $\phi \in (0, 1)$ be bounded away from $0$ and $d \le n$. Also let $$\kappa \coloneqq \tfrac{n}{\sigma_{\min}(Y)} = 
		\tfrac{n}{\lambda_{\min}(Y^{\t}Y)}$$ 
		($\lambda_{\min}(\cdot)$ being the smallest eigenvalue) and %$\mvv = \proj_{\colsp(X)^{\perp}}(\mv\mu)$ 
		$\mvv \in \R^n$ be any vector satisfying $\mvv^{\t}X = 0$, \ie  $\mvv$ is orthogonal to $\colsp(X)$. %be as in~\eqref{def:v}. 
		Then we have, with $\mvz_n^{\gsw}$ as given by \eqref{update:ztgs},
		\begin{equation}\label{eq:formal_clt}
			\frac{\langle \mvz_n^{{\gsw}}, \mvv \rangle}{\sqrt{\var\langle \mvz_n^{{\gsw}}, \mvv \rangle}}  \xrightarrow[n \to \infty]{{\rm law}} N(0, 1) 
		\end{equation}
		\text{ provided }
		%\begin{equation}\label{eq:formal_cond}\lim_{n \to \infty} d^{1/2} \, \frac{\|\mv v\|_{\infty}}{\|\mv v\|}\,  \kappa^2 \log n = 0 \, \mbox{ and }\lim_{n \to \infty} \frac{\|\mvgb\|}{\|\mvv\|} = 0.\end{equation}
		\begin{equation}\label{eq:formal_cond}
			\lim_{n \to \infty} d^{1/2} \, \frac{\|\mv v\|_{\infty}}{\|\mv v\|}\,  \kappa^2 \log n = 0. %\, \mbox{ and }\lim_{n \to \infty} \frac{\|\mvgb\|}{\|\mvv\|} = 0.
		\end{equation}
		Furthermore, in this case, we have the following asymptotic formula for the variance:
		\begin{equation}\label{formula:var}
			\lim_{n \to \infty} \frac{\var\langle \mvz_n^{{\gsw}}, \mvv \rangle}{\|\mvv\|^2} = 1.
		\end{equation}
		% 	
		% 	
		% 	Technically, we can think of these problem parameters as sequences indexed by $n$ as well although we do not write this explicitly to avoid notational clutter. Suppose the problem parameters satisfy the following limiting statements as $n \rightarrow \infty$,
		% 	
		% 	
		% 	\begin{itemize}
			% 		\item 
			% 		
			% 		\begin{equation}\label{eq:cond1}
				% 			\frac{\kappa^2 \log nd}{n} \rightarrow 0.
				% 		\end{equation}
			% 		
			% 		\item 
			% 		
			% 		\begin{equation}\label{eq:cond}
				% 			d \|v\|_{\infty}^{4} \kappa^3 (\log en)^{3} \rightarrow 0.
				% 		\end{equation}
			% 		
			% 		\item 
			% 		\begin{equation}\label{eq:cond2}
				% 			\|v\|_{\infty}^{4} \kappa^4 \rightarrow 0.
				% 		\end{equation}
			% 		
			% 		
			% 		
			% 		\item 
			% 		\begin{equation}\label{eq:cond5}
				% 			\|v\|_{\infty}^{2} (\kappa^2 + d) \log end \rightarrow 0.
				% 		\end{equation}
			% 		
			% 		\item 
			% 		\begin{equation}\label{eq:cond6}
				% 			\|v\|_{\infty} (\log en) \rightarrow 0.
				% 		\end{equation}
			% 		
			% 		\item 
			% 		\begin{equation}\label{eq:cond7}
				% 			d \|v\|_{\infty}^{2} \kappa^6 (\log end)^{3} \rightarrow 0.
				% 		\end{equation}
			% 		
			% 		
			% 		
			% 		Then 
			% 		
			% 		\begin{equation*}
				% 			M_n^{{\gsw}} \stackrel{{\rm d}}{\rightarrow} \N(0, 1).
				% 	\end{equation*}	\end{itemize}
	\end{theorem}

	\begin{proof}[Proof of Theorem~\ref{thm:main}]
		We assume Theorem~\ref{thm:formal} and proceed to deduce Theorem~\ref{thm:main} from it. Recall that 
		\[Y = \xi^{-1}\textstyle\frac{\sqrt{1 - \phi}}{\sqrt{\phi}}\, X.\]
		Since $d$ is bounded and $\phi$ is bounded away from $0$ and $1$, 
		Assumptions~\ref{assumption:outcome-regularity}-\ref{assumption:covariate-regularity} %\ref{assumption:nondegeneracy}
		imply condition \eqref{eq:formal_cond} for $\mvv = \proj_{{\rm 
				ColSp}(X)^{\perp}}(\mv\mu)$ with $\kappa \le C  \log (n)$
		and hence 
		\eqref{eq:formal_clt} and \eqref{formula:var} hold for this choice of $\mvv$ 
		by Theorem~\ref{thm:formal}. On the other hand, from 
		\cite[Lemma~1.1]{harshaw2019balancing} one can write
		\begin{equation}\label{eq:inner_prod}
			\hat\tau_{\gsw} - \tau = \frac1n\langle\mvz_n^{\gsw}, \mv\mu\rangle = \frac1n\langle\mvz_n^{\gsw}, X\mv\beta_{\ls}\rangle + \frac1n\langle\mvz_n^{\gsw}, \mvv\rangle.  
		\end{equation}
		However, in view of Theorem~\ref{thm:msebnd} applied to the case when $\mvmu = X\mv\beta_{\ls}$, we can bound
		\begin{equation*}
			\var\langle\mvz_n^{\gsw}, X\mv\beta_{\ls}\rangle \le 
			\frac{\xi^2\|\mv\beta_{\ls}\|^2}{(1 - \phi)}.
		\end{equation*}
		But since $\|\mv\beta_{\ls}\|^2 \le C\log n$ is bounded and $\xi^2 \le C \log n$ by 
		Assumptions~\ref{assumption:outcome-regularity} and 
		\ref{assumption:covariate-regularity} respectively (recall that $d$ is bounded), we immediately 
		obtain
		\begin{equation*}
			\var\langle\mvz_n^{\gsw}, X\mv\beta_{\ls}\rangle \le \frac{C \log^2(n)}{1 - \phi}
		\end{equation*}
		for some constant $C > 0$. Therefore, from \eqref{formula:var} and 
		Assumption~\ref{assumption:nondegeneracy}, we immediately deduce~\eqref{eq:variance_formula} as well as
		\begin{equation*}
			\frac{\langle\mvz_n^{\gsw}, X\mv\beta_{\ls}\rangle}{\sqrt{\var\langle \mvz_n^{{\gsw}}, \mvv \rangle}} \mbox{ converges to $0$ in probability.}
		\end{equation*}
		Together with \eqref{eq:formal_clt}, this yields the CLT for $\hat\tau_{\gsw}$ in view of 
		\eqref{eq:inner_prod}.
	\end{proof}	
	\begin{remark}[Relaxing the assumptions for Theorem~\ref{thm:main}]
		\label{rmk:generalcond}
		It is clear from the statement of Theorem~\ref{thm:formal} as well as the proof above that there are 
		flexibilities for relaxing or modifying the 
		Assumptions~\ref{assumption:outcome-regularity}--\ref{assumption:nondegeneracy} as might be 
		suitable in some situations.
	\end{remark}	
	The rest of the article is devoted to proving Theorem~\ref{thm:formal}. It is not difficult to see that 
	$\{\langle\mvz_t^{\gsw}, \mvv\rangle\}_{t = 0}^n$ is a martingale sequence (see 
	Section~\ref{sec:martingales} below) and we will be using the Martingale central limit theorem to prove 
	Theorem~\ref{thm:formal}. However, running this recipe for the process $\{\langle\mvz_t^{\gsw}, 
	\mvv\rangle\}_{t = 0}^n$ is far from being straightforward. The GSW design introduces smeared but 
	irregular dependence across all $t \in [n]$ which makes it very hard to control the martingale 
	differences. Herein enters the {\em skeletal} process $\{M_t\}_{t = 0}^n$ introduced in 
	Section~\ref{sec:martingales} renders this dependence ``well-behaved'' and the explicit 
	computation of variance almost immediate. Roughly speaking, $M_n$ is the projection along $\mvv$ 
	of an i.i.d. Rademacher linear combination of vectors obtained by applying the classical Gram-Schmidt 
	process to a uniformly random permutation of the columns of a matrix $B$. In Section~\ref{sec:ideal} 
	we prove the CLT for $M_n$ and compute its variance in the process. This part involves, among other 
	important ideas, careful application of the concentration for the sum of unit-rank matrices in a random 
	sample using Stein's method for {\em exchangeable pairs}.
	
	\smallskip
	
	However, one still needs to transfer this result to our original random variable $\langle \mvz_n^{\gsw}, 
	\mvv\rangle$. To this end, we first identify a time point $t = n - m$ until which the two underlying 
	processes stay very close to each other on some good event $\mathcal G$ occurring with high 
	probability. Then we control the tail ends of these two processes, \ie  the difference between time $n$ 
	and $n-m$. To facilitate these steps, we introduce yet another intermediate process, namely 
	$\{\tilde{M}_t\}_{t = 0}^n$ in Section~\ref{sec:martingales}. This part of our analysis is the content of Section~\ref{sec:reduction_ideal} and hinges on some of the ideas used in Section~\ref{sec:ideal}.

	\subsection{Notations and convention for constants}
	We use boldface to denote vectors. For any natural number $n$, $[n]$ denotes the set of integers 
	$\{1, \ldots, n\}$. For any $m \times n$ matrix $D$ and subsets $\cA_1 \subset [m]$ and $\cA_2 
	\subset [n]$, we use $D[\cA_1, \cA_2]$ to denote the submatrix of $D$ formed by the 
	rows and columns with indices in $\cA_1$ and $\cA_2$ respectively. If $\cA_1 = [m]$ or $\cA_2 = [n]$, 
	we denote the corresponding submatrix as $D[:,\, \cA_2]$ and $D[\cA_1, \, :]$ 
	respectively. Similar notations are also used for vectors. $D^\t$ denotes the 
	transpose of the matrix $D$. $\colsp(D)$ denotes the column space of $D$, and $\Tr(D)$ denotes its trace, \ie  the sum of 
	diagonal elements of $D$. The {\em Frobenius norm} of $D$, \ie $\Tr(D^\t D) = \Tr(DD^\t)$ ($D$ is always real in our case) is 
	denoted by $\|D\|_{\frob}$ whereas its {\em operator norm}, 
	\ie  the maximum singular value of $D$ is denoted as $\|D\|_{{\rm op}}$. For any linear subspace $S \subset \R^n$, $\proj_S$ denotes the orthogonal projector onto $S$. We use $I_n$ to denote the $n \times n$ identity matrix. 
	
	\medskip
	
	Our convention regarding constants is the following. Throughout, $c, c', C, C', \ldots$ denote positive 
	constants that may change from place to place. It might be helpful to think of upper and lower-case 
	letters as denoting large and small constants respectively. Numbered constants are defined the first 
	time they appear and remain fixed thereafter. All constants are assumed to be absolute unless 
	explicitly mentioned otherwise. To avoid the cluttering of notations, we suppress the 
	implicit dependence on $n$ in the problem parameters $\mvmu, X$, etc., as well as various stochastic processes that we define in the course of our analysis.
	\medskip

	\noindent\textbf{Note added later:} After the article appeared online, by private communication, we came to know that the authors in~\cite{harshaw2019balancing} are in the process of extending their CLT in a similar direction, albeit under possibly different assumptions.

	\section{Definition of the skeletal process and the relevant martingales}\label{sec:stoch_proc}
	Recall from the introduction that the Gram-Schmidt Walk algorithm furnishes several stochastic
	processes defined on the same probability space $(\Omega, \cA, \P)$ (say); these are the
	fractional assignments $\{\mv z_t^{{\gsw}}\}_{t = 0}^{n}$, the sequence of pivots $\{p_t^{{\rm
			gs}}\}_{t = 0}^{n}$, the associated sequence of active sets $\{\cA^{{\gsw}}_t\}_{t = 0}^{n}$,
	the sequence of direction vectors $\{\mv u_t^{{\gsw}}\}_{t = 0}^{n} \in \R^n$ and finally the sequence
	of step sizes $\{ \delta_t^{{\gsw}}\}_{t = 0}^{n}$.  %Henceforth, we will use the superscript `${\gsw}$' to denote the processes associated with  the original Gram-Schmidt algorithm. This is because we want to preserve the lighter, unsuperscripted notations for a {\em parallel} process also defined on $(\Omega, \cA, \P)$. %Apart from these stochastic processes, we would need to define 
	%a number of other analogous stochastic processes, namely, $\{\mv z_s\}_{s = 0}^{n} \in \R^n$, 
	%$\{p_s\}_{s = 0}^{n} \in [n]$, $\{\cA_s\}_{s = 0}^{n}$, $\{\mv u_s\}_{s = 0}^{n} \in \R^n$ 
	%and $\{ \delta_s\}_{s = 0}^{n}$ on the {\em same} probability space. 
	In this section we will define, on the same underlying probability space $(\Omega, \cA, \P)$, a
	{\em parallel} process for each of these processes. We will refer to these processes collectively and --- with a slight abuse of notation --- individually as the {\em skeletal (Gram-Schmidt) process(es)}. In
	fact, the only CLT we will prove {\em directly} in
	this paper is for the random variables $\langle \mvz_n, \mvv\rangle$ where $\{\mvz_t\}_{t = 0}^n$ is
	the skeletal process parallel to $\{\mvz_t^{{\gsw}}\}_{t = 0}^n$. This is also the reason why we chose
	relatively heavier notations for the processes associated with the Gram-Schmidt algorithm in the
	introduction so that we can preserve the lighter notations for the skeletal processes.

	%We will now define these coupled processes recursively along with their superscripted relatives. We 
	%would need access to a sequence $U_1,\dots,U_n$ of i.i.d.~${\rm Unif}(0,1)$ random variables.

	%Along the way, we will also recursively define its 
	%associated random sequence of subsets 
	%$\cA_1,\dots,\cA_n \subset [n].$ We initialize with $\mv z^{{\gsw}}_0 = \mv z_0 = 0$ 
	%and $\cA^{{\gsw}}_1 = \cA_1 = [n].$ 

	To this end let us start with $\mv z_0^{{\gsw}} = \mv z_0 = \mv0$, $\cA_1^{{\gsw}} = \cA_1 = [n]$
	and a sequence $U_1,\dots,U_n$ of i.i.d.~${\rm Unif}(0,1)$ random variables. In the sequel, we assume
	that {\em all} the random variables are defined on a {\em common} probability space $(\Omega, \mathcal
	A, \mathbb P)$. At each round $t \in [n]$, we will create a number of ``allied'' pair of random variables,
	namely the partial assignments $(\mv z^{{\gsw}}_t, \mv z_t)$ (vectors in $\R^n$), pivots $(p_t^{{\gsw}}, p_t)$ (elements of $[n]$), the active sets $(\cA_t^{{\gsw}}, \cA_t)$ (subsets of
	$[n]$), the directions $(\mvu_t^{{\gsw}}, \mvu_t)$ (vectors in $\R^n$) and the step sizes
	$(\delta_t^{{\gsw}}, \delta_t)$ (takes values in $[-2, 2]$). We slightly modify the matrix $B$ from the introduction as follows (we use $\vB$ to denote the original definition from~\cite{harshaw2019balancing})
	\begin{equation}\label{def:B}
		B = \mat{I_n}{Y^{\t}}
	\end{equation}
	where $Y \coloneqq \xi^{-1}\textstyle\frac{\sqrt{1 - \phi}}{\sqrt{\phi}}\, X$ is an $n\times d$ matrix with
	rank $d$. Note that this matches the definition in~\eqref{def:B_intro} upto a rescaling by
	$\frac1{\sqrt{\phi}}$ which does not affect the vectors $\mvz_t^{{\gsw}}$ as they depend only on
	$\colsp(B)$ (revisit the algorithm in the introduction).  Also since $\xi = \max_{i \in [n]}\|\mv x_i\|$,
	it follows that
	\begin{equation}\label{def:zeta}
		\max_{i \in [n]}\|\mv y_i\| \le \textstyle\frac{\sqrt{1 - \phi}}{\sqrt{\phi}} \eqqcolon \zeta
	\end{equation}
	Finally notice that
	\begin{equation}\label{eq:vY}
		\mvv^{\t}Y = 0
	\end{equation}
	as $\mvv^\t X = 0$ in view of~\eqref{def:v}.
	
	Now suppose for some $t \in [n]$, we have already defined these processes for all $0 \le s < t$
	(whenever appropriate) and that they are measurable relative to $\cF_{t - 1} \subset \cA$ where $\cF_0$ is the trivial $\sigma$-algebra. In order to define the processes at time
	$t$, it will be helpful to introduce some new objects and notations.
	
	For any non-empty $\cA \subset [n]$ and index $p \in \cA$, we define $\mv u(p,
	\cA)  \in \R^n$ to be a vector $\mvu$ satisfying
	\begin{equation}\label{def:u_ortho_project}
		\begin{gathered}
			\mvu\left[ [n]\setminus\cA\right] = \mv0, \mvu[p] = 1 \mbox{   and}\\
			B[:, \, \cA \setminus \{p\}]\,\mvu[\cA \setminus \{p\}] = -
			\proj_{\,\colsp(B[: \, \cA - \{p\}])}\, \left( \mvb_{p}\right)
		\end{gathered}
	\end{equation}
	where $\mvb_p \coloneqq B[:, p]$ and $\proj_S(\mvv)$ denotes the orthogonal projection of
	$\mvv \in \R^n$ onto the subspace $S$. Since the matrix $B$ has full column rank, it follows that the
	vector $\mvu(p, \cA)$ is in fact unique. Notice that $\mvu_t = \mv u(p_t, \cA_t)$ in step~2 of our original Gram-Schmidt algorithm.
	
	Next we introduce two matrices which we would need to state our augmented algorithm: 
	\begin{equation}\label{def:XX_tBC}
		\begin{gathered}%used in place of ``split'' environment for center alignment
			Y^{{\gsw}}_t \coloneqq Y[\cA^{{\gsw}}_t, \, :] \mbox{ and } Y_t \coloneqq Y[\cA_t,:] \mbox{ for } t \in [n]. %\: Y_t^- = Y[\cA_{t + 1},:], \qquad
			%B_t = B[:, \cA_{t}^-] =  \mat{I_n[:, \cA_{t}^-]}{{Y_t^-}^{\t}} \\
			%C^{\gsw}_t \coloneqq (I_{d}+(Y^{{\gsw}}_t)^{\t} Y^{\gsw}_t)^{-1} \mbox{ and }
			%C_t^- = (I_{d} + {Y_t^-}^{\t} Y_t^-)^{-1}
			%C_t \coloneqq (I_{d} + Y_t^{\t} Y_t)^{-1}
		\end{gathered}
	\end{equation} 
	
	Notice that both $\cA^{{\gsw}}_t$ and $\cA_t$ are measurable relative to $\cF_{t-1}$ and therefore so are $Y^{{\gsw}}_t$ and $Y_t$. 
	Let $\{\varepsilon_n\}_{n = 1}^{\infty}$ be a given decreasing sequence of numbers going to $0$ with
	$\varepsilon_1 \leq 1/2$ (any number strictly less than $1$ would do). %({\color{red} which we will choose later, half is arbitrary}).
	Using these we now define two sequences of events which will be then used below to demacate the
	time until which the {\em supescripted} and {\em unsuperscripted} processes coincide. For any $t \in [n]$, let
	\begin{equation}\label{def:G1sG2s}
		\begin{split}
			\mathcal{G}_{1, t} &\coloneqq \{\|\mv z_{t - 1}[\cA_{t}]\|_{\infty} < \varepsilon_n\} \mbox{ and }\\
			\mathcal{G}_{2,t} &\coloneqq \left\{\big\|Y_t^{\t} Y_t - \frac{n - t + 1}{n} Y^{\t} Y\big\|_{{\rm op}} \leq  \frac{n -
				t + 1}{2n} \lambda_{\min}\big(Y^{\t} Y\big)\right\}.
		\end{split}
	\end{equation}
	
	Note that since $\cA_t$ is $\cF_{t - 1}$ measurable, $\mathcal{G}_{1,t}$ and $\mathcal{G}_{2,t}$ are as well. 
	Now we are ready to extend our processes into round $t$. Below and in the remainder of the article,
	we let
	\begin{equation}\label{def:Czeta}
		\Cl{C:xi}(\zeta) \coloneqq 1 + \zeta^2(1 + \zeta^2)
	\end{equation}
	whose significance will be clear later (see, e.g.,~\eqref{eq:identity2} in
	Lemma~\ref{lem:matrix_identities} and Lemma~\ref{lem:pivotlemma}, both in Section~\ref{sec:prelim}).
	
	\textbf{\labelText{Case~1}{label:case1}: If $t \le n - 6 \Cr{C:xi} \zeta^2 \kappa$ and for all $s \in [t]$, both $\mathcal{G}_{1,s}$ and $\mathcal{G}_{2,s}$ occur and $\cA_{s}^{{\gsw}} =
		\cA_s$:}
	\begin{itemize}
		\item Choose a pivot $p_t$ uniformly at random from $\cA^{{\gsw}}_t$ and set $p_t^{{\gsw}} = p_t$.
		
		\item Define
		$$\mv u_t^{{\gsw}} = \mv u_t = \mv u(p_t^{{\gsw}},\cA_{t-1}^{{\gsw}})$$
		and set
		\begin{equation}\label{def:Delta}
			\Delta_t = \{\delta \in \R: \mv z^{{\gsw}}_{t - 1} + \delta \mv u_t^{{\gsw}}  \in [-1,1]^n\}.
		\end{equation}
		Now letting $\delta^{+}_t = \sup \Delta_t$ and $\delta^{-}_t = |\inf \Delta_t|$, define $\delta^{{\rm
				gs}}_t = \delta_t$ as follows:
		\begin{equation}\label{def:delta_t}
			\delta^{{\gsw}}_t = \delta_t = \delta^{+}_t \, \mb 1\left\{U_t \leq \textstyle\frac{\delta^{-}_t}{\delta^{+}_t
				+ \delta^{-}_t}\right\} -  \delta^{-}_t \, \mb1\left\{U_t > \textstyle\frac{\delta^{+}_t}{\delta^{+}_t + \delta^{-}_t}\right\}.
		\end{equation}
		
		\item Update
		\begin{equation}\label{eq:ztgsupdate}
			\mv z^{{\gsw}}_t = \mv z_t = \mv z^{{\gsw}}_{t - 1} + \delta^{{\gsw}}_t \mv u^{{\gsw}}_t = \mv
			z_{t - 1} + \delta_t \mv u_t.
		\end{equation}
		
		\item Define
		\begin{equation}\label{def:etat}
			\eta_t =  \mb 1 \left\{U_t \leq 1/2\right\} -  \mb 1 \left\{U_t > 1/2\right\}.
		\end{equation}
		
		\item Update
		\begin{equation}\label{eq:Atgssupdate}
			\cA_{t + 1}^{{\gsw}} = \{i \in \cA^{{\gsw}}_{t}: |\mv z^{{\gsw}}_t[i]| <
			1\} \mbox{ and }	\cA_{t + 1} = \cA_{t} - \{p_t\}.
		\end{equation}
	\end{itemize}
	
	\textbf{\labelText{Case~2}{label:case2}: Otherwise}:
	\begin{itemize}
		\item If $\cA_t^{{\gsw}} = \emptyset$, we set $\mv z^{{\gsw}}_t = \mv z_t, \cA^{{\rm
				gs}}_{t + 1} = \emptyset, \mv u^{{\gsw}}_t = \mv 0$ and $\delta^{{\gsw}}_t = 0$. Otherwise if $p_{t -
			1}^{{\gsw}} \in \cA^{{\gsw}}_{t}$ we set $p_t^{{\gsw}} = p_{t - 1}^{{\gsw}}$, or else
		choose a pivot $p_t^{{\gsw}}$ uniformly at random from $\cA^{{\gsw}}_t$.
		
		\item With $p^{{\gsw}}_t$ already defined, update $\mv z^{{\gsw}}_t,\cA^{{\gsw}}_{t +
			1},\mv u^{{\gsw}}_t,\delta^{{\gsw}}_t$ in exactly the same way as in \nameref{label:case1}.
		
		\item Define $\eta_t$ in exactly the same way as in \nameref{label:case1}.
		
		\item Choose a pivot $p_t$ uniformly at random from $\cA_t$. Set $\delta_t = \eta_t$,  $\mv u_t = \mv u(p_t,\cA_t)$ and define
		\begin{equation}\label{eq:updatezt}
			\mv z_t = \mv z_{t - 1} + \delta_t \mv u_t.
		\end{equation}
		\item %Choose a pivot $p_t$ uniformly at random from $\cA_t$ and 
		Update $\cA_{t + 1} =  \cA_{t} - \{p_t\}$.
	\end{itemize}
	
	Set,
	\begin{equation}\label{def:F_t}
		\cF_t = \sigma(\cF_{t-1}, U_t, p_t, p_t^{{\gsw}}) \mbox{ so that } %the random variables } 
	p_t^{{\gsw}}, p_t, \cA_t^{{\gsw}}, \cA_t \mbox{ etc. are all }\cF_t\mbox{-measurable.}
\end{equation}
%Update $\cA_{t + 1},\cA^{(int)}_{t + 1}$ as in Case $1$.
%\begin{remark}\label{rmk:}
%In case $\cA_{t + 1}$ is empty then do not follow the above but stop updating $z_t$ and define 
%$z_{[t:n]} = z_{t - 1}.$ Similarly, if $\cA^{int}_{t + 1}$ is empty then do not follow the above but 
%simply define $z^{int}_{[t:n]} = z^{int}_{t - 1}$. %{\color{red} Define the process when the active set becomes empty...how to exactly set this convention...make a box around the entire defn? no need do this in the intro inside the proof sketch....leaving this for Subhajit.}
%\end{remark}
%We need to introduce some further notations. Recall that we write the rows of $Y$ as $y_{1}^\t, 
%y_{2}^\t,\ldots,y_{n}^\t \in \R^{d}$. Fix a time $t$ where the pivot column index (for the GSW algorithm) is $p_t = \pi(t)$ which lies in the 
%set of active indices $\cA_t$ whose cardinality we henceforth denote as $a_t$. For future use, we also 
%denote $\cA^-_t = \cA_t \setminus \{p_t\}$. Notice that $a_t = n - t + 1$ and $\cA_t^- = \mathcal 
%A_{t + 1}$ in our case.
%\supset \cA_{t + 1}$. %and $\cA_{t}$ is the set of active vertices of size $a_t:=\abs{\cA_t}=n-t+1$. 
\begin{observation}\label{obs:zellinfty}
	In view of the displays~\eqref{def:Delta}--\eqref{eq:ztgsupdate}, it follows that $\mvz_t \in [-1, 1]^n$ for
	all $t \in [n] \cup \{0\}$. Furthermore, from our rule of updating $\cA_{t + 1}^{{\gsw}}$ in
	\eqref{eq:Atgssupdate} we have that $\|\mvz_t^{{\gsw}}[\cA_{t+1}^{{\gsw}}]\|_{\infty} < 1$.
\end{observation}
\begin{observation}\label{obs:swor}
	It is clear from the definition of $\cA_t$ its elements are drawn at random without replacement
	from $[n]$. In other words, $\cA_t$ is distributed uniformly over all subsets of $[n]$ with
	cardinality $n - t + 1$.
\end{observation}
\subsection{Definitions of relevant Martingales} \label{sec:martingales}
Since we want to prove the CLT for the random variable $\langle\mvz_n^{{\gsw}}, \mvv\rangle$, let
us first consider the process $\{M_t^{\rm {gs}}\}_{t = 0}^{n}$ defined as
\begin{equation}\label{def:Mgs}
	M_t^{{\gsw}} = \left\langle  \mv z^{{\gsw}}_t, \mvv \right\rangle \stackrel{\eqref{eq:ztgsupdate}}{=}
	\sum_{s \in [t]} \delta^{{\gsw}}_s \left\langle \mv u^{{\gsw}}_s, \mv v \right\rangle
	\stackrel{\eqref{def:B}}{=} \sum_{s \in [t]} \delta^{{\gsw}}_s \left \langle B \mv u^{{\gsw}}_s, \mat{\mv
		v}{0} \right\rangle
\end{equation}
where we omit the dependence on $\mvv$. It follows from our definitions in the previous subsection that the process $\{M_t^{{\gsw}}\}_{t =
	0}^n$ is adapted to the filtration $(\cF_t)_{t = 0}^n$. Furthermore in view of~\eqref{def:delta_t},
we have that
\begin{equation*}
	\E [ \delta_t^{{\gsw}} \, | \, \cF_{t-1}]	= 0 \mbox{ for all } t \in [n]
\end{equation*}
and hence the process $\{M_t^{\rm {gs}}\}_{t = 1}^{n}$ is a martingale. We now define two additional
$(\cF_t)_{t = 0}^n$-adapted processes, namely,
\begin{equation}\label{def:MttildeMt}
	\begin{split}
		&\tilde{M}_t \coloneqq \left\langle  \mv z_t,\mvv \right\rangle  \stackrel{\eqref{eq:ztgsupdate},~\eqref{eq:updatezt}}{=} \sum_{s \in [t]} \delta_s \left\langle \mv u_s, \mv v \right\rangle \stackrel{\eqref{def:B}}{=} \sum_{s \in
			[t]} \delta_s \left \langle B \mv u_s, \mat{\mv v}{0} \right\rangle, \mbox{ and } \\
		&M_t \coloneqq \sum_{s \in [t]} \eta_s \left\langle \frac{B\mv u_s}{\|B \mvu_s\|}, \mat{\mv v}{0}
		\right\rangle.
	\end{split}
\end{equation}
By the same reasoning as before, the processes $\{\tilde M_t^{{\gsw}}\}_{t = 0}^n$ and $\{M_t^{{\rm
		gs}}\}_{t = 0}^n$ are also $(\cF_t)_{t = 0}^n$-martingales. Observe that the process
$\{M_t^{{\gsw}}\}_{t = 0}^n$ admits of the following nice interpretation.

The vectors $\left(\tfrac{B\mv u_1}{\|B \mvu_1\|}, \ldots, \tfrac{B\mv u_n}{\|B \mvu_n\|}\right)$ are
obtained by applying the classical Gram-Schmidt process to a uniformly random rearrangement of the
columns of $B$, namely $(\mvb_{p_1}, \ldots, \mvb_{p_n})$. The variables $\{\eta_t\}_{t = 0}^n$, on the
other hand, are i.i.d. Rademacher and independent of the random permutation $\pi \coloneqq (p_1,\ldots, p_n)$. We will prove the CLT directly for this {\em skeletal} martingale $\{M_t\}_{t= 0}^n$.

The process $\{\tilde{M}_t\}_{t = 1}^{n}$ is just an intermediate process between $\{M_t^{{\rm
		gs}}\}_{t = 1}^{n}$ and $\{{M}_t\}_{t = 1}^{n}$ and will only be used Section~\ref{sec:reduction_ideal} to
transfer the CLT from $M_n$ to $M_n^{{\gsw}}$.
%The process $\{M_t\}_{t = 1}^{n}$ would henceforth be referred to as the \textit{ideal} process.
%{\color{red} explain that $M_t$ is the ideal process...the differences with the intermediate process..what is ideal about it? maybe this goes in the intro?}
%\begin{remark}
%Note that with the way we have defined the three martingale processes above, they are coupled and defined on the same probability space. This is important for our anaysis. 
%\end{remark}

%\input{prelim}
\section{Preliminaries}\label{sec:prelim}
In this section we lay the groundwork for our main arguments in the remainder of the paper. The proofs
of the results in this section can be skipped on a first reading. \subsection{Some linear algebraic identities}\label{subsec:lin_alg}
We first derive some identities which only uses matrix algebra. We collect these identities in a
stand-alone lemma because these are used repeatedly in our analysis.%

\begin{lemma}\label{lem:matrix_identities}
	Fix any non-empty subset $\cA \subset [n]$ and an index $p \in \cA$. Given a full
	column rank matrix $B \in \R^{(n + d) \times n}$, recall the definition of the vector $\mvu = \mv
	u(p,\cA)\in \R^n$ from~\eqref{def:u_ortho_project}. Now consider the case when
	$$B=\mat{I_n}{Y^{\t}}$$ where $Y$ is an $n\times d$ matrix,	and $\max_{i \in [n]} \|\mvy_i\| \leq \zeta$  (recall~\eqref{def:zeta}) 	with $\mvy_i^\t$ (a row vector) denoting the $i$-th row of $Y$.
	We will use the shorthands $\cA^{-} \coloneqq \cA - \{p\}$, $Y_\star = Y_{\star}(\cA) \coloneqq
	Y[\cA, \,:]$, $Y_\star^{-} \coloneqq Y_\star[\cA^{-}, \, :]$ and $D \coloneqq \left(I_{d} + Y_\star^\t
	Y_\star\right)^{-1}$.
	Then $\mvu$ has the following expression:
	\begin{equation}\label{eq:identity1}
		\mv u = \frac{1}{1-\mvy_p^{\t}D \mvy_p} \:\:\left(\mve_p - I_n[:\, , \cA] Y_\star D \mvy_p\right)
	\end{equation}
	where $\mve_p \in \R^n$ is unit vector whose $p$-th element equals 1. The norm of $B\mvu$ has also a simple formula, namely,
	\begin{equation}\label{eq:identity2}
		\norm{B \mv u}^{2} = \frac1{1-\mvy_p^{\t}D \mvy_p} \leq 1 + \zeta^2(1 + \zeta^2) = \Cr{C:xi}
	\end{equation}
	(recall~\eqref{def:Czeta}). Notice that $\|B\mvu\| \ge 1$. The inner product $\langle\mvu, \mvv\rangle = \left \langle B\mv u, \mat{\mv v}{0} \right \rangle$, on the
	other hand, admits of the expression
	\begin{equation}\label{eq:identity3}
		\begin{split}
			%\left \langle B\mv u, \mat{\mv v}{0} \right \rangle 
			\langle \mvu, \mvv\rangle = \norm{B\mvu}^{2}\mvv[\cA]^{\t}(I_{a} - Y_\star D Y_\star^{\t})\mve_p[\cA]
		\end{split}
	\end{equation}
	where $a \coloneqq |\cA|$. Finally, we would often need to deal with the normalized inner
	product in our analysis; hence an equivalent and convenient way to rewrite the above identity is the
	following.
	\begin{equation}\label{eq:identity4}
		\begin{split}
			%\norm{B\mvu}^{-2}\left \langle B\mv u, \mat{\mv v}{0} \right \rangle^2 
			\norm{B\mvu}^{-2} \langle \mv u, \mv v \rangle^2 &= \norm{B\mvu}^{2}
			\left(\mvv[\cA]^{\t}(I_{a} - Y_\star D Y_\star^{\t})\mve_p[\cA]\right)^2\\
			%\mve_p[\cA]^{\t} (I_{a} - 
			%YCY^{\t})\mvv[\cA]\\
			&= \cQ + (\norm{B\mvu}^{2} - 1)\cQ
		\end{split}
	\end{equation}
	where
	\begin{equation}\label{eq:identity5}
		\cQ = \cQ(\cA) \coloneqq \left(\mvv[\cA]^{\t}(I_{a} - Y_\star D Y_\star^{\t})\mve_p[\mathcal
		A]\right)^2 = \norm{B\mvu}^{-4} \langle \mv u, \mvv\rangle^2.
	\end{equation}%	\item If $p$ is chosen uniformly at random from $\cA$ then we have  
\end{lemma}

\begin{proof}
	\noindent{\em Proof of~\eqref{eq:identity1}}. Within this proof, we will also use the notation $$D^{-} =
	\left(I_{d} + {Y_\star^-}^{\t}Y_\star^-\right)^{-1}$$ as well as the notation
	$$B^{-} = B[:\,, \cA^{-}] =  \mat{I_n[:\,,\cA^{-}]}{{Y_\star^{-}}^{\t}}.$$
	By definition~\ref{def:u_ortho_project} of $B\mvu$ and the standard formula for orthogonal projections, we can write
	\begin{equation}\label{eq:expr_Bu_t}
		B \mvu = B \mve_p - B^{-}\left({B^-}^{\t} B^{-}\right)^{-1} {B^{-}}^{\t} \:B\mve_p.
	\end{equation}
	Now, we can write ${B^{-}}^{\t} B^{-} = I_{a - 1} + Y_\star^- {Y_\star^-}^{\t}$
	(recall that $a = |\cA|$) which leads to, by using the Sherman-Woodbury-Morrison formula (see, e.g., \cite{MR1927606}),
	\begin{equation}\label{eq:bTbinv}
		\begin{split}
			\left({B^{-}}^{\t} B^{-}\right)^{-1} = I_{a - 1} - Y_\star^- (I_{d} + {Y_\star^-}^{\t} Y_\star^-)^{-1} {{Y_\star^-}}^{\t}
			\stackrel{\eqref{def:XX_tBC}}{=} I_{a - 1} - Y_\star^- D^- {Y_\star^-}^{\t}.
		\end{split}
	\end{equation}
	From this, we get,
	\begin{equation}\label{eq:bTbinvbTb}
		\begin{split}
			\mve_p - \mv u &\stackrel{\eqref{eq:expr_Bu_t}}{=} I_n[:\, , \cA^-]\left({B^{-}}^{\t} B^{-}\right)^{-1}
			(B^{-})^{\t} \cdot B\mve_p \stackrel{\eqref{def:XX_tBC}}{=} I_n[:\,, \cA^-]\left({B^{-}}^{\t} B^{-}\right)^{-1} {Y_\star^-}\mvy_p \\
			&\stackrel{\eqref{eq:bTbinv}}{=} I_n[:\,, \cA^-]{Y_\star^-}(I_d  - D^-{Y_\star^-}^{\t}Y_\star^-)\mvy_p\\
			&
			\stackrel{\eqref{def:XX_tBC}}{=} I_n[:\,, \cA^-]{Y_\star^-}\big(I_d  - D^-((D^-)^{-1} - I_d) \big)\mvy_p \\
			&\stackrel{}{\ \,=} I_n[:\,, \cA^-]Y_\star^{-}D^{-}\mvy_p = I_n[:\,, \cA]Y_\star D^-\mvy_p - \mvy_p^{\t}D^-\mvy_p \, \mve_p.
		\end{split}
	\end{equation}
	Also since
	\[I_{d} + {Y_\star^-}^{\t}Y_\star^- = I_{d} + Y_\star^{\t}Y_\star - \mvy_p\cdot \mvy_p^{\t},\]
	we get from the Sherman-Woodbury-Morrison formula,
	\begin{align*}
		D^- & = D + \frac1{1-\mvy_p^{\t}D \mvy_p}\cdot D \mvy_p\mvy_p^{\t}D.
	\end{align*}
	This immediately gives us
	\begin{equation}\label{eq:C_txp}
		D^-\mvy_p = \frac{D \mvy_p}{1-\mvy_p^{\t}D \mvy_p}.
	\end{equation}
	Plugging this into the right-hand side of~\eqref{eq:bTbinv} yields~\eqref{eq:identity1}.
	
	\smallskip
	
	\noindent{\em Proof of~\eqref{eq:identity2}}. In view of~\eqref{def:B}, we can write
	\begin{equation}\label{expr:But}
		\begin{split}
			B\mvu &= \mat{\mvu}{Y_\star^{\t} \mvu} \stackrel{\eqref{eq:bTbinvbTb}}{=} \mat{\mve_p - I_{n}[:,\cA^-]Y_\star^- D^- \mvy_p}{\mvy_p - {Y_\star^-}^{\t}Y_\star^-D^{-}\mvy_p} = \mat{\mve_p - I_{n}[:,\cA^-]Y_\star^- D^- \mvy_p}{\mvy_p - ((D^{-})^{-1} - I_d)D^{-}\mvy_p} \\
			&= \mat{\mve_p - I_{n}[:,\cA^-]Y_\star^- D^- \mvy_p}{D^- \mvy_p}.
		\end{split}
	\end{equation}
	Thus we get
	\begin{equation}\label{eq:norm_butsq}
		\begin{split}
			\norm{B\mvu}^{2} &=  1 + \mvy_p^{\t}D^-{Y_\star^-}^{\t}Y_\star^-D^-\mvy_p + \mvy_p^{\t} (D^-)^2
			\mvy_p \\
			&\stackrel{\eqref{def:XX_tBC}}{=} 1 + \mvy_p^{\t}D^-((D^-)^{-1} - I_d)D^-\mvy_p + \mvy_p^{\t}
			(D^-)^2 \mvy_p = 1 + \mvy_p^{\t} D^- \mvy_p (\ge 1).
		\end{split}
	\end{equation}
	Now plugging~\eqref{eq:C_txp} into the right-hand side of~\eqref{eq:norm_butsq}, we obtain
	\begin{equation}\label{eq:norm_butsq2}
		\norm{B\mvu}^{2} = \frac1{1-\mvy_p^{\t}D\mvy_p}.
	\end{equation}
	
	All that remains is to bound the above. Let us start with $\norm{B\mvu}^{2} - 1$.
	\begin{equation}\label{eq:Butnormbnd}
		\begin{split}
			\norm{B\mvu}^{2} - 1 \stackrel{\eqref{eq:norm_butsq2}}{=}& \mvy_p^{\t}D\mvy_p\norm{B\mvu}^{2}
			\stackrel{\eqref{eq:norm_butsq}}{=} \mvy_p^{\t}D\mvy_p(1 + \mvy_p^{\t} D^- \mvy_p)\\ \le& \mvy_p^{\t}D\mvy_p(1 + \|D^-\|_{{\rm op}} \|\mvy_p\|^2)\stackrel{\eqref{def:XX_tBC}}{\le} (1 +  \zeta^2)
			\,\mvy_p^{\t}D\mvy_p \leq  \zeta^2(1 + \zeta^2).
		\end{split}
	\end{equation}
	where in the last two steps we also used the fact that $\max_{i \in [n]} \|\mvy_i\| \leq \zeta$. This
	finishes the proof of~\eqref{eq:identity2}.
	
	\smallskip
	
	\noindent{\em Proofs of~\eqref{eq:identity3} and~\eqref{eq:identity4}}. Plugging~\eqref{eq:C_txp},
	\eqref{eq:norm_butsq2} as well as the observation
	$$I_{n}[:,\cA]Y_\star = I_{n}[:,\cA^-]Y_\star^- + \mve_p \mvy_p^{\t}$$
	into~\eqref{expr:But}, we get
	\begin{equation}\label{expr:Butalt}
		B\mvu = \norm{B\mvu}^{2}\mat{\mve_p - I_{n}[:,\cA]Y_\star D \mvy_p}{D \mvy_p}.
	\end{equation}
	Now we are ready to evaluate the inner product
	\begin{equation}
		\begin{split}
			\left \langle B\mv u, \mat{\mv v}{0} \right \rangle &\stackrel{\eqref{expr:Butalt}}{=}
			\norm{B\mvu}^{2}(\mvv[\cA]^{\t}\mve_p - \mvv[\cA]^{\t}Y_\star D\mvy_p) \\
			&= \norm{B\mvu}^{2}\mvv[\cA]^{\t}(I_{a} - Y_\star D Y_\star^{\t})\mve_p[\cA]
		\end{split}
	\end{equation}
	(recall that $p \in \cA$), hence~\eqref{eq:identity3}.
	As to~\eqref{eq:identity4}, we can just expand
	\begin{equation}\label{eq:Q_told}
		\begin{split}
			\norm{B\mvu}^{-2}\left \langle B\mv u, \mat{\mv v}{0} \right \rangle^2 &= \norm{B\mvu}^{2}
			\mvv[\cA]^{\t}(I_{a} - Y_\star D Y_\star^{\t})\mve_p[\cA]\mve_p[\cA]^{\t} (I_{a} -
			Y_\star D Y_\star^{\t})\mvv[\cA]\\
			&= \cQ + (\norm{B\mvu}^{2} - 1)\cQ
		\end{split}
	\end{equation}
	where $\cQ$ is as defined in~\eqref{eq:identity5}.
\end{proof}
In the sequel, we we will denote
\begin{equation}\label{def:YD}
	Y_t = Y_{\star}(\cA_t) = Y[\cA_t, \,:],\,\,  D_t = (I_d + Y_t^{\t}Y_t)^{-1}
\end{equation}
and similarly for $Y_t^{{\gsw}}$ and $D_t^{{\gsw}}$ with $\cA_t$ replaced by $\cA_t^{{\gsw}}$.

\subsection{Some results on random sampling without replacement}
Recall from Observation~\ref{obs:swor} that $\cA_t$ is a random sample of size $n - t + 1$ from
$[n]$ without replacement. In the currect subsection we will gather some results about the behavior of
the sum of numbers or matrices indexed by $[n]$ evaluated over $\cA_t$. For clarity of
presentation, we will state these results for a general family of objects $\{x_1, \ldots, x_n\}$ --- which
would be either real numbers or matrices --- and a random sample $\cA$ of size $a$ from
$[n]$ without replacement. Denote by $\pr$ and $\E$ the corresponding probability measure and
expectation respectively and consider the random variable $W = W(\cA)$ defined as
\begin{equation*}
	W = \sum_{i \in \cA} x_i \, \mbox{ so that }	\E[W] = \frac an \sum_{i \in [n]} x_i.
\end{equation*}

Our first result concerns the concentration of $W$ around $\E[W]$ when $x_i$'s are matrices with bounded operator norm.
\begin{proposition}\label{prop:lambdaminconc}
	Let $x_1, \ldots, x_n$ denote $d \times d$ symmetric matrices with $\max_{i \in [n]}\|x_i\|_{{\rm op}} \le
	1$. Then for any $a \in [n]$ and $x \ge 0$, we have
	\begin{equation}\label{eq:lambdaminconc}
		\pr\left[\norm{W-\E[W]}_{{\rm op}}\ge x \right] \le 2d\cdot \exp(-nx^{2}/2a(n-a)).
	\end{equation}
	%For {\color{red}any} $t \in [n]$ and $\varepsilon \in (0, 1)$, we have with $\hat \lambda_{{\min}} \coloneqq 
	%\lambda_{{\rm min}}(YY^{\t})$,
	%\begin{equation*}
	%\P\left[ \lambda_{{\rm min}}(Y_tY_t^{\t})  \le  (1 - \varepsilon) \frac{a_t}{n}\,\hat \lambda_{{\min}}\right]	\le C \exp\left( -c \varepsilon^2\frac{a_t}{n^2}\hat \lambda_{{\min}}^2\right).
	%\end{equation*}
\end{proposition}
\begin{proof}
	We use Stein's method for exchangeable pairs to prove the concentration inequality. Let $\pi$ be a
	random uniform permutation of $[n]$. It is easy to see that $\cA\equald \{\pi(i)\mid i\in[a]\}$. Thus, we
	can define
	\[
	W=\sum_{i\in[a]}x_{\pi(i)}.
	\]
	We create an exchangeable pair in the following way. Let $(I,J)$ be a uniform sample from the set
	$\{(i,j)\mid 1\le i<j\le n\}$ {\em independent} of $\pi$. Define $$\pi' = \pi\circ (I,J) \mbox{ and }  W' = \sum_{i\in[a]}x_{\pi'(i)}\,.$$
	One can check that
	$$\gD W:=W'-W=(x_{\pi(J)} - x_{\pi(I)})\cdot \ind\{I\le a < J\}.$$
	In particular, we have
	\[
	\E[\gD W\mid \pi]=\frac{2}{n(n-1)}\left(a\sum_{j > a}x_{\pi(j)} - (n-a)\sum_{i\le a}x_{\pi(i)}\right) = -\ga\cdot \left( W-\E[W]\right)
	\]
	where $\ga=\frac{2}{n-1}$. Similarly, we have
	\begin{align*}
		\E[\gD W^{2}\mid \pi] & = \E[(x_{\pi(J)}^{2}+x_{\pi(I)}^{2}-x_{\pi(I)}x_{\pi(J)}-A_{\pi(J)}x_{\pi(I)})\cdot \ind\{I\le t<J\}\mid \pi] \\
		& \le \E[(x_{\pi(J)}^{2}+x_{\pi(I)}^{2})\cdot \ind\{I\le t<J\}\mid \pi]
		\le 2\pr\left[I\le t<J\right]\cdot I_{d}\\
		& = \frac{4a(n-a)}{n(n-1)}\cdot I_{d},
	\end{align*}
	\ie 
	\[
	\frac1{2\ga}\E[\gD W^{2}\mid \pi] \le \frac{a(n-a)}{n}\cdot I_{d}.
	\]
	Here, for symmetric matrices ``$\le$'' means positive definite ordering.
	Using \cite[Theorem~4.1]{MR3189061}, we get for all $x \ge 0$
	\[
	\pr\left[\gl_{\min}(W-\E[W])\le -x\right]\vee \pr\left[\gl_{\max}(W-\E[W])\ge x\right] \le d\cdot e^{-nx^{2}/(2a(n-a))}.
	\]
	This completes the proof.
\end{proof}

The following result gives formulae for the moments of $W$ when $x_i$'s are real numbers. Their proofs 
are relatively standard, which we provide for the sake of completeness.

\begin{prop}%[Sampling Without Replacement Result]
	\label{thm:srswor_moment}
	Let $x_1, \ldots, x_n$ be real numbers satisfying $\sum_{i \in [n]} x_i = 0$. %and $\cA \subset [n]$ be a uniformly random subset of $[n]$ with cardinality $m > 0$. Then, for $$W = W(\cA) \coloneqq \sum_{i \in \cA} x_i,$$
	Then for any $a \in [n]$, we have
	\begin{equation}\label{eq:srswor_moment}
		\begin{split}
			&\E \left[W\right] =  0,\, 	\E [W^2] = \frac{a(n - a)}{(n)_2}\sum_{i \in [n]}x_i^2, \mbox{ and }\\
			&\E[W^4] = \frac{3(a)_2(n-a)_2}{(n)_4} \big(\sum_{i \in [n]} x_i^2 \big)^2 +
			\frac{a(n-a)}{(n)_2} \left(1 - \frac{6(a - 1)(n - a - 1)}{(n - 2)(n - 3)}\right)\sum_{i \in [n]}x_i^4,
		\end{split}
	\end{equation}
	where $(n)_k \coloneqq n(n-1)\ldots(n - k + 1)$ is the $k$-th downward factorial of $n$.
\end{prop}

\begin{proof}
	Let $\pi$ be a uniform random permutation of the set $[n]$. It is clear to see that $\{ \pi(i)\mid 1\le i\le a \}$ has the same distribution as $\cA$. So we can take $W=\sum_{i=1}^a x_{\pi(i)}$. 
	The zero-mean result follows from the fact that $\sum_{i \in [n]} x_i = 0$. For the second moment result, we note that $\E[W^2]=a\E[x_{\pi(1)}^2] + a(a-1)\E[x_{\pi(1)}x_{\pi(2)}]$. Moreover, we have $$\E[x_{\pi(1)}x_{\pi(2)}]= \E\left[x_{\pi(1)}\cdot\frac{1}{n-1}\sum_{i\neq \pi(1)}x_{i}\right] = -\frac1{n-1}\E[x_{\pi(1)}^2]$$ and $\E[x_{\pi(1)}^2] = \frac1n\sum_{i\in[n]}x_i^2.$  Similarly, for the fourth moment, we get
	\begin{align*}
		\E[W^4] &=a\E[x_{\pi(1)}^4] + 4(a)_2\E[x_{\pi(1)}^3x_{\pi(2}] +3(a)_2 \E[x_{\pi(1)}^2x_{\pi(2)}^2]\\
		&\quad + 6(a)_3 \E[x_{\pi(1)}^2x_{\pi(2)}x_{\pi(3)}] + (a)_4 \E[x_{\pi(1)}x_{\pi(2)}x_{\pi(3)}x_{\pi(4)}].
	\end{align*}
	Moreover, we have
	\begin{align*}
		\E[x_{\pi(1)}x_{\pi(2)}x_{\pi(3)}x_{\pi(4)}]
		&= -\frac{3}{n-3} \E[x_{\pi(1)}^2x_{\pi(2)}x_{\pi(3)}],\\
		\E[x_{\pi(1)}^2x_{\pi(2)}x_{\pi(3)}] 
		&= -\frac{1}{n-2} \E[x_{\pi(1)}^2x_{\pi(2)}^2 + x_{\pi(1)}^3x_{\pi(2)}],\\
		\E[x_{\pi(1)}^3x_{\pi(2)}] &= -\frac{1}{n-1}\E[x_{\pi(1)}^4]\\
		\text{and } 
		\E[x_{\pi(1)}^2x_{\pi(2)}^2] &= \frac{1}{n-1}\sum_{i\in[n]}x_i^2\cdot \E[x_{\pi(1)}^2] - \frac{1}{n-1} \E[x_{\pi(1)}^4].
	\end{align*}
	Simplifying, we get the result.
\end{proof}

\subsection{The events $\cG_{1, t}$ and $\cG_{2, t}$}
Recall the events $\cG_{1, t}$ and $\cG_{2, t}$ defined in~\eqref{def:G1sG2s}. One of the important
implications of our next result is that the condition $\cA_s^{{\gsw}} = \cA_s$ required
for \nameref{label:case1} in Section~\ref{sec:stoch_proc} is {\em in fact} redundant and hence the
processes $\{M^{{\gsw}}_s\}_{s = 0}^{t}$ and $\{\tilde{M}_s\}_{s = 0}^{t}$ are identical (see~\eqref{def:Mgs} and~\eqref{def:MttildeMt} for definitions) whenever both $\cG_{1, s}$ and
$\cG_{2, s}$ occur for all $s \in [t]$. Also it turns out that for all such $s$, $\delta_s$ and $\eta_s$
(see~\eqref{def:etat}) are ``$\varepsilon_n$-close'' to each other. These observations will help us in the
next section to {\em reduce} the CLT of $M_n^{{\gsw}}$ to that of $M_n$
(Proposition~\ref{prop:main1}).

We now proceed to state the lemma. To this end let us define, for any $t \in [n]$,
\begin{equation}\label{def:E1E2}
	\mathcal E_{1, t} = \bigcap_{s \in [t]} \cG_{1, s} \mbox{ and } \mathcal E_{2, t} = \bigcap_{s \in
		[t]} \cG_{2, s}.
\end{equation}
\begin{lemma}\label{lem:pivotlemma}
	Suppose that $t \in [n]$ satisfies
	\begin{equation}\label{eq:req}
		t < n + 1 - 6 \Cr{C:xi} \zeta^2 \kappa
	\end{equation}
	where $\kappa = \tfrac n{\lambda_{{\rm min}}(Y^{\t}Y)}$ as already defined in Theorem~\ref{thm:formal}. %If both the events $E_{1,t} = \cap_{s = 1}^{t} 	\mathcal{G}_{1,s}$ and $E_{2,t} = \cap_{s = 1}^{t} \mathcal{G}_{2,s}$ hold then 
	Then on the event $\cE_{1,t} \cap \cE_{2,t}$, one has $\cA^{{\gsw}}_s = \cA_s$ and consequently
	$\delta_s^{{\gsw}} = \delta_s$ and $M^{{\gsw}}_s = \tilde{M}_s$ for all $s \in [t]$. %Also we have 
	Furthermore, we have
	\begin{equation}\label{depart:delta}
		\max \{|\delta^{+}_{s} - 1|,|\delta^{-}_{s} - 1|\} \leq \varepsilon_n%C \varepsilon_n
	\end{equation}
	for all $s \in [t]$.
\end{lemma}
\begin{proof}
	We will show this via induction. Recall from Section~\ref{sec:stoch_proc} that $\cA^{{\gsw}}_1 =
	\cA_1 = [n]$ and $\mvz_0^{{\gsw}} = \mvz_0 = \mv0$. Thus the events $\cG_{1, 1} = \cE_{1, 1}$ and
	$\cG_{2, 1} = \cE_{2, 1}$ occur almost surely in view of~\eqref{def:G1sG2s}. On the other hand,
	we obtain from~\eqref{def:delta_t} that $\delta^{+}_1 = \delta^{-}_1 = 1$. Hence the base case of
	induction is covered. Now fix some positive integer $t < n - 6\Cr{C:xi}\zeta^2\kappa$ and suppose that
	the event $\cE_{1,t+1} \cap \cE_{2,t+1}$ occurs. Since $\cE_{1,t+1} \cap \cE_{2,t+1} \subset \cE_{1,t}
	\cap \cE_{2, t}$, we have from our induction hypothesis
	\begin{equation*}
		\cA^{{\gsw}}_s = \cA_s
	\end{equation*}
	as well as
	\begin{equation*}
		\max \{|\delta^{+}_{s} - 1|,|\delta^{-}_{s} - 1|\} \leq \varepsilon_n  \:\:\forall\:\: s \in [t].
	\end{equation*}
	Therefore it suffices to show that $\cA^{{\gsw}}_{t + 1} = \cA_{t + 1}$ and $\max \{|\delta^{+}_{t + 1} -
	1|,|\delta^{-}_{t + 1} - 1|\} \leq \varepsilon_n$ on the event $\cE_{1, t + 1} \cap \cE_{2, t + 1}$.
	%Note that the above will imply that $$M^{{\gsw}}_j = \tilde{M}_j \:\:\forall\:\: j \in [t + 1]$$ just by definition of these processes as in Section~\ref{sec:stoch_proc}.
	
	%the lemma is true for some integer $t - 1$ with $t > 2.$ We will argue that the lemma holds for $t$ as well as long as $t$ satisfies~\eqref{eq:req}. 
	
	%The induction hypothesis implies that $$\cA_j = \cA^{int}_j \:\:\forall \:\:j \in [t - 1].$$ 
	
	We first verify that $\cA_{t+1} = \cA_{t+1}^{{\gsw}}$. Since we are in the purview of
	\nameref{label:case1} on the event $\cE_{1, t} \cap \cE_{2, t}$ ($\supset \cE_{1, t+1} \cap \cE_{2, t+1}$)
	at round $t$ by our induction hypothesis, %and hence, in particular, for $s = t -1$ and $t$, 
	we have $p_t^{{\gsw}} = p_t = p$ (say) and $\mvz_t^{{\gsw}} = \mvz_t$. Therefore in view of
	\eqref{eq:Atgssupdate}, it suffices to prove that
	\begin{equation*}
		\left\{i \in \cA_t^{{\gsw}} : |\mvz_t^{{\gsw}}[i]|  = 1 \right\} = \left\{i \in \cA_t : |\mvz_t[i]|  = 1 \right\} = \{p_t\}
	\end{equation*}
	(recall from Observation~\ref{obs:zellinfty} that $\|\mvz_t^{{\gsw}}\|_{\infty} \le 1$).  However, due to
	\eqref{def:Delta}--\eqref{eq:ztgsupdate}, this would follow from
	\begin{align}\label{eq:onlypivot}
		& \mv z_{t - 1}[\cA_t \setminus \{p\}] + (1 - \mvz_{t - 1}[p]) \mv u_t[\cA_t \setminus \{p\}] \in (-1,1)^{|\cA_t \setminus \{p\}|}\:\:\text{and} \nonumber \\&
		\mv z_{t - 1}[\cA_t \setminus \{p\}] - (1 + \mvz_{t - 1}[p]) \mv u_t[\cA_t \setminus \{p\}] \in (-1,1)^{|\cA_t
			\setminus \{p\}|}.
	\end{align}
	(recall that $\mv u_t^{{\gsw}} = \mv u_t$ since we are in \nameref{label:case1}) which we now
	proceed to show. To this end, we first bound for any $i \in  \cA_t \setminus \{p\}$,
	\begin{equation*}
		|\mvz_{t - 1}[i] + (1 - \mvz_{t - 1}[p]) \mvu_t[i]| \leq |\mvz_{t - 1}[i]| + |1 - \mvz_{t - 1}[p]| \, |\mvu_{t}(i)| \leq \frac{1}{2} + \frac{3}{2}  |\mvu_{t}(i)|
	\end{equation*}
	where in the final step we used $\|\mvz_{t-1}[\cA_t]\|_{\infty} < \varepsilon_n \le
	\tfrac12$ as we are on the event $\cG_{1, t} \supset \cE_{1, t + 1} \cap \cE_{2, t + 1}$
	(recall~\eqref{def:G1sG2s}). The same bound also holds for $|\mvz_{t - 1}[i] - (1 + \mvz_{t
		- 1}[p]) \mvu_t[i]|$. Hence it is enough to show that
	\begin{equation}\label{eq:max_bnd}
		\max_{i \in \cA_t \setminus \{p\}}\, \left( \frac{1}{2} + \frac{3}{2} |\mvu_{t}(i)| \right) < 1.
	\end{equation}
	Now recall the definitions of $Y_t$ and $D_t$ from~\eqref{def:YD} and also that $\mvu_t = \mvu(p_t,
	\cA_{t-1})$. Using~\eqref{eq:identity1} we can write for any $i \in \cA_t \setminus \{p\}$,
	\begin{equation}\label{eq:uti}
		\begin{split}
			&|\mvu_t[i]| \stackrel{\eqref{eq:identity1}}{=} \frac{\mvy_i^{\t} D_t \mvy_p}{1 - \mvy_p^{\t} D_t \mvy_p}
			\stackrel{\eqref{eq:identity2}}{\leq} \Cr{C:xi} \:\mvy_i^{\t} D_t \mvy_p \leq \Cr{C:xi} \|\mvy_i\|\|D_t\|_{{\rm op}}\|\mvy_p\| \\
			\stackrel{\eqref{def:zeta}}{\leq} &  \frac{\Cr{C:xi} \:\zeta^2}{1 + \lambda_{\min}\big(Y_t^{\t} Y_t\big)}
			\stackrel{\eqref{def:G1sG2s}}{\leq} \frac{2n \Cr{C:xi} \zeta^2}{(n - t + 1)\lambda_{\min}\big(Y^{\t}
				Y\big)} = \frac{2\kappa \Cr{C:xi} \zeta^2}{(n - t + 1)} \:\: \big(\because \kappa = \textstyle{\frac
				n{\lambda_{\min}(Y^\t Y)}} \big)
		\end{split}
	\end{equation}
	where in the penultimate step we used the fact that we are on the event $\cG_{2, t} \supset
	\mathcal E_{1, t + 1} \cap \mathcal E_{2, t + 1}$ so that by the Weyl's inequality (see, e.g.,
	\cite{franklin2012matrix}):
	\begin{equation}\label{eq:weyl}
		\begin{split}
			\left|\lambda_{{\rm min}}(Y_{t}^{\t} Y_{t}) -  \frac{n - t + 1}{n} \lambda_{{\rm min}} (Y^{\t} Y)\right| &\leq
			\big\|Y_{t}^{\t} Y_{t} -  \frac{n - t + 1}{n} Y^{\t} Y\big\|_{\rm op} \\
			&\leq \frac{n - t + 1}{2n} \lambda_{{\rm \min}}\big(Y^{\t}Y\big).
		\end{split}
	\end{equation}
	However, the final term in~\eqref{eq:uti} is bounded above by $\tfrac13$ in view of our assumption~\eqref{eq:req}
	yielding~\eqref{eq:max_bnd}.

	It remains to prove that $\max \{|\delta^{+}_{t + 1} - 1|,|\delta^{-}_{t + 1} - 1|\} \leq \varepsilon_n$. %Since we are on the event $\mathcal E_{1, t + 1} \cap \mathcal E_{2, t+1}$ and $\cA_{s}^{{\gsw}} = \cA_{s}$ for all $s \in [t+1]$ (which we just proved), we are in \nameref{label:case1} at round $t + 1$
	To this end, notice that the same argument applied to round $t + 1$ (recall that we are on the event
	$\mathcal E_{1, t + 1} \cap \mathcal E_{2, t+1}$) gives us
	\begin{equation*}
		\left\{i \in \cA_{t+1} : |\mvz_{t + 1}[i]|  = 1 \right\} = \{p_{t+1}\}.
	\end{equation*}
	This in turn implies that the maximum and minimum of the set $\Delta_{t+1}$ defined in
	\eqref{def:Delta} is achieved along the coordinate $p_{t+1}$. Also let us recall that $\mvu_{t+1}[p_{t +
		1}] = 1$ as $\mvu_{t+1} = \mvu(p_{t+1}, \cA_{t+1})$ (see~\eqref{def:u_ortho_project}) and
	$|\mvz_{t+1}[p_{t+1}]| < \varepsilon_n$ since we are on the event $\cG_{1, t+1} \supset
	\mathcal E_{1, t + 1} \cap \mathcal E_{2, t + 1}$. %it then follows that  %$\sup \Delta_{t + 1} \ge 0$  $\inf \Delta_{t + 1} \le 0$ and consequently 
	Together these imply that both $|1 - \delta_{t + 1}^+| = |1 - \sup \Delta_{t + 1}|$ and  $|1	- \delta_{t +
		1}^-| = |1 + \inf \Delta_{t + 1}|$ are bounded by $\varepsilon_n$. \qedhere%
\end{proof}
In our next two results, we give lower bounds on the probabilities of joint occurrence for $\cG_{1, t}$
and $\cG_{2, t}$'s starting with the latter.
\begin{lemma}\label{lem:smallevent2}
%Recalling the events for $s \in [n]$,
%\begin{equation*}
%\mathcal{G}_{2,s} \coloneqq \left\{\|Y_s^{\t} Y_s - \frac{n - s + 1}{n} Y^{T} Y\|_{{\rm op}} \geq  \frac{n - s + 1}{2n} \lambda_{\min}\big(Y^{\t} Y\big)\right\}
%\end{equation*}
% from~\eqref{def:G1sG2s}, %\begin{equation*}
%\mathcal{G}_{2,s} = \lambda_{\min}\big((Y[;,\cA_s])^{\t} Y[;,\cA_s]\big) \geq \frac{n - s + 1}{2n} 
%\lambda_{\min}\big(Y^{\t} Y\big)
%\end{equation*}
For any $t \in [n]$, we have
\begin{equation}\label{eq:smallevent2}
	\P\left[\cG_{2,t}^{c}\right] \leq C d \exp\left(-c \frac{n(n - t + 1)}{\kappa^2\zeta^4t}\right).
\end{equation}
Moreover, the following bounds hold on the event $\mathcal{G}_{2,t}$,
\begin{equation}\label{eq:Ctopbnd}
	%\begin{split}
	%\|D_t\|_{{\rm op}}  \le \left(1 + \frac{n - t + 1}{2n}\lambda_{{\rm min}}(Y^{\t}Y)\right)^{-1} \mbox{ and } 
	%\|Y_s \|_{\frob}^2 \le \frac{3}{2} \frac{n - s + 1}{n}\|Y\|_{\frob}^2\,.
	%\end{split}		
	\begin{split}
		\|D_t\|_{{\rm op}}  \le \frac{2\kappa}{n - t + 1} \, \mbox{ and }\, \|Y_t\|_{\frob}^2 \le \frac{3}{2}
		\frac{n - t + 1}{n}\|Y\|_{\frob}^2\,.
	\end{split}
\end{equation}
\end{lemma}
\begin{proof}
Notice that
\begin{equation*}
	\frac{1}{\zeta^2} \, Y_t^{\t} Y_t = \frac{1}{\zeta^2} \sum_{i \in \cA_t} \mvy_i \mvy_i^\t.
\end{equation*}
Now in view of Observation~\ref{obs:swor}, $\cA_t$ is a random sample of size $n - t + 1$ from $[n]$
without replacement. Also since $\max_{i \in [n]}\|\mvy_i\| \leq \zeta$, the operator norm of each $\frac{1}{\zeta^2} \mvy_i \mvy_i^{\t}$ is at most $1$. Therefore, we are exactly in the setting of
Proposition~\ref{prop:lambdaminconc} whereby we obtain for any $x \geq 0$,
\begin{equation*}
	\P\left[\big\|Y_t^{\t} Y_t - \frac{n - t + 1}{n} Y^{\t} Y\big\|_{{\rm op}} \geq \zeta^2 x\right] \leq 2d \exp\left(-\frac{n x^2}{2 (n - t + 1)(t - 1)}\right).
\end{equation*}
We plug $x = \frac{n - t + 1}{2n \zeta^2} \lambda_{\min}\big(Y^{\t}Y\big)$ into the above display to
obtain
\begin{equation*}%\label{eq:explicit_bnd_G2s}
	\P[\mathcal{G}_{2,t}^{c}] \leq 2d \exp\Big(-\frac{(n - t + 1) \lambda_{\min}^2\big(Y^{\t}Y\big)}{8 n(t - 1) \zeta^4}\Big)
\end{equation*}
which leads to~\eqref{eq:smallevent2} by substituting $\kappa$ for $\textstyle{\frac n{\lambda_{\min}(Y^\t Y)}}$.

Proof of the first bound in~\eqref{eq:Ctopbnd} has already been given in
\eqref{eq:uti}--\eqref{eq:weyl}. For the second bound, we can write
\begin{equation*}
	\begin{split}
		&\|Y_t \|_{\frob}^2 = {\rm Tr}(Y_t^{\t} Y_t) = \sum_{j \in [d]} \lambda_{j}(Y_t^{\t} Y_t) \\
		&\leq
		\frac{n - t + 1}{n} \sum_{j \in [d]} \lambda_{j}(Y^{\t} Y) + \frac{n - t + 1}{2n} \sum_{j \in [d]} \lambda_{\min}(Y^{\t} Y)\\
		&= \frac{3}{2} \frac{n - t + 1}{n}\cdot  {\rm Tr}(Y^{\t} Y) = \frac{3}{2} \frac{n - t + 1}{n} \|Y \|_{\frob}^2
	\end{split}
\end{equation*}
where $\lambda_j(\cdot)$ denotes the $j$-th largest eigenvalue of the corresponding (hermitian)
matrix and in the third step we used the Weyl's inequality (cf.~\eqref{eq:weyl}).
\end{proof}
We will in fact use this result for the proof of our next lemma.
\begin{lemma}\label{lem:smallevent1}
%	Recall the events for any $s \in [n]$,
%	\begin{equation*}
	%		\mathcal{G}_{1,s} \coloneqq \{\|\mv z_{s - 1}[\cA_{s}]\|_{\infty} < \varepsilon_n\}
	%	\end{equation*}
We have for any $t \in [n]$,
\begin{equation}\label{eq:smallevent1}
	\P\left[\Big(\bigcap_{s \in [t]} \cG_{1,s}\Big)^c \right] \leq \frac{\bE \left[ \max_{s \in [t],\, i \in
			\cA_s \setminus \{p_s\}}  |\mvz_s[i]|^2\right]}{\varepsilon_n^2}
\end{equation}%where $\dots = \E \max_{1 \leq t \leq n - m}$ and 
where
\begin{equation}\label{eq:maxmodzs}
	\bE \left[\max_{s \in [t]} \, \max_{i \in \cA_s \setminus \{p_s\}} |\mv z_s[i]|^2\right] \leq
	\frac{C\zeta^2\kappa^2}{n - t} + C \zeta^2 nd \exp\left(-c \,\frac{n(n - t + 1)}{\kappa^2\zeta^4t}\right).
\end{equation}
\end{lemma}
\begin{proof}
The bound~\eqref{eq:smallevent1} is just an application of the Markov's inequality in view of the
definition of $\cG_{1, s}$ in~\eqref{def:G1sG2s}. Towards proving~\eqref{eq:maxmodzs}, let us first recall that we have
$$\mv z_s = \sum_{j \in [s]} \delta_j \mv u_j$$
and in view of~\eqref{eq:identity1} we can write for any $i \in \cA_j \setminus \{p_j\}$,
$$\mvu_j[i] = - \frac{\mvy_i^{\t} D_j \mvy_{p_j}}{1 - \mvy_{p_j}^{\t} D_j \mvy_{p_j}}.$$
Therefore we have for any $i \in \cA_s \setminus \{p_s\}$,
\begin{equation*}
	\mvz_s[i] = - \mvy_i^{\t} \sum_{j \in [s]} \frac{D_j \mvy_{p_j}}{1 - \mvy_{p_j}^{\t} D_j \mvy_{p_j}} \,
	\delta_j.
\end{equation*}
Within this proof let us denote
\begin{equation*}
	\mvv_j = \frac{D_j \mvy_{p_j}}{1 - \mvy_{p_j}^{\t} D_j \mvy_{p_j}} \mbox{ and } \mvV_s = \sum_{j \in [s]} \delta_j \mvv_j.
\end{equation*}
Just like the martingales in Section~\ref{sec:martingales}, it is routine to check that $\{\mvV_t\}_{t =
	1}^{n}$ is a vector-valued martingale adapted to the filtration $(\cF_t)_{t = 0}^n$. Equipped
with the above notations we can write for any $i \in \cA_s\setminus \{p_s\}$,
\begin{equation*}
	|\mvz_s[i]| \leq \|\mvV_s\|\|y_i\|  \leq \zeta \|\mvV_s\|
\end{equation*}
where we used the Cauchy--Schwarz inequality and hence,
\begin{equation*}
	\max_{i \in \cA_s \setminus \{p_s\}} |\mvz_s[i]|^2 \leq \zeta^2 \|\mvV_s\|^2.
\end{equation*}
Since $\{\mvV_t\}_{t = 1}^{n}$ is a martingale, $\{\|\mvV_t\|^2\}_{t = 1}^{n}$ becomes a submartingale
sequence. Therefore, we can use the Doob's maximal inequality to deduce
\begin{equation}\label{eq:G1bd1}
	\bE \left[\max_{s \in [t],\, i \in \cA_s \setminus \{p_s\}} |\mvz_s[i]|^2\right] \leq \E \|\mvV_{t}\|^2.
\end{equation}
Further we can write,
\begin{equation*}
	\begin{split}
		\bE \left[\|\mvV_{t}\|^2\right] &= \sum_{s \in [t]} \E \left[\| \delta_s \mvv_s\|^2\right] \leq C \sum_{s \in [t]} \bE \left[\| \mvv_s\|^2\right] = C \sum_{s \in [t]} \bE \left[\frac{\mvy_{p_s}^{\t} D_s^2 \mvy_{p_s}}{(1 - \mvy_{p_s}^{\t} D_s \mvy_{p_s})^2}\right]\\ &\stackrel{\eqref{eq:identity2}}{\leq} C \sum_{s \in  [t]}\E \left[\mvy_{p_s}^{\t} D_s^2
		\mvy_{p_s}\right]%\\&\le C \sum_{t = 1}^{n - m} \E \|C_t\|_{{\rm op}}^2\|\mvy_{p_t}\|^2 (\mathrm{1} (\mathcal{G}_{2,j}) + \mathrm{1} (\mathcal{G}_{2,j}^{c})) \le C\sum_{t = 1}^{n - m} \frac{n^2}{(n - t + 1)^2\lambda_{{\rm min}}^2(Y^{\t}Y)} \, \frac{1}{n - t + 1}\E \sum_{i \in \cA_t}\|\mvy_{i}\|^2 + \cG_{2, t}^c \mbox{ terms}\\
		%&\color{red}{\le  C \frac{\kappa^2}{n}\|Y\|_{\frob}^2\sum_{t = 1}^{n - m} \frac{1}{(n - t + 1)^2} + \cG_{2, t}^c \mbox{ terms}} = \color{red}{\le  C \frac{\kappa}{m} + \mathcal 
		%	G_{2, t}^c \mbox{ terms}}
	%= \\& (1 + 2 \xi^4) \sum_{j = 1}^{n - m} \E \frac{d}{(n - j + 1) (1 + \lambda_{\min} \big(Y_j^{T} Y_j))}.
\end{split}
\end{equation*}
where in the second step we used the fact that $|\delta_s| \le 2$ almost surely. We will now bound the
expectations on the right by first partitioning into the events $\cG_{2, s}$ and $\cG_{2, s}^c$
and then bounding the corresponding terms separately.
\begin{equation*}
\begin{split}
	\sum_{s \in [t]}&\bE \left[\mvy_{p_s}^{\t} D_s^2
	\mvy_{p_s}\mathrm{1} _{\cG_{2, s}}\right] \le \sum_{s \in [t]}\bE \left[\|D_s\|_{{\rm op}}^2\|\mvy_{p_s}\|^2 \mathrm{1} _{\cG_{2, s}}\right] %\sum_{s \in [t]}\bE \left[\frac{\|\mvy_{p_s}\|^2}{\big(1 + \lambda_{{\rm min}}(Y^{\t}_s Y_s)\big)^2} \, \mathrm{1} _{\cG_{2, s}}\right]
	\stackrel{\eqref{eq:Ctopbnd}}{\leq} C\sum_{s  \in [t]} \frac{\kappa^2}{(n - s + 1)^2} \, \bE \left[\|\mvy_{p_s}\|^2\right]\\
	%\stackrel{\eqref{def:G1sG2s}+\eqref	{eq:weyl}}{\leq}& C\sum_{s  \in [t]} \frac{\kappa^2}{(n - s + 1)^2} \, \bE \left[\|\mvy_{p_s}\|^2\right] \, 
	\stackrel{{\rm Obs.}~\ref{obs:swor}}{=}& \,\, C\kappa^2 \sum_{s  \in [t]} \frac{1}{(n - s + 1)^2} \,\frac1n \sum_{i \in [n]} \|\mvy_i\|^2 \stackrel{\eqref{def:zeta}}{\le} C\zeta^2\kappa^2\sum_{s \in [t]} \frac1{(n - s + 1)^2} \, \le \, \frac{C\zeta^2\kappa^2}{n - t}.
	%& C \frac{\kappa^2}{n}\|Y\|_{\frob}^2 \sum_{t = 1}^{n - m} \frac{1}{(n - t + 1)^2} = C \frac{\kappa^2}{m}
\end{split}
\end{equation*}
On the other hand, we can write%Also, we have %
\begin{equation*}
\begin{split}
	\sum_{s \in [t]}&\bE \left[\mvy_{p_s}^{\t} D_s^2
	\mvy_{p_s}\mathrm{1} _{\cG_{2, s}^c}\right] \leq \zeta^2\sum_{s \in [t]} \P\left[\mathcal{G}_{2,s}^{c}\right] \stackrel{\eqref{eq:smallevent2}}{\leq} C \zeta^2 nd \exp\left(-c
	\,\frac{n(n - t + 1)}{\kappa^2\zeta^4t}\right)
	%\\& C d \sum_{t = 1}^{n - m} \exp\big(-\frac{(n - t + 1) \lambda_{\min}^2\big(Y^{T}Y\big)}{c n(t - 1)}\big) \leq C nd \exp\big(-\frac{(m + 1) \lambda_{\min}^2\big(Y^{T}Y\big)}{c n^2}\big) \leqC nd \exp\big(-\frac{(m + 1) }{c \kappa^2}\big) 
\end{split}
\end{equation*}
where in the first inequality we used the fact that
\begin{equation*}%\label{eq:bndypD2yp}
\mvy_p^{\t}D_{s}^2\mvy_p \le  \|D_s\|_{{\rm op}}^2 \|\mvy_p^{\t}\|^2  \stackrel{\eqref{def:zeta} +~\eqref{def:YD}}{\le} \zeta^2.
\end{equation*}
Adding the previous two bounds yields us~\eqref{eq:maxmodzs}.
\begin{comment}
Furthermore, we can bound
\begin{align*}
	& \E \lambda_{\max} \big(C_j\big) \leq \E \mathrm{1} (\mathcal{G}_{2,j}^{c})
	+ \E \frac{1}{1 + \lambda_{\min} \big(Y_j^{T} Y_j\big)} \mathrm{1}(\mathcal{G}_{2,j}) \leq \\&
	2d \exp\big(-\frac{(n - j + 1) \lambda_{\min}^2\big(Y^{T}Y\big)}{c n(j - 1)}\big) + \frac{2n}{(n - j + 1) \lambda_{\min}\big(Y^T Y\big)}
\end{align*}

\begin{align*}
	& \E \|V_{n - m}\|^2 \leq (1 + 2 \xi^4) \sum_{j = 1}^{n - m} \frac{2d^2}{n - j + 1}  \exp\big(-\frac{(n - j + 1) \lambda_{\min}^2\big(Y^{T}Y\big)}{8 n(j - 1)}\big) +                                                             \\&
	(1 + 2 \xi^4) \sum_{j = 1}^{n - m} \frac{2nd}{(n - j + 1)^2  \lambda_{\min}\big(Y^T Y\big)} \leq (1 + 2 \xi^4) 2d^2 \exp\big(-\frac{(m + 1) \lambda_{\min}^2\big(Y^{T}Y\big)}{8 n^2}\big) \sum_{j = 1}^{n - m} \frac{1}{n - j + 1} \\& + (1 + 2 \xi^4)  \frac{2nd}{  \lambda_{\min}\big(Y^T Y\big)} \sum_{j = 1}^{n - m} \frac{1}{(n - j + 1)^2} \leq \\&
	(1 + 2 \xi^4) 2d^2 \exp\big(-\frac{(m + 1) \lambda_{\min}^2\big(Y^{T}Y\big)}{8 n^2}\big) \log n  + (1 + 2 \xi^4)  \frac{2nd}{m \lambda_{\min}\big(Y^T Y\big)}.
\end{align*}
\end{comment}
\end{proof}

\begin{comment}
\begin{lemma}[Approximation of $\widetilde M_n$ by $M_n$]
\label{lem:MngsMn}
$|\widetilde M_n - M_n| =  1 + o_p(1)$ as $n \to \infty$ provided \dots
\end{lemma}
\begin{proof}
Recalling the definitions of $M_n$ and $\widetilde M_n$, we can write
\begin{equation*}
	\widetilde M_n - M_n = \sum_{t = 1}^n \delta_t \left(1 - \|B\mv u_t\|^{-1}\right)\left\langle B\mv u_t, \mat{\mv v}{0} \right\rangle
\end{equation*}
Since both $M_n$ and $\widetilde M_n$ are $(\cF_t)$-martingales with mean $0$, so is
$\widetilde M_n - M_n$. Therefore,
\begin{equation*}
	\begin{split}
		\E[(\widetilde M_n - M_n)^2] = \sum_{t = 1}^n &\E \left[\delta_t^2 \left(1 - \|B\mv u_t\|^{-1}\right)^2\left\langle B\mv u_t, \mat{\mv v}{0} \right\rangle^2\right]\\
		\le \sum_{t = 1}^n &\E \left[\left(1 - \|B\mv u_t\|^{-2}\right)^2\left\langle B\mv u_t, \mat{\mv v}{0} \right\rangle^2\right] \\
		\stackrel{\eqref{eq:Q_t}}{\le} \sum_{t = 1}^n &\E \left[(\|B\mv u_t\|^{2} - 1)^2 \cQ_t\right]  \stackrel{\eqref{eq:Butnormbnd}}{\le} (1 + \xi^2)^2\sum_{t = 1}^n  \E \left[(\|B\mv u_t\|^{2} - 1) \mathcal
		Q_t\right].
		%\stackrel{\eqref{eq:norm_butsq2} +~\eqref{eq:ypCtbnd}}{\le} &\xi^4 \, \sum_{t = 1}^n \E \left[\left\langle B\mv u_t, \mat{\mv v}{0} \right\rangle^2\right]\\ %(\because\text{Cauchy-Schwarz and }  \|v\| = 1)\\
		%&\stackrel{\eqref{eq:Butnormbnd}}{\le} (1 + \xi^2 + \xi^4)^2 \mvy_p^{\t}C_{t}\mvy_p
	\end{split}
\end{equation*}
However, the latter is bounded by a {\color{red}previous computation.}
\end{proof}
\end{comment}

%\input{reduction_to_ideal_proc}
\section{Equivalence of CLT for different processes}\label{sec:reduction_ideal}
The main result of this section is the equivalence of CLT between $M_n^{{\gsw}}$ and $M_n$ (see
Section~\ref{sec:martingales} for the definitions of all the relevant martingales).
\begin{proposition}\label{prop:main1}
Under the same assumptions as in Theorem~\ref{thm:formal}, we have
\begin{equation}\label{eq:main1}
\frac{M^{{\gsw}}_n}{\|\mvv\|} \xrightarrow[n \to \infty]{{\rm law}} N(0,1) \,\mbox{ if and only if }\,  \frac{M_n}{\|\mvv\|} \xrightarrow[n \to \infty]{{\rm law}} N(0,1).
\end{equation}
\end{proposition}%{\color{red} Define the $\kappa$ notation..measures inverse non singularity}
We need some intermediate results capturing the closeness between different related random
variables in order to prove Proposition~\ref{prop:main1}. Our first result gives an upper bound on the
probability that $M^{{\gsw}}_{n - m}$ and $\tilde{M}_{n - m}$ are different. %\begin{proof}[Proof of Proposition~\ref{prop:main1}]
%The proposition will be proved if the following two intermediate lemmas are shown.
\begin{lemma}\label{lem:proplem1.1}
For any $m \in [n]$ satisfying $m  \ge 6\Cr{C:xi}\zeta^2\kappa$, we have
\begin{equation}\label{eq:proplem1.1}
\P\left[|M^{{\gsw}}_{n - m} - \tilde{M}_{n - m}| > 0\right] \leq  \frac{C\zeta^2\kappa^2}{m} + C \, nd
\exp\left(-c (\kappa\zeta^2)^{-2}m\right).
\end{equation}%\begin{equation*}
%	|M^{{\gsw}}_n - \tilde{M}_n| \stackrel{\rm p}{\rightarrow} 0
%\end{equation*}
\end{lemma}
\begin{proof}
We know from Lemma~\ref{lem:pivotlemma} that $M^{{\gsw}}_{n - m}$ and $M_{n - m}$ are
identical on the event $\mathcal E_{1, n-m} \cap \mathcal E_{2, n-m}$ {\em provided} $n - m < n + 1 - 6\Cr{C:xi}\zeta^2\kappa$, \ie  $m \ge 6\Cr{C:xi}\zeta^2\kappa$. Hence we only need to show that the
probability of the event $(\mathcal E_{1, n-m} \cap \cE_{2, n-m})^c$ is bounded by the right-hand side
in~\eqref{eq:proplem1.1}. But this follows from Lemmas~\ref{lem:smallevent2} and
\ref{lem:smallevent1} along with a union bound for bounding the probability of the event $\cE_{2, t}^c$.
%Let us denote
%\begin{equation*}
%T_{m, n} = |M^{{\gsw}}_{n - m} - \tilde{M}_{n - m}| \mbox{ and } \cE_{n - m} = \cE_{1, n - m} \cap \cE_{2, n - m}
%\end{equation*}%The exact value of the sequence $m_n$ will be chosen later. %It is now enough to show that the three random sequences $T_{1_n},T_{2,n},T_{3,n}$ converge to $0$ in probability. %Let us denote the event 
%where the events $\mathcal E_{1, n - m}$ and $\mathcal E_{2, n - m}$ are from~\eqref{def:E1E2}. %	\begin{equation*}
%%		%\mathcal{E}_n = \cap_{s = 1}^{n - m_n} \mathcal{G}_{1,s}  \cap \cap_{s = 1}^{n - m_n}. \mathcal{G}_{2,s}
%%		\mathcal{E}_n = \cap_{s = 1}^{n - m_n} \mathcal{G}_{1,s}.
%%	\end{equation*} {\color{red} Is it $m_n$ or $m_n - 1$ above?} %Let's first work with $T_{2,n}.$ 
%We can write 
%\begin{equation*}
%T_{m,n} = T_{m,n} 1_{\cE_{1, n - m} \cap \cE_{2, n - m}} + T_{m, n} \mathrm{1}(\mathcal{E}_n)^{c} = T_{1,n} 
%\mathrm{1}(\mathcal{E}_n)^{c}
%\end{equation*}
%	
%	where the last equality follows because on the event $\mathcal{E}_n $, we have $M^{{\gsw}}_{n - m_n} = \tilde{M}_{n - m}$ by  as long as 
%	
%	\begin{equation*}
%		n - m < n + 1 - 6 \xi^2 C_{\xi} \kappa.
%	\end{equation*}
%	
%	
%	
%	Therefore, $P(T_{1,n} > 0)$ is bounded above by if $P(\mathcal{E}_n)^{c}$ which in turn can be bounded as in Lemma~\ref{lem:smallevent1}.
%	
%	
%	
%	
%	
\end{proof}
Our next lemma gives an upper bound on the $\ell^2$-distance between $M^{{\gsw}}_n$ and
$M^{{\gsw}}_{n - m}$ as well as between $\tilde M_n$ and $\tilde M_{n - m}$.
\begin{lemma}\label{lem:proplem1.2}
For any $m \in [n]$, we have
\begin{equation}\label{eq:proplem1.2}
\max\left(\bE \left[(M^{{\gsw}}_n - M^{{\gsw}}_{n - m})^2\right],\:  \bE
\left[(\widetilde{M}_n - \widetilde{M}_{n - m})^2\right]\right) \leq C d \max(1, \zeta^2) \|\mvv\|_{\infty}^2  m^3.
\end{equation}%\begin{equation*}%	|M^{{\gsw}}_n - \tilde{M}_n| \stackrel{\rm p}{\rightarrow} 0%\end{equation*}
\end{lemma}
\begin{proof}
Let us examine the difference
\begin{equation}
\Delta_{m, n} \coloneqq M^{{\gsw}}_n - M^{{\gsw}}_{n - m}.
\end{equation}%To show that $T_{1,n} \stackrel{\rm p}{\rightarrow} 0$ it is enough to show that $\bE T_{1,n}^2 \rightarrow 0.$
In view of the first item under \nameref{label:case2} in our definition of the relevant processes, we can
write
\begin{equation}\label{def:Deltamn}
\begin{split}
	\Delta_{m, n} = \sum_{t = n - m + 1}^{n} \delta^{{\gsw}}_t \left \langle B \mv u^{{\gsw}}_t,
	\mat{\mv v}{0} \right\rangle \1_{\{\tau > t\}}
\end{split}
\end{equation}
with $$\tau \coloneqq \min\{t > n - m : \cA^{{\gsw}}_t = \emptyset\}$$
where we adopt the convention that the minimum of an empty set is $n$, \ie  the limit of the above
summation is empty. Henceforth, in this proof we will drop the superscript ``${\gsw}$'' in order to
avoid the notational clutter. This is particularly appropriate because our argument relies only on
properties shared by $\{\tilde M_t^{{\gsw}}\}_{t = 0}^n$ and $\{\tilde M_t\}_{t = 0}^n$ and hence works {\em mutatis mutandis} for $\{\tilde M_t\}_{t = 0}^n$. %	\begin{align*}
%		T_{2,n} = \sum_{t = n - m + 1}^{\tau} \delta_t \left \langle B \mv u_t, \mat{\mv v}{0} \right\rangle \
%	\end{align*}
%This is because in the following argument whatever properties of the original process we use 
%also will hold for the ideal process. That is, we do not use any particular property of the ideal process 
%that does not hold for the original process. Hence, there should not be any confusion arising due to us dropping the ${\gsw}$ superscript.
%{\color{red} what is original and what is ideal? Is that clear by now?}

\vspace{0.2cm}

Since the event $\{\tau > t\}$ is $\cF_{t - 1}$ measurable~\eqref{def:F_t}, it follows from the definitions of the
terms involved (see Section~\ref{sec:stoch_proc}) that the partial sums of $\Delta_{m, n}$ form a
martingale sequence relative to $(\cF_t)_{t = 0}^n$. Hence, in view of orthogonality of
martingale differences, we can write:
\begin{equation}\label{eq:Deltamnbnd}
\begin{split}
	&\bE [\Delta_{m,n}^2] \\
	&= \sum_{t = n - m + 1}^{n} \bE\left[ \delta_t^2 \left \langle B \mv u_t, \mat{\mv v}{0} \right\rangle^2 \1_{\{\tau > t\}}\right] \leq C \sum_{t = n - m + 1}^{n} \bE \left[\left \langle B \mv u_t, \mat{\mv v}{0} \right\rangle^2 \1_{\{\tau > t\}}\right]^2 \\
	&\stackrel{\eqref{eq:identity3}}{=}  C \sum_{t = n - m + 1}^{n} \bE \big[\norm{B \mv u_t}^{4} \left(\mvv[\mathcal \cA_t]^{\t}(I_{a_t} - Y_t D_t Y_t^{\t}) \mve_{p_t}[\cA_t] \right)^2\1_{\{\tau >
		t\}}\big]\\ &\stackrel{\eqref{eq:identity2}}{\leq} C  \sum_{t = n - m + 1}^{n} \bE
	\left[\left(\mvv[\mathcal \cA_t]^{\t}(I_{a_t} - Y_t D_t Y_t^{\t}) \mve_{p_t}[\cA_t]\right)^2 \1_{\{\tau >
		t\}}\right]
\end{split}
\end{equation}
where in the second step we used the fact that $|\delta_t| \leq 2$ almost surely and $a_t =
|\mathcal \cA_t|$. %Also throughout this proof, we adopt the convention that if $\tau = n - m$ then the sum from $t = n - m + 1$ to $n - m$ is $0$. {\color{red} why? also explain the  notation ep[a], is this tau convention ok?}
Here we used the event $\{\tau > t\}$ to ensure that $\cA_t$ is non-empty so that
\eqref{eq:identity2} and~\eqref{eq:identity3} apply. We will frequently use the following elementary
inequality on quadratic forms. Let $\mvx \in \R^n$, $A, B \in R^{n\times n}$ be symmetric and $C \in
\R^{n \times n}$ be an n.n.d. matrix. Then we have,
\begin{equation}\label{eq:qformineq}
\mvx^\t(A - B)C(A - B)\mvx \le 2(\mvx^\t AC A\mvx + \mvx^\t BC B\mvx).
\end{equation}	We can use this to bound the right-hand side of~\eqref{eq:Deltamnbnd} by writing,
\begin{comment}
\begin{align*}
	& \sum_{t = n - l_n + 1}^{\tau} \bE \mvv[\mathcal \cA^{{\gsw}}_t]^{\t}(I_{a^{{\gsw}}_t} - Y^{{\gsw}}_t C^{{\gsw}}_t (Y^{{\gsw}}_t)^{\t})\mve_{p^{{\gsw}}_t}[\cA^{{\gsw}}_t] \mve_{p^{{\gsw}}_t} [\cA^{{\gsw}}_t]^{T} (I_{a^{{\gsw}}_t} - Y^{{\gsw}}_t C^{{\gsw}}_t (Y^{{\gsw}}_t)^{\t}) \mvv[\mathcal \cA^{{\gsw}}_t] \leq \\&
	2 \sum_{t = n - l_n + 1}^{\tau} \bE \mvv[\mathcal \cA^{{\gsw}}_t]^{\t}(I_{a^{{\gsw}}_t}) \mve_{p^{{\gsw}}_t}[\cA^{{\gsw}}_t] \mve_{p^{{\gsw}}_t} [\cA^{{\gsw}}_t]^{T} (I_{a^{{\gsw}}_t}) \mvv[\mathcal \cA^{{\gsw}}_t] +                                                                                                                \\& 2 \sum_{t = n - l_n + 1}^{\tau} \bE \mvv[\mathcal \cA^{{\gsw}}_t]^{\t}(Y^{{\gsw}}_t C^{{\gsw}}_t (Y^{{\gsw}}_t)^{\t}) \mve_{p^{{\gsw}}_t}[\cA^{{\gsw}}_t] \mve_{p^{{\gsw}}_t} [\cA^{{\gsw}}_t]^{T} (Y^{{\gsw}}_t C^{{\gsw}}_t (Y^{{\gsw}}_t)^{\t}) \mvv[\mathcal \cA^{{\gsw}}_t].
\end{align*}
\end{comment}
\begin{equation}\label{eq:Deltasqbreakup}
\begin{split}
	&\sum_{t = n - m + 1}^{n} \bE \left[\mvv[\mathcal \cA_t]^{\t}(I_{a_t} - Y_t D_t Y_t^{\t})\mve_{p_t}[\cA_t] \mve_{p_t} [\cA_t]^{\t} (I_{a_t} - Y_t D_t Y_t^{\t}) \mvv[\mathcal \cA_t] \1_{\{\tau > t\}} \right] \\ \, \leq & \:\:\:\:\: 2 \sum_{t = n - m + 1}^{n} \bE \left[\mvv[\mathcal \cA_t]^{\t} \mve_{p_t}[\cA_t] \mve_{p_t} [\cA_t]^{\t} \mvv[\mathcal \cA_t] \1_{\{\tau > t\}}\right] \\
	&+  2 \sum_{t = n - m + 1}^{n} \bE
	\left[\mvv[\mathcal \cA_t]^{\t}Y_t D_t Y_t^{\t} \mve_{p_t}[\cA_t] \mve_{p_t} [\cA_t]^{\t} Y_t D_t
	Y_t^{\t} \mvv[\mathcal \cA_t]\1_{\{\tau > t\}}\right].
\end{split}
\end{equation}
Now we will separately bound the two terms on the right-hand side of the above display. For the
second term, we can bound its $t$-th summand as follows (with $\1_{\{\tau > t\}}$ kept implicit and
$p_t$ denoted by $p$ to reduce the notational clutter):
\begin{equation*}
\begin{split}
	&\bE\left[ \mvv[\cA_t]^{\t}Y_t D_t Y_t^{\t} Y_t D_t Y_t^{\t}\mvv[\cA_t] \right]\\
	&= \bE\left[\mvv[\cA_t]^{\t}Y_t D_t \left(D_t^{-1} - I_d\right) D_t Y_t^{\t} \mvv[\cA_t] \right] \leq  \bE \left[\mvv[\cA_t]^{\t}Y_t D_t Y_t^{\t} \mvv[\cA_t]\right] \\
	&\leq \bE\left[ \mvv[\cA_t]^{\t}Y_t Y_t^{\t} \mvv[\cA_t] \right] = \sum_{j \in [d]} \bE \left[\left(\mvv[\mathcal \cA_t]^{\t} Y_t[:\,, j] \right)^2\right] \\
	&\leq \sum_{j \in [d]} \bE  \left[\|\mvv[\mathcal \cA_t]\|^2 \|Y_t[:\,, j]\|^2\right] \leq \zeta^2 \|\mvv\|_{\infty}^2 \sum_{j \in [d]} \bE \left[|\cA_t|^2\right] \leq \zeta^2 d \|\mvv\|_{\infty}^2 m^2.
\end{split}
\end{equation*}
Here in the very first term we used the fact that the matrix $\mve_p[\cA_t] \mve_p[\cA_t]^{\t}$ is
bounded by the identity matrix $I_{a_t}$ in the {\em Loewner order}. Similarly in the first inequality we used that $D_t - D_t^2$ is at most $D_t$ in the Loewner order. In the third inequality, we used the
bound $\|D_t\|_{{\rm op}} \leq 1$~\eqref{def:YD}. The fourth inequality uses the Cauchy-Schwarz
inequality whereas in the last inequality we used the fact that $|\cA_t| \leq n - t + 1$ for any $t \in [n]$.
Summing this from $t = n - m + 1$ to $n$ yields the bound $C \|v\|_{\infty}^2 d m^3.$

\vspace{0.2cm}

As for the first term in~\eqref{eq:Deltasqbreakup}, we can similarly write
\begin{equation*}
\begin{split}
	&\sum_{t = n - m + 1}^{n} \bE \left[\mvv[\cA_t]^{\t} \mve_p[\cA_t] \mve_p[\cA_t]^{\t} \mvv[\mathcal \cA_t] \1_{\{\tau > t\}} \right]  \leq  \sum_{t = n - m + 1}^{n} \bE \left[\mvv[\cA_t]^{\t}  \mvv[\cA_t]\right] \\
	\leq &\, \|v\|_{\infty}^2 \sum_{t = n - m + 1}^{n} \bE |\cA_t| \leq \|v\|_{\infty}^2 m^2.
\end{split}
\end{equation*}%{\color{red} Again..add the indicator terms..where is the 1 by cardinality and the indicator terms...add them}
Therefore combining the bounds from the previous two displays and plugging it into
\eqref{eq:Deltasqbreakup}  and subsequently into~\eqref{eq:Deltamnbnd}, we obtain
\begin{equation*}
\bE \left[\Delta_{m,n}^2\right] \leq C d \max(1, \zeta^2) \|\mvv\|_{\infty}^2  m^3.
\end{equation*}
This gives us the required bound in~\eqref{eq:proplem1.2} for both $\{M_t\}_{t = 0}^n$ and $\{\tilde
M_t\}_{t = 0}^n$ (see the discussion following the definition of $\Delta_{m, n}$ in~\eqref{def:Deltamn}).
\end{proof}
Finally we need the following result on the $\ell^2$-distance between $M_n$ and $\tilde M_n$.
\begin{lemma}\label{lem:proplem2}
For any integer $m \in {[n]}$ we have the following bound,
\begin{equation}\label{eq:intclosetoideal}
\begin{split}
	%\bE[(\widetilde M_n - M_n)^2] &\leq C(\zeta) \left((d + \kappa) \|\mv v\|_{\infty}^2 \log en + \frac{d n}{\varepsilon_n}\, \exp\left(-c \frac{m}{\kappa^2}\right) +   \|\mv v\|_\infty^2 m + \frac{\kappa^2}{m \varepsilon_n} + \varepsilon_n \right)
	%\bE[(\widetilde M_n - M_n)^2] &\leq C(\zeta) \left((d + \kappa) \|\mv v\|_{\infty}^2 \log en + \frac{d n}{\varepsilon_n}\, \exp\left(-c \frac{m}{\kappa^2}\right) +   \|\mv v\|_\infty^2 m + \frac{\kappa^2}{m \varepsilon_n} + \varepsilon_n \right)
	\bE&[(\widetilde M_n - M_n)^2] \\
	&\leq  C\cdot\Cr{C:xi}^3 \left(\left(m + (d + \kappa)\log en\right) \|\mv v\|_{\infty}^2   + C d  n \zeta^2 \exp\left(-c \frac{m}{\kappa^2 \zeta^4}\right) \right) + C\varepsilon_n^2 \|\mvv\|^2.
\end{split}
\end{equation}%\begin{equation*}%|\tilde{M}_n - M_n| \stackrel{\rm p}{\rightarrow} 0 %\end{equation*}
%where $C(\zeta) = C \cdot \Cr{C:xi}$.
\end{lemma}%{\color{red} any reason we take $l_n \leq n/2$ or can we replace it by $n$?}%Let us first show Lemma~\ref{lem:proplem2}.
\begin{proof}%[Proof of Lemma~\ref{lem:proplem2}]	
Recalling the definitions of $M_t$ and $\widetilde M_t$ from~\eqref{def:MttildeMt}, we can write
\begin{align*}
&\tilde M_t - M_t \\
&\quad= \underbrace{\sum_{s \in [t]} \delta_s \left(1 - \|B\mv u_s\|^{-1}\right)\left\langle B\mv u_s, \mat{\mv v}{0} \right\rangle}_{A_t} + \underbrace{\sum_{s \in [t]} (\delta_s - \eta_s) \|B\mv
	u_s\|^{-1} \left\langle B\mv u_s, \mat{\mv v}{0} \right\rangle}_{B_t}.
\end{align*}
%\begin{equation*}
%\tilde M_t - M_t = \underbrace{\sum_{s \in [t]} \delta_s \left(1 - \|B\mv u_s\|^{-1}\right)\left\langle B\mv u_s, \mat{\mv v}{0} \right\rangle}_{A_t} + \underbrace{\sum_{s \in [t]} (\delta_s - \eta_s) \|B\mv 
%u_s\|^{-1} \left\langle B\mv u_s, \mat{\mv v}{0} \right\rangle}_{B_t}.
%\end{equation*}

Consequently, \begin{equation*}
\bE[(\widetilde M_n - M_n)^2] \leq 2 \bE \left[A_n^2\right] + 2 \bE \left[B_n^2\right].
\end{equation*}
We now observe that both $A_t$ and $B_t$ are $(\cF_t)_{t = 0}^n$-martingales with mean
$0$. Hence by the orthogonality of martingale differences, we can write
\begin{equation*}
\begin{split}
	\bE \left[A_n^2\right] = \sum_{t \in [n]} &\bE \left[\delta_t^2 \left(1 - \|B\mv
	u_t\|^{-1}\right)^2\left\langle B\mv u_t, \mat{\mv v}{0} \right\rangle^2\right]\\
	\le C \sum_{t \in [n]} &\bE \left[\left(1 - \|B\mv u_t\|^{-2}\right)^2\left\langle B\mv u_t, \mat{\mv v}{0} \right\rangle^2\right] \\
	= C \sum_{t \in [n]} &\bE \left[\left( \|B\mv u_t\|^{2} - 1\right)^2 \frac{1}{\|B\mv u_t\|^{4}} \left\langle B\mv u_t, \mat{\mv v}{0} \right\rangle^2\right] \\
	\stackrel{\eqref{eq:identity5}}{=} C \sum_{t \in [n]} &\bE \left[(\|B\mv u_t\|^{2} - 1)^2 \cQ_t\right]
	\stackrel{\eqref{eq:identity2}}{\le} C \cdot \Cr{C:xi}\sum_{t \in [n]}  \bE \left[(\|B\mv u_t\|^{2} - 1) \cQ_t\right]
\end{split}
\end{equation*}
where in the first inequality we used the facts that $|\delta_t| \leq 2$ almost surely and that $\|B\mv u_t\| \geq 1$ which we noted after~\eqref{eq:identity2}.

\smallskip

We now proceed to bound the expectations in the last line above.  In the sequel, we will use the
notation $a_t = n - t + 1 = |\cA_t|$ (recall Observation~\ref{obs:swor}) and also use $p$ instead
of $p_t$.  %Let us start with 
%$\norm{B\mvu_{t}}^{2} - 1$.
%\begin{equation*}
%\begin{split}
%	(\norm{B\mvu_{t}}^{2} - 1) \stackrel{\eqref{eq:norm_butsq2}}{=}&
%	\mvy_p^{\t}C_{t}\mvy_p\norm{B\mvu_{t}}^{2} \stackrel{\eqref{eq:norm_butsq}}{=} 
%	\mvy_p^{\t}C_{t}\mvy_p(1 + \mvy_p^{\t} C_{t}^- \mvy_p) \le \mvy_p^{\t}C_{t}\mvy_p(1 + \|D_t^-\|_{{\rm 
	%			op}} \|\mvy_p\|^2)\\
%	\stackrel{\eqref{def:XX_tBC}}{\le}& 2 \,\mvy_p^{\t}C_{t}\mvy_p
%\end{split}
%\end{equation*}
%where in the last step we also used the fact that $\max_{i \in [n]} \|y_i\| \leq 1$. In fact, by the same 
%reasoning as employed in the last line, we also have
%\begin{equation}\label{eq:ypCtbnd}
%\mvy_p^{\t}C_{t}\mvy_p \le 1.
%\end{equation}
For any $t \in [n]$, we obtain from~\eqref{eq:bnd_cond_expect1} proved in Section~\ref{sec:ideal} below that
\begin{equation*}
\bE \left[(\norm{B\mvu_{t}}^{2} - 1)\cQ_t \, | \, \cF_{t-1}\right] \le \frac{2\Cr{C:xi}}{n - t + 1}\, \sum_{j \in \cA_t} v_j^2\mvy_j^{\t} D_t \mvy_j + \frac {2\Cr{C:xi}^2}{n - t + 1}\,\mvv[\cA_t]^{\t}Y_tD_tY_t^{\t}\mvv[\cA_t].
\end{equation*}
We bound the expectation of the quantity on the right-hand side above (summed over $t$ in $[n]$) in
a {\em delicate} bound~\eqref{eq:idealclt2} proved as part of Lemma~\ref{lem:idealclt} in the next
section which finally yields
\begin{equation}\label{eq:intclosetoideal1}
\bE \left[A_n\right]^2 \leq C\cdot\Cr{C:xi}^3 \left((m + \kappa\log en) \|\mv v\|_{\infty}^2  + d \log en \|\mvv\|_{\infty}^2  + C d  n \zeta^2 \exp(-c (\kappa \zeta^2)^{-2}m) \right).
\end{equation}%yet another calculation~ perfomed to finally obtain for any integer $m \in {[n]}$,
%\begin{equation*}
%\begin{split}	
%&\sum_{t \in [n]} \bE \left[(\norm{B\mvu_{t}}^{2} - 1)\cQ_t\right] \leq \\& d\|\mvv\|_{\infty}^2 \log en 
%+ \frac{C \|v\|_\infty^2 \cdot \log n}{\lambda_{{\rm min}}(Y^{\t} Y)} \|Y\|_{\frob}^2 + C d n\, \exp\left(-c \frac{\ell_n \lambda_{\min}^2\big(Y^{\t}Y\big)}{n^2}\right) +  \frac{\|v\|_\infty^2 \cdot 
%\|Y\|_{\frob}^2 \cdot \ell_n}{n}
%\end{split}
%\end{equation*} 
%
%
%
%By further using the fact that $\|Y\|_{\frob}^2 \leq n \xi^2$ we can simplify the above bound to obtain

Next we will bound $\bE \left[B_n^2\right]$. By the martingale property, we can write
\begin{equation*}
\begin{split}
	&\bE \left[B_n\right]^2 \\
	&= \sum_{t \in [n]} \bE \left[(\delta_t -\eta_t)^2 \|B\mv u_t\|^{-2} \left\langle B\mv u_t, \mat{\mv v}{0} \right\rangle^2\right] \le C\varepsilon_n^2 \sum_{t \in [n]} \bE \left[\eta_t^2 \|B\mv u_t\|^{-2} \left\langle B\mv u_t, \mat{\mv v}{0} \right\rangle^2\right]\\
	&\stackrel{\eqref{def:MttildeMt}}{=} C\varepsilon_n^2 \bE\left[M_n^2\right] \stackrel{\eqref{form:sigmansq}}{=} C\varepsilon_n^2 \|\mvv\|^2
\end{split}
\end{equation*}
where in the second step we used the definitions of $\delta_t$ and $\eta_t$ from~\eqref{def:delta_t},
\eqref{def:etat} and~\eqref{eq:updatezt} along with the bound given by~\eqref{depart:delta}. Combined
with~\eqref{eq:intclosetoideal1}, this yields~\eqref{eq:intclosetoideal}.
\end{proof}

We are now ready to prove Proposition~\ref{prop:main1}.
\begin{proof}[Proof of Proposition~\ref{prop:main1}]
It is enough to show that
\begin{equation}\label{eq:inprob2}
\frac{1}{\|\mvv\|}|\,M_n - \tilde{M}_n| \xrightarrow[n \to \infty]{{\rm P}} 0
\end{equation}
and
\begin{equation}\label{eq:inprob1}
\frac{1}{\|\mvv\|}\,|M^{{\gsw}}_n - \tilde{M}_n| \xrightarrow[n \to \infty]{{\rm P}} 0
\end{equation}
where ${\rm P}$ indicates {\em convergence in probability}. Recall that there is a parameter
$\varepsilon_n$ that goes in the definition of $\tilde{M}_n$. To this end, we choose $$\varepsilon_n = \min\left(\frac{1}{\sqrt{\log en}}, \frac13\right)$$ and proceed to show~\eqref{eq:inprob2}
and~\eqref{eq:inprob1}.

To show~\eqref{eq:inprob2} it suffices to show that
\begin{equation}\label{eq:MnMntildeell2}
\lim_{n \to \infty}\frac{1}{\|\mvv\|^2}\bE \left[(M_n - \tilde{M}_n)^2\right] = 0.
\end{equation}
For this, we can set $$m =  C \kappa^2 \zeta^4 \log en$$ for an appropriately large enough constant $C$ and use Lemma~\ref{lem:proplem2} to deduce~\eqref{eq:inprob1} provided
\begin{equation*}%\label{lim:conc}
\lim_{n \to \infty} \max(m^{1/2}, d^{1/2} (\log n)^{1/2}, \kappa^{1/2}(\log n)^{1/2}) \frac{\|\mvv\|_\infty}{\|\mvv\|} = 0
\end{equation*}
(recall that $d \le n$) which follows from the first part of assumption~\eqref{eq:formal_cond} in
Theorem~\ref{thm:formal}. We also need to show that $m$ as chosen above is at
most $n$ for all $n$ large enough. But this follows from assumption~\eqref{eq:formal_cond} in view of the facts that
$\tfrac{1}{\sqrt{n}} \le \tfrac{\|\mvv\|_{\infty}}{\|\mvv\|}$ and
\begin{equation}\label{bnd_kappa}
\kappa = \frac{n}{\lambda_{{\rm min}}(Y^{\t}Y)}	\ge \frac{dn}{\|Y\|_{\frob}^2} \stackrel{\eqref{def:zeta}}{\ge} \frac{dn}{n \zeta^2} = \zeta^{-2}
\end{equation}
(notice that $\zeta^{-1}$ is bounded~\eqref{def:zeta} since $\phi$ is bounded away from $0$ by our
assumption).

%\begin{align*}
%	&\bE[(\widetilde M_n - M_n)^2] \leq C \left((d + \kappa) \|\mv v\|_{\infty}^2 \log en + \frac{d n}{\epsilon_n}\, \exp\left(-c \frac{\ell_n}{\kappa^2}\right) +   \|\mv v\|_\infty^2 l_n + \frac{\kappa^2}{l_n \epsilon_n} + \epsilon_n \right) \rightarrow 0.
%	\end{align*}

In order to show~\eqref{eq:inprob1}, we again set $$m =  C \kappa^2 \zeta^4 \log en$$ for some suitably large constant $C$ and write
\begin{equation}\label{eq:M_ngsM_ntilde}
M^{{\gsw}}_n - \tilde{M}_n = \underbrace{M^{{\gsw}}_n - M^{{\gsw}}_{n - m}}_{T_1} + \underbrace{M^{{\gsw}}_{n - m} - \tilde{M}_{n - m}}_{T_2} + \underbrace{\tilde{M}_{n - m} - \tilde{M}_{n}}_{T_3}.
\end{equation}
We can use Lemma~\ref{lem:proplem1.2} and a similar reasoning as before to argue that both
$\|\mvv\|^{-2}\bE \left[T_1^2\right]$ and $\|\mvv\|^{-2}\bE \left[T_3^2\right]$ converge to $0$ under
assumption~\eqref{eq:formal_cond}. This implies that $T_1$ and $T_3$ converge to $0$ in probability. On the other hand, we have $m \ge 6\Cr{C:xi}\zeta^2\kappa$ for large enough $n$ in view
of~\eqref{bnd_kappa}. Hence we can use Lemma~\ref{lem:proplem1.1} to show that $\mathbb P[T_2 > 0] \rightarrow 0$ which implies that $T_2$ converges to $0$ in	probability as $n \to \infty$.

Together these yield~\eqref{eq:inprob1} in view of~\eqref{eq:M_ngsM_ntilde} completing the proof of Proposition~\ref{prop:main1}.
\end{proof}

\section{Asymptotic normality of $M_n$ and the proof of Theorem~\ref{thm:formal}}\label{sec:ideal}
In this section, we conclude the proof of Theorem~\ref{thm:formal} by combining
Proposition~\ref{prop:main1} from Section~\ref{sec:reduction_ideal} and the CLT for $M_n$ proved in
Proposition~\ref{prop:cltideal} below. The latter is the main result of this section.
%under certain conditions on the sequence $(\delta_t)_{t = 1}^{n}$ and the vector $v.$ 
\begin{proposition}\label{prop:cltideal}
Under the same assumptions as in Theorem~\ref{thm:formal}, we have
\begin{equation*}
\frac{M_n}{\|\mvv\|} \xrightarrow[n \to \infty]{{\rm law}} \N(0, 1).
\end{equation*}
\end{proposition}
%{\color{red} Rewrite the proposition with the negligible term I guess..}
The proof of Proposition~\ref{prop:cltideal} occupies the majority of the remainder of this article. But
before that let us finish the proof of Theorem~\ref{thm:formal} assuming this result.
\begin{proof}[Proof of Theorem~\ref{thm:formal}]
We obtain from Propositions~\ref{prop:main1} and~\ref{prop:cltideal} that,
\begin{equation}\label{eq:M_ngsCLT}
\frac{M_n^{{\gsw}}}{\|\mvv\|} \xrightarrow[n \to \infty]{{\rm law}} \N(0, 1)
\end{equation}
(recall that $\langle \mvz_n^{{\gsw}}, \mvv \rangle = M_n^{{\gsw}}$). However, since  %the second part of assumption~\eqref{eq:formal_cond} as well as the fact 
$\phi$ is bounded away from $0$ and the matrix $B$ in definition~\ref{def:B} and in
\cite{harshaw2019balancing} (see~\eqref{def:B_intro}) only differ by a factor $1/\sqrt{\phi}$, it follows
from \cite[Theorem~6.1]{harshaw2019balancing} that $\tfrac{M_n^{{\gsw}}}{\|\mvv\|}$ is a subgaussian variable with variance parameter $1/\phi$. In particular, the random variables $\tfrac{(M_n^{{\gsw}})^2}{\|\mvv\|^2}$ are {\em uniformly integrable} in $n$. Hence by classical 
results on uniform integrability, this implies together with~\eqref{eq:M_ngsCLT}:
\begin{equation*}
\lim_{n \to \infty} \frac{\var\left[M_n^{{\gsw}}\right]}{\|\mvv\|^2} = 1  \mbox{ and hence }
\frac{M_n^{{\gsw}}}{\sqrt{\var\left[M_n^{{\gsw}}\right]}}  \xrightarrow[n \to \infty]{{\rm law}} \N(0, 1)
\end{equation*}
which finishes the proof.
\end{proof}
We now return to the
\begin{proof}[Proof of Proposition~\ref{prop:cltideal}]
We will verify Lindeberg's conditions for the CLT of the  triangular array of martingales $\{M_t\}_{t
= 0}^n$ (recall that we keep the dependence on $n$ implicit). To this end, let us  define the {\em conditional quadratic variation} $Q_n$ of $M_n$ %is given by
as 
\begin{equation} \label{def:Qn}
%	{\rm Q}_n = \sum_{1 \le t \le n} \bE \left [ (\delta_t)^2 \left\langle \frac{B\mv u_t}{\|B \mvu_t\|}, 
%	\mat{\mv v}{0} \right\rangle^2  \,\Big\vert  \, \cF_{t-1} \right ].
Q_n = \sum_{t \in [n]} \bE [\Delta_t^2\mid \cF_{t - 1}] \mbox{ where } \Delta_t \coloneqq M_t - M_{t-1} = \eta_t \left\langle \frac{B\mv u_t}{\|B \mvu_t\|}, \mat{\mv v}{0} \right\rangle
%	\mat{\mv v}{0} \right\rangle^2  \,\Big\vert  \, \cF_{t-1} \right ].	
\end{equation}
(recall~\eqref{def:MttildeMt} as well the recursive definition of the $\sigma$-algebra $\cF_t$ from~\eqref{def:F_t} in Section~\ref{sec:stoch_proc}). Notice that $\bE [Q_n] =
\var[M_n] \eqqcolon \sigma_n^2$. Lindeberg's conditions are implied by  Lyapunov's moment
condition which in our case are summarized below (see, e.g., \cite{HallHeyde80book}):
\begin{enumerate}[label=(\roman*)]
\item $\frac{\sigma_n}{\|\mvv\|} \rightarrow 1$,
%		\item for any $\delta > 0$,
%		\begin{equation*}
%			\sum_{1 \le t \le n} \bE [ \Delta_t^2 1_{\{|\Delta_t| \ge \delta|\}}] \to 0 \mbox{ as } n \to \infty	\quad {and}       	
%		\end{equation*}
%	where $\Delta_t \coloneqq M_{t} - M_{t-1}$ and
\item $\frac{1}{\|\mvv\|^4}\sum_{t \in [n]} \bE [ \Delta_t^4] \rightarrow 0,
\mbox{ and}$

\item $\frac{Q_n}{\|\mvv\|^2} \rightarrow 1 \mbox{ in probability}$
\end{enumerate}
as $n \to \infty$.
We now verify each of these three conditions.

\smallskip

\noindent {\em Verifying condition~(i).}
%We note that by the uncorrelatedness of martingale differences, 
%\begin{equation*}
%\sigma_n^2 = \bE \left[M_n^2\right] = %\bE (\sum_{t = 1}^{n} \Delta_t)^2 = 
%\sum_{t \in [n]} \bE \left[\Delta_t^2\right].
%\end{equation*}	
The following observation is crucial for our proof:
\begin{equation}\label{eq:sigman_approx}
\sum_{t \in [n]}\left\langle \frac{B\mv u_t}{\|B \mvu_t\|}, \mat{\mv v}{0} \right\rangle^2 = \| \mv v \|^2.
\end{equation}
Let us first show~(i) using this claim. %For $t \in [n]$, let us define $\cF_{t - \frac 12}$ to be the $\sigma$-algebra $\sigma(\cF_{t - 1}, p_t)$ (see Section~\ref{sec:stoch_proc} and before Observation~\ref{obs:zellinfty}). Now 
%Now observe from the definition~\ref{def:etat} %as well as the conditional distribution of $\eta_t$ given $\cF_{t - \frac 12}$ is Rademacher,
%that $\eta_t$ is distributed as a Rademacher variable {\em independently} of $\mvu_t$
In view of the description of the underlying stochastic processes in Section~\ref{sec:stoch_proc},
observe that $|\eta_t| = 1$ for all $t \in [n]$ %$U_t$, and hence $\eta_t$ (see~\ref{def:etat}), is independent of $\mvu_t$ for all $t \in [n]$. However, since $\eta_t$ is Rademacher, it follows that
and hence
\begin{equation*}\label{eq:detlatstarsqFt-1/2}
\bE \left [ \Delta_t^2 \right ] = %\bE\left[ \eta_t^2 \right] \bE \left[\left\langle \frac{B\mv u_t}{\|B \mvu_t\|}, \mat{\mv v}{0} \right\rangle^2\right]  
\bE \left[\eta_t^2\left\langle \frac{B\mv u_t}{\|B \mvu_t\|}, \mat{\mv v}{0} \right\rangle^2\right] = \bE \left[\left\langle \frac{B\mv u_t}{\|B \mvu_t\|}, \mat{\mv v}{0} \right\rangle^2\right].
\end{equation*}
Therefore, by~\eqref{eq:sigman_approx} and the orthogonality of martingale differences, we get
\begin{equation}\label{form:sigmansq}
\sigma_n^2 = \bE[M_n^2] = \sum_{t \in [n]} \bE[\Delta_t^2] = \bE\left[ \|\mvv\|^2\right] = \|\mvv\|^2
\end{equation}
which finishes the verification of condition~(i).

It remains to verify~\eqref{eq:sigman_approx}. Recall from the discussion after
\eqref{def:MttildeMt} that the vectors $\left(\tfrac{B\mv u_1}{\|B \mvu_1\|}, \ldots, \tfrac{B\mv u_n}{\|B
\mvu_n\|}\right)$ form an ONB for $\colsp(B)$ where
\begin{equation*}
B = \mat{I_n}{Y^{\t}}
\end{equation*}
(recall~\eqref{def:B}) and hence
\begin{equation*}
\sum_{t \in [n]}\left\langle \frac{B\mv u_t}{\|B \mvu_t\|}, \mat{\mv v}{0} \right\rangle^2 = \|{\rm
Proj}_{\colsp(B)}(\mvv) \|^2.
\end{equation*}
However, since
\begin{equation*}
\mat{\mv v}{0} = B \mvv
\end{equation*}
(recall that $\mvv^{\t}Y = \mv0$),~\eqref{eq:sigman_approx} follows.

\smallskip

\noindent {\em Verifying condition~(ii).}  	We need the following moment bound.
\begin{lemma}\label{lem:lyapunov}
For any $m \in [n]$, we have
\begin{equation*}
\begin{split}
	&\sum_{t \in [n]} \bE \left[\Delta_t^{4}\right] \\
	&\quad\le C \Cr{C:xi}^2 \|\mvv\|_{\infty}^2\|\mvv\|^2  + C d \Cr{C:xi}^2 \zeta^8\left(\|\mvv\|_{\infty}^4\ m^3 + \|\mvv\|_{\infty}^4 \frac{\kappa^{4}}{m} +  n^3\|\mvv\|^4\exp\left(- c (\kappa\zeta^2)^{-2} m\right)\right).
\end{split}	\end{equation*}
\end{lemma}
One can now check that by setting
$$m = C \kappa^2 \zeta^4\log en$$
for a suitably large enough constant $C$, the right hand side above converges to $0$ as soon as
\begin{equation}\label{lim:lyapunov}
\lim_{n \to \infty} d^{1/4}\frac{\|\mvv\|_{\infty}}{\|\mvv\|}\kappa^{3/2}(\log n)^{3/4} = 0
\end{equation}
(recall that $d \le n$ and $\kappa \ge \zeta^{-2}$ by~\eqref{bnd_kappa}) which is implied by the first part of assumption~\eqref{eq:formal_cond} in
Theorem~\ref{thm:formal}. However, we also need to ensure that $m$, as chosen above, is at
most $n$ for all large enough $n$. But this follows from ~\eqref{lim:lyapunov} in view of the facts that
$\tfrac{1}{\sqrt{n}} \le \tfrac{\|\mvv\|_{\infty}}{\|\mvv\|}$ and~\eqref{bnd_kappa}.
%\begin{equation}\label{bnd_kappa}
%\kappa = \frac{n}{\lambda_{{\rm min}}(Y^{\t}Y)}	\ge \frac{dn}{\|Y\|_{\frob}^2} \stackrel{\eqref{def:zeta}}{\ge} \frac{dn}{n \zeta^2} = \zeta^{-2}.
%\end{equation}	
We now proceed to the
\begin{proof}[Proof of Lemma~\ref{lem:lyapunov}]
Since $\mvu_t = \mvu(p_t, \cA_t)$, we have in view of~\eqref{eq:identity3} and~\eqref{def:YD},
%\begin{equation*}
%\begin{split}
%\left \langle B\mv u_t, \mat{\mv v}{0} \right \rangle &\stackrel{\eqref{eq:expr_Bu_t}}{=} [\mvv^\t 0^\t] B \mve_p - [\mvv^\t 0^\t] B_t(B_t^{\t}B_t)^{-1}B_t^{\t} \cdot B\mve_p \\& = [\mvv^\t 0^\t] \mat{\mv 
%e_p}{y_p} - [\mvv^\t 0^\t] \mat{I_n[:, \cA_{t}^-]}{{Y_t^-}^{\t}} (B_t^{\t}B_t)^{-1}B_t^{\t} \cdot B\mve_p\\ 
%&= v_p - v^\t[\cA_{t}^-] (B_t^{\t}B_t)^{-1}B_t^{\t} \cdot B\mve_p = \|B\mvu_t\|^{2}\,v^\t[\cA_{t}](I_{n - t + 
%1} - Y_tD_tY_t^\t)\mve_p.
%\end{split}
%\end{equation*}
\begin{equation}\label{eq:Butinprod}
\begin{split}
\left \langle B\mv u_t, \mat{\mv v}{0} \right \rangle =
%\norm{B\mvu_{t}}^{2}(\mvv[\cA_t]^{\t}\mve_p - \mvv[\cA_t]^{\t}Y_tD_t\mvy_p) \\
%	&= 
\norm{B\mvu_{t}}^{2}\mvv[\cA_t]^{\t}(I_{a_t} - Y_tD_tY_t^{\t})\mve_p[\cA_t].
\end{split}
\end{equation}
where, like in the previous sections, we use the shorthands $a_t = |\cA_t| = n - t + 1$
(Observation~\ref{obs:swor}) and $p = p_t$. Therefore,
\begin{equation}\label{eq:Q_t1}
\begin{split}
\Delta_t^{2} &\le
%C \delta_t^{-2}\Delta_t^{2} = 
C \norm{B\mvu_{t}}^{-2}\left\langle B\mv u_t, \mat{\mv v}{0} \right \rangle^2 \stackrel{\eqref{eq:Butinprod}}{=} C\norm{B\mvu_{t}}^{2}\left(\mvv[\cA_t]^{\t}(I_{a_t} -
Y_tD_tY_t^{\t})\mve_p[\cA_t]\right)^2\\
%=\:& C \norm{B\mvu_{t}}^{2} \mvv[\cA_t]^{\t}(I_{a_t} - Y_tD_tY_t^{\t})\mve_p[\mathcal 
%A_t]\mve_p[\cA_t]^{\t} (I_{a_t} - Y_tD_tY_t^{\t})\mvv[\cA_t] 
&\stackrel{\eqref{eq:identity4}}{=}
C \norm{B\mvu_{t}}^{2}\cQ_t
\end{split}
\end{equation}
where in the first inequality we used the fact that $|\delta_t| \leq 2$ almost surely and in the final step, \begin{equation*}
\cQ_t \coloneqq \left(\mvv[\cA_t]^{\t}(I_{a_t} - Y_tD_tY_t^{\t})\mve_p[\cA_t]\right)^2.
\end{equation*}
Consequently,
\begin{equation}\label{eq:Deltat4}
\begin{split}
\sum_{t \in [n]}\bE \left[\Delta_t^{4}\right] \leq C \sum_{t \in [n]} \bE\left[\norm{B\mvu_{t}}^{4}\cQ_t^2 \right]
\stackrel{\eqref{eq:identity2}}{\le} C \cdot \Cr{C:xi}^2  \sum_{t \in [n]} \bE \left[\cQ_t^2\right].
\end{split}
\end{equation}
Let us now expand $\cQ_t^2$.
\begin{equation}\label{eq:Qt2}
\begin{split}
&\cQ_t^2 \\
&= 	\big(\mvv[\cA_t]^{\t}(I_{a_t} - Y_tD_tY_t^{\t})\mve_p[\cA_t]\big)^4 \le C \left( \left(\mvv[\cA_t]^{\t} e_p[\cA_t]\right)^4 + \left(\mvv[\cA_t]^{\t}Y_tD_tY_t^{\t}e_p[\cA_t]\right)^4 \right)\\
&= C \left( v_p^4 + \left(\mvv[\cA_t]^{\t}Y_tD_t\mvy_p\right)^4 \right) = C \left( v_p^4 + \left(\mvy_p^{\t}D_tY_t^{\t}\mvv[\cA_t]\mvv[\cA_t]^{\t}Y_tD_t\mvy_p\right)^2 \right).
\end{split}
\end{equation}
We will estimate the expected values of each of the two terms separately. The first term can be dealt
with in a simple manner.
\begin{equation}\label{eq:Expcvp4}
\begin{split}
\bE \left[v_p^4\right] &=  \bE \left[\bE \left[v_p^4 \, | \, \cF_{t-1}\right]\right] = \bE\left[ \frac{1}{n - t + 1} \sum_{i \in \cA_t} v_i^4\right]  \le \|\mvv\|_\infty^2\:\:\bE \left[\frac{1}{n - t + 1} \sum_{i \in \cA_t} v_i^2\right]\\
&\stackrel{{\rm Obs.}~\ref{obs:swor}}{=} \|\mvv\|_\infty^2  \, \frac{\|\mvv\|^2}{n}.
\end{split}
\end{equation}
The second term, on the other hand, requires a more delicate treatment. In particular, we will take
advantage of the {\em cancellations} inherent in the relation $\mvv^{\t}Y = 0$ (see~\eqref{eq:vY}) in
order to obtain {\em favorable} bounds on the moments of $\mvv[\cA_t]^{\t}Y_t$. This is
needed because we stipulated very weak upper bound on the density of the vector $\mvv$ 
in~\eqref{eq:formal_cond} %or~\eqref{eq:informal_cond} 
as already discussed in the introduction. To see this, consider the ``good'' scenario when $d = 1$ and 
$y_i$'s are same (observe that $y_i$'s are scalars in this case). A simple algebra then reveals that the
best possible bound one can obtain in this generality is the following:
\begin{equation*}
\bE\left[ \left(\mvy_p^{\t}D_tY_t^{\t}\mvv[\cA_t]\mvv[\cA_t]^{\t}Y_tD_t\mvy_p\right)^2\right] \le C(\zeta) \, \|\mvv\|_{\infty}^4\, \frac{|\cA_t|^4}{|\cA_t|^4} = C(\zeta) \,\|\mvv\|_{\infty}^4.
\end{equation*}
Summing it over $t \in [n]$, we get
\begin{equation*}
\lim_{n \to \infty}n\, \frac{\|\mvv\|_{\infty}^4}{\|\mvv\|^4} = 0
\end{equation*}
which is a {\em much} stronger requirement than the one we stipulated in %~\eqref{eq:informal_cond} and
\eqref{eq:formal_cond}.

In order to exploit the relation $\mvv^\t Y = 0$, we will use Proposition~\ref{thm:srswor_moment} to derive the following lemma.
\begin{lemma}%[Applying Sampling Without Replacement Result]
\label{lem:srsworappli}
%Recall that the sequence of active sets $\{\cA_t\}_{t = 1}^{n}$ is generated within the ideal process. 
%Therefore, for any $t \in [n]$, $\cA_t$ is a uniformly random subset of $[n]$ with cardinality $n - t + 1$. 
%This fact lets us use Theorem~\ref{thm:srswor_moment} to obtain 
For any $t \in [n]$, we have
\begin{equation}\label{eq:vtYmomentbnd}
\begin{split}
	&\bE\left[ \|\mvv[\cA_t]^{\t}Y_t\|^4\right] \leq C d \|\mvv\|_{\infty}^{4} \left(\frac{(n - t + 1)^2 (t - 1)^2}{n^4}  \|Y\|_{\frob}^{4}  +  \frac{(n - t + 1)(t - 1)}{n^2}  \zeta^2\|Y\|_{\frob}^{2}\right)\\
	&\bE \left[\|\mvv[\cA_t]^{\t}Y_t\|^2\right] \leq \frac{(n - t + 1)(t - 1)}{n(n - 1)}
	\sum_{j \in [d]}\sum_{i \in [n]} v_i^2 y_{ij}^2 \le \frac{(n - t + 1)(t - 1)}{n(n - 1)}  \|\mvv\|_\infty^2 \cdot
	\|Y\|_{\frob}^2.
\end{split}
\end{equation}%where we also recall that the matrix $Y_t = Y[\cA_t,:].$
%\begin{equation}
%\bE \|\mvv[\cA_t]^{\t}Y_t\|^2 \leq blah..
%\end{equation}
\end{lemma}
\begin{proof}%[Proof of Lemma~\ref{lem:srsworappli}]
Firstly we can bound,
\begin{equation*}%\label{eq:easy_case4th}
\begin{split}
	\|\mvv[\cA_t]^{\t}Y_t\|^4 =  \left(\sum_{j \in [d]} \left(\sum_{i \in \cA_t} v_i
	y_{ij}\right)^2\right)^2 \le d \sum_{j \in [d]} \left(\sum_{i \in \cA_t} v_i y_{ij}\right)^4
\end{split}
\end{equation*}
where we used the Cauchy-Schwarz inequality. Now taking expectation on both sides and using
$\mvv^{\t} Y = 0$, we obtain from the second formula given in~\eqref{eq:srswor_moment}:
\begin{align*}
&\bE \left[\|\mvv[\cA_t]^{\t}Y_t\|^4\right]\\
&\leq C d \sum_{j \in [d]} \left(\frac{(n - t + 1)^2 (t - 1)^2}{n^4} \left(\sum_{i \in [n]} v_i^2
y_{ij}^2\right)^2  +  \frac{(n - t + 1)(t - 1)}{n^2}\sum_{i \in [n]} v_i^4 y_{ij}^4\right).
\end{align*}
Next, we note that
\begin{align*}
\sum_{j \in [d]} \left(\sum_{i \in [n]} v_i^2
y_{ij}^2\right)^2 
&\leq \|\mvv\|_{\infty}^{4} \sum_{j \in [d]} \left(\sum_{i \in [n]}
y_{ij}^2\right)^2\\
&\leq \|\mvv\|_{\infty}^{4} \left(\sum_{j \in [d]} \sum_{i \in [n]} y_{ij}^2 \right)^2 =
\|\mvv\|_{\infty}^{4} \|Y\|_{\frob}^{4}.
\end{align*}
Also,
\begin{equation}
\sum_{j \in [d]} \sum_{i \in [n]} v_i^4 y_{ij}^4 \leq \|\mvv\|_{\infty}^{4} \sum_{j \in [d]} \sum_{i \in [n]} y_{ij}^4 \leq \zeta^2 \|\mvv\|_{\infty}^{4} \sum_{j \in [d]} \sum_{i \in [n]} y_{ij}^2 \stackrel{\eqref{def:zeta}}{\le} \zeta^2 \|\mvv\|_{\infty}^{4} \|Y\|_{\frob}^{2}.
\end{equation}
The last three displays finish the proof of the fourth moment bound. To prove the second moment
bound, we can argue similarly. All we have to do differently is to use the second moment formula in 
Proposition~\ref{thm:srswor_moment}.
\end{proof}
Now returning to the second term in~\eqref{eq:Qt2}, we start with some matrix algebra.
\begin{equation*}
\begin{split}
&\left(\mvy_p^{\t}D_tY_t^{\t}\mvv[\cA_t]\mvv[\cA_t]^{\t}Y_tD_t\mvy_p\right)^2 \le
\|\mvy_p\|^4 \left\|D_tY_t^{\t}\mvv[\cA_t]\mvv[\cA_t]^{\t}Y_tD_t\right\|_{{\rm op}}^2 \\
=& \, \|\mvy_p\|^4 \left({\rm Tr} (D_tY_t^{\t}\mvv[\cA_t]\mvv[\cA_t]^{\t}Y_tD_t)\right)^2
= \|\mvy_p\|^4 \left(\mvv[\cA_t]^{\t}Y_tD_t^2Y_t^{\t}\mvv[\cA_t]\right)^2.
\end{split}
\end{equation*}
Here in the second step we used the observation that $\Tr(A) = \|A\|_{{\rm op}}$ for any rank 1
matrix $A$. Taking conditional expectation {\em over} $p$, \ie  w.r.t. $\cF_{t - 1}$, we then get
\begin{equation}\label{eq:yp4qtcondexpect}
\begin{split}
\bE \left[ \|\mvy_p\|^4 \left(\mvv[\cA_t]^{\t}Y_tD_t^2Y_t^{\t}\mvv[\cA_t]\right)^2	\,| \,
\cF_{t - 1}\right]  = \frac 1{n - t + 1} \sum_{i \in \cA_t} \|y_i\|^4 \cdot \left(\mvv[\cA_t]^{\t}Y_tD_t^2Y_t^{\t}\mvv[\cA_t]\right)^2.
\end{split}
\end{equation}
We will bound this term in two different ways based on how large $t$ is. Firstly we can bound,
\begin{equation}\label{eq:easy_case4th}
\begin{split}
\frac 1{n - t + 1}&\sum_{i \in \cA_t} \|y_i\|^4 \cdot \left(\mvv[\cA_t]^{\t}Y_tD_t^2Y_t^{\t}\mvv[\cA_t]\right)^2\le \zeta^4 \|D_t\|_{{\rm op}}^4\|\mvv[\cA_t]^{\t}Y_t\|^4\\
&\le \zeta^4 \|\mvv[\cA_t]^{\t}Y_t\|^4
\end{split}
\end{equation}
where in the last step we used the fact that $\|D_t\|_{{\rm op}} \leq 1$ from~\eqref{def:YD}.

Taking expectations on both sides and using the fourth moment bound in Lemma~\ref{lem:srsworappli} and subsequently plugging it into~\eqref{eq:yp4qtcondexpect}, we
deduce
\begin{equation}\label{eq:srsmomentbnd}
\begin{split}
&\bE \left[ \|\mvy_p\|^4 \left(\mvv[\cA_t]^{\t}Y_tD_t^2Y_t^{\t}\mvv[\cA_t]\right)^2\right] \\
&\le C d \, \zeta^4  \,\|\mvv\|_{\infty}^{4} \left(\frac{(n - t + 1)^2 (t - 1)^2}{n^4}  \|Y\|_{\frob}^{4}  +  \frac{(n - t + 1)(t - 1)}{n^2}  \zeta^2 \|Y\|_{\frob}^{2}\right).
\end{split}
\end{equation}
%Now 
From this we get the following bound for any $m \in [n]$:
\begin{equation}
\begin{split}
&\sum_{t = n - m + 1}^n\bE \left[ \|\mvy_p\|^4 \left(\mvv[\cA_t]^{\t}Y_tD_t^2Y_t^{\t}\mvv[\cA_t]\right)^2\right]\\
&\le C d \, \zeta^4\, \|\mvv\|_{\infty}^4\, \left( \frac{m^3}{n^2} \|Y\|_{\frob}^4 + \frac{m^2}{n} \zeta^2 \|Y\|_{\frob}^2\right) \\ &\le C d \, \zeta^8 \, \|\mvv\|_{\infty}^4\ m^3.
\end{split}
\end{equation}
where in the final step we used the fact that $\|Y\|_{\frob}^2 \leq n \zeta^2$~\eqref{def:zeta}. %We obtained the above by using $t - 1 \leq n, n - t + 1 \leq l_n$ in~\eqref{eq:srsmomentbnd} when $n - l_n + 1 \leq t \leq n$.
%In the sequel, we will use the shorthand
%\begin{equation}\label{def:SSj}
%{\rm SS}_j \coloneqq \sum_{i \in [n]} y_{ij}^2 \mbox{ for } j \in [d].
%\end{equation}

%\newline

Next we deal with the sum over $t \in [n - m]$ which requires additional work. To this end, let us
%consider the event, for any $t \in [n]$  (cf.~Proposition~\ref{prop:lambdaminconc}),
%\begin{equation*}
%\begin{split}
%\mathcal H_t \coloneqq \left\{\left\| Y[\, : \cA_t]^{\t}Y[\, : \cA_t] - \frac{n - t + 
%1}{n}Y^{\t}Y\right\|_{{\rm op}} \le  \frac{n - t + 1}{2n} \|Y\|_{\frob}^2\right\}.
%\end{split}	
%\end{equation*}Also 
recall the event $\cG_{2, t}$ from~\eqref{def:G1sG2s}:
\begin{equation*}
\mathcal{G}_{2,t} = \left\{\|Y_t^{\t} Y_t - \frac{n - t + 1}{n} Y^{\t} Y\|_{{\rm op}} \geq  \frac{n - t + 1}{2n} \lambda_{\min}\big(Y^{\t} Y\big)\right\}.
\end{equation*} %from~\eqref{def:G1sG2s}. %and define $\cG_t = \cG_{2, t} \cap \mathcal H_{t}$. 
%Observe that, by Lemma~\ref{lem:smallevent2} and~\eqref{eq:Ctopbnd}, we have 
We can bound the expectation {\em on} the event $\cG_{2, t}$ as follows:
%\begin{equation*}
%\begin{split}
%\lambda_{{\rm min}}(Y_t^{\t}Y_t)  \ge \frac{n - t + 1}{2n}\lambda_{{\rm min}}(Y^{\t}Y) \mbox{ and } \|Y_t \|_{\frob}^2 \le C \frac{n - t + 1}{n}\|Y\|_{\frob}^2.
%\end{split}		
%\end{equation*}	
%\begin{equation}\label{eq:Ctopbnd}
%\begin{split}
%\|D_t\|_{{\rm op}}  \le \left(1 + \frac{n - t + 1}{2n}\lambda_{{\rm min}}(Y^{\t}Y)\right)^{-1} \mbox{ and } \|Y_t \|_{\frob}^2 \le C \frac{n - t + 1}{n}\|Y\|_{\frob}^2.
%\end{split}		
%\end{equation}	
%Therefore 
%we can write  %after first taking a conditional expectation w.r.t. $\cF_{t-1}$ and subsequently taking the unconditional expectation,
\begin{equation*}
\begin{split}
&\sum_{t = 1}^{n - m}\bE \left[ \|\mvy_p\|^4 \left(\mvv[\cA_t]^{\t}Y_tD_t^2Y_t^{\t}\mvv[\cA_t]\right)^2 \1_{\cG_{2, t}}\right] \\
=&  \sum_{t = 1}^{n - m}\bE\left[\bE \left[ \|\mvy_p\|^4 \left(\mvv[\cA_t]^{\t}Y_tD_t^2Y_t^{\t}\mvv[\cA_t]\right)^2 \, | \, \cF_{t-1}\right]\1_{\cG_{2, t}}\right]\\
\stackrel{\eqref{eq:yp4qtcondexpect}}{=}& \sum_{t = 1}^{n - m}\frac 1{n - t + 1} \, \bE \left[\sum_{i \in \cA_t} \|y_i\|^4 \cdot \left(\mvv[\cA_t]^{\t}Y_tD_t^2Y_t^{\t}\mvv[\cA_t]\right)^2\1_{\cG_{2, t}}\right]\\
\stackrel{\eqref{def:zeta}+\eqref{eq:Ctopbnd}}{\le}& \sum_{t = 1}^{n - m}\frac {\zeta^2}{n - t + 1} \cdot \frac{n - t + 1}{n}\cdot\|Y\|_{\frob}^2 \,\bE \left[\left(\mvv[\cA_t]^{\t}Y_tD_t^2Y_t^{\t}\mvv[\cA_t]\right)^2\1_{\cG_{2, t}}\right]\\
\le& \sum_{t = 1}^{n - m}\frac {\zeta^2}{n - t + 1} \cdot \frac{n - t + 1}{n}\cdot\|Y\|_{\frob}^2 \,\bE \left[\|D_t\|_{{\rm op}}^4\|\mvv[\cA_t]^{\t}Y_t\|^4\1_{\cG_{2, t}}\right]\\
%\stackrel{\eqref{eq:Ctopbnd}}{\le}& \sum_{t = 1}^{n - m}\frac {\zeta^2}{n} \cdot\frac{\|Y\|_{\frob}^2}{\left(\frac{n-t+1}{n}\lambda_{{\rm min}}(Y^{\t}Y)\right)^4} \bE \left[\|\mvv[\mathcal 
%A_t]^{\t}Y_t\|^4\right].
\stackrel{\eqref{eq:Ctopbnd}}{\le}& \sum_{t = 1}^{n - m}\frac {C\zeta^2}{n} \cdot\frac{\|Y\|_{\frob}^2\kappa^4}{(n - t + 1)^4} \cdot \bE \left[\|\mvv[\cA_t]^{\t}Y_t\|^4\right].
%&\le C d\zeta^2 \sum_{t = 1}^{n - m} \frac 1{n - t + 1} \cdot \frac{n - t + 1}{n} \|Y\|_{{\rm 
%Frob}}^2 \cdot \frac{\bE \sum_{j \in [d]} \left(\sum_{i \in \cA_t} v_i y_{ij}\right)^4}{\frac{(n - t + 1)^4}{n^4}\lambda_{{\rm min}}(Y^{\t}Y)^4}  + \bE {\rm Rem}_3,
\end{split}
\end{equation*}
%
%In the above, in the first inequality we used $\|y_i\| \leq \zeta$, in the last inequality we used the bound on $\|D_t\|_{{\rm op}}$ in~\eqref{eq:Ctopbnd} because we are working under the event $\cG_{2, t}$ and then dropped the indicator of this event.
Now using the fourth moment bound in~\eqref{eq:vtYmomentbnd} to bound the expectation term in the
last step above, we further obtain
\begin{equation}\label{eq:bnd_maintermlyapunov}
\begin{split}
&\,\sum_{t = 1}^{n - m}\bE \left[ \|\mvy_p\|^4 \left(\mvv[\cA_t]^{\t}Y_tD_t^2Y_t^{\t}\mvv[\cA_t]\right)^2 \1_{\cG_{2, t}}\right] \\
%\le &\sum_{t = 1}^{n - m}\frac {Cd\zeta^2 \|v\|_{\infty}^4}{n} \cdot\frac{\|Y\|_{\frob}^2}{\left(\frac{n-t+1}{n}\lambda_{{\rm min}}(Y^{\t}Y)\right)^4}  \cdot \left( \frac{(n - t + 1)^2(t - 1)^2}{n^4} \|Y\|_{\frob}^4 +  \zeta^2 \frac{(n - t + 1)(t - 1)}{n^2} \|Y\|_{\frob}^2\right)
\le &\,\sum_{t = 1}^{n - m}\frac {Cd \, \zeta^2 \|\mvv\|_{\infty}^{4}}{n} \cdot\frac{\|Y\|_{\frob}^4\kappa^4}{(n - t + 1)^4} \left(\frac{(n - t + 1)^2 (t - 1)^2}{n^4}  \|Y\|_{\frob}^{2}  +  \frac{(n - t + 1)(t - 1)}{n^2} \zeta^2\right)\\
%\le &\frac {Cd\zeta^2 \|v\|_{\infty}^4}{n} \cdot\frac{\|Y\|_{\frob}^2}{\left(\lambda_{{\rm min}}(Y^{\t}Y)\right)^4}  \cdot \left( \frac{n^2}{m} \|Y\|_{\frob}^4 +  \zeta^2 \frac{n^3}{m^2} \|Y\|_{\frob}^2\right)
\le &\,Cd \,\zeta^2 \|\mvv\|_{\infty}^4 \cdot \|Y\|_{\frob}^4\kappa^4  \cdot \left( \frac{\|Y\|_{\frob}^{2}}{n^5} \sum_{t = 1}^{n - m} \frac{(t - 1)^2}{(n - t + 1)^2} + \frac{\zeta^2}{n^3}\sum_{t = 1}^{n - m}  \frac{(t - 1)}{(n - t + 1)^3}\right)\\
\le &\,Cd \, \zeta^2 \|\mvv\|_{\infty}^4 \cdot \|Y\|_{\frob}^4\kappa^4\cdot \left( \frac{1}{mn^3} \|Y\|_{\frob}^2 + \frac{\zeta^2}{m^2n^2} \right)
\le \,Cd \,\zeta^8 \|\mvv\|_{\infty}^4 \frac{\kappa^{4}}{m}.
\end{split}
\end{equation}
where in the last step we again used the fact that $\|Y\|_{\frob}^2 \leq n \zeta^2$~\eqref{def:zeta}.
Bounding the same expectation {\em on} the {\em unlikely} event $\cG_{2, t}^c$, is relatively
simpler in view of Lemma~\ref{lem:smallevent2}.
\begin{equation}\label{eq:bnd_remtermlyapunov}
\begin{split}
&\sum_{t = 1}^{n - m}\bE \left[ \|\mvy_p\|^4 \left(\mvv[\cA_t]^{\t}Y_tD_t^2Y_t^{\t}\mvv[\cA_t]\right)^2 \1_{\cG_{2, t}^c}\right]\\
&%\stackrel{\eqref{eq:yp4qtcondexpect}}{=} \sum_{t = 1}^{n - m} \bE \left[\frac 1{n - t + 1}\sum_{i \in \cA_t} \|y_i\|^4 \cdot \left(\mvv[\cA_t]^{\t}Y_tD_t^2Y_t^{\t}\mvv[\cA_t]\right)^21_{\cG_{2, t}^c}\right]\\
\stackrel{\eqref{eq:yp4qtcondexpect} +~\eqref{def:zeta}}{\le} \zeta^4 \sum_{t = 1}^{n - m} \bE \left[\|D_t\|_{{\rm op}}^4\|\mvv[\cA_t]^{\t}Y_t\|^4\1_{\cG_{2, t}^c}\right] \\
\le& \, \zeta^4  \sum_{t = 1}^{n - m} \bE \left[\|\mvv[\cA_t]\|^4 \|Y_t\|_{\frob}^4\1_{\mathcal
G_{2, t}^c}\right] \stackrel{\eqref{def:zeta}}{\le} \zeta^8 \|\mvv\|^4 \sum_{t = 1}^{n - m} (n - t + 1)^2\P[\cG_{2, t}^c]  \\ \stackrel{\eqref{eq:smallevent2}}{\le}& \,  C d \zeta^8 \|\mvv\|^4n^3\ \exp\left(- c (\kappa\zeta^{2})^{-2}m\right)
\end{split}
\end{equation}
where in the second step we used the Cauchy-Schwarz inequality to deduce $$\|\mvv[\cA_t]^{\t}Y_t\|^4 \leq \|\mvv[\cA_t]\|^4 \|Y_t\|^4$$ along with the fact that $\|D_t\|_{{\rm op}}
\leq 1$.

Finally we obtain the lemma by collecting the bounds in~\eqref{eq:Expcvp4} and
\eqref{eq:bnd_maintermlyapunov}--\eqref{eq:bnd_remtermlyapunov} into the right hand side of
\eqref{eq:Qt2} after taking expectations on both sides and adding over $t \in [n]$ and subsequently
plugging the resulting bound into the right hand side of~\eqref{eq:Deltat4}.
%we finally get%\begin{equation*}
%	\begin{split}
%		&\sum_{t \in [n]} \bE \Delta_t^{4} \le \,C \Bigg(\| v\|_{\infty}^2  + d \|\mvv\|_{\infty}^4\, \left( \frac{m^3}{n^2} \|Y\|_{\frob}^4 +  \frac{m^2}{n} \|Y\|_{\frob}^2\right) + \\&d^2 \|v\|_{\infty}^4 \cdot\frac{\|Y\|_{\frob}^2}{\left(\lambda_{{\rm min}}(Y^{\t}Y)\right)^4}  \cdot \left( \frac{n}{m} \|Y\|_{\frob}^4 +   \frac{n^2}{m^2} \|Y\|_{\frob}^2\right)\\
%		&+ d n^3\, \exp\left(-c \frac{m \lambda_{\min}^2\big(Y^{\t}Y\big)}{n^2}\right) \Bigg)
%	\end{split}	
%	\end{equation*}	
%\begin{equation*}
%\begin{split}
%&\sum_{t \in [n]} \bE \left[\Delta_t^{4}\right] \le \,\Cr{C:xi}^2C \|\mvv\|_{\infty}^2\|\mvv\|^2  + C d \|\mvv\|_{\infty}^4\ l_n^3 + Cd \|v\|_{\infty}^4 \frac{\kappa^{4}}{l_n} +  C d n^3\ \exp\left(- c \kappa^{-2} l_n\right).
%\end{split}	
%\end{equation*}	
%
%
%This finishes the proof.
\end{proof}

%{\color{red} Now use this lemma to show Lyapunov's condition.}

\smallskip

\noindent {\em Verifying condition~(iii).}  Letting $\cF_{t - 1/2}$ denote the $\sigma$-algebra
generated by $\cF_{t-1}, p_t$ and $p_t^{{\gsw}}$, observe that $\cF_{t-1} \subset \cF_{t -1/2}
\subset \cF_t$ (cf.~\eqref{def:F_t}) and that $\mvu_t$ is measurable relative to $\cF_{t -1/2}$. Also
notice that $\eta_t$ is distributed as a Rademacher variable conditionally on $\cF_{t - \frac 12}$
(see~\eqref{def:etat}). Therefore we can rewrite~\eqref{def:Qn} as follows:
\begin{equation*}
\begin{split}
%\label{redef:Qn}%Q'_n \coloneqq \sum_{1 \leq t \leq n} \bE \left[\norm{B\mvu_{t}}^{-2}\left \langle B\mv u_t, \mat{\mv v}{0} \right \rangle^2 \, \Big | \, \cF_{t - 1} \right].
Q_n &= \sum_{t \in [n]} \bE \left[\eta_t^2\norm{B\mvu_{t}}^{-2}\left \langle B\mv u_t, \mat{\mv v}{0} \right \rangle^2 \, \Big | \, \cF_{t - 1} \right]  \\
&= \sum_{t \in [n]} \bE \left[\norm{B\mvu_{t}}^{-2}\left \langle B\mv u_t, \mat{\mv v}{0} \right \rangle^2 \bE\left[ \eta_t^2 \, \Big | \, \cF_{t - \frac12} \right] \, \Big | \, \cF_{t - 1} \right] = \sum_{t \in [n]} \bE \left[\norm{B\mvu_{t}}^{-2}\left \langle B\mv u_t, \mat{\mv v}{0} \right \rangle^2  \, \Big | \, \cF_{t - 1} \right].
\end{split}
\end{equation*}
The main result in this part is the following lemma.
\begin{lemma}\label{lem:idealclt}
Let us consider the random variables
\begin{equation*}
T_n \coloneqq  \sum_{t \in [n]} \frac{1}{|\cA_t|} \|\mvv[\cA_t]\|^2 = \sum_{t \in [n]} \frac{1}{n - t + 1} \|\mvv[\cA_t]\|^2
\end{equation*}
and
\begin{equation*}
\mathcal{QF} \coloneqq \sum_{t \in [n]}\frac{1}{n - t + 1}\sum_{j \in \cA_t}v_j^2\mvy_j^{\t}
D_t \mvy_j  + \sum_{t \in [n]}\frac1{n - t + 1}\,\mvv[\cA_t]^{\t}Y_tD_tY_t^{\t}\mvv[\cA_t].
\end{equation*}
Then we have,
\begin{equation}\label{eq:idealclt1}
|Q_n - T_n| \leq C \cdot \Cr{C:xi}^2\mathcal{QF}
\end{equation}
where for any $m \in {[n]}$,
\begin{equation}\label{eq:idealclt2}
\bE \left[\mathcal{QF}\right] \leq C \zeta^2\|\mvv\|_\infty^2 \left(m + \kappa \log en\right) + d
\|\mvv\|_\infty^2\log en + C d n \zeta^2 \|\mvv\|^2 \exp\left(-c\:(\kappa\zeta^2)^{-2}m\right).
\end{equation}
Furthermore, we have the following expressions or bounds for the mean and variance of $T_n$:
\begin{equation}\label{eq:idealclt3}
\bE [T_n] = \|\mvv\|^2
\end{equation}
and
\begin{equation}\label{eq:idealclt4}
\frac{\var\left[T_n\right]}{\|\mvv\|^4} \leq \frac{\|\mvv\|_{\infty}^2}{\|\mvv\|^2} (\log en)^2 + \frac{1}{n - 1}.
\end{equation} \end{lemma}
Now let us set
\begin{equation*}
m = C \kappa^2 \zeta^4 \log n
\end{equation*}
for an appropriately large enough constant $C$ so that, by~\eqref{eq:idealclt2},
\begin{equation*}%\label{lim:conc}
\lim_{n \to \infty} \frac{\bE\left[\mathcal{QF}\right]}{\|\mvv\|^2} = 0 \mbox{ as soon as } \lim_{n \to \infty} \max(m^{1/2}, d^{1/2} (\log n)^{1/2}, \kappa^{1/2}(\log n)^{1/2}) \frac{\|\mvv\|_\infty}{\|\mvv\|}
= 0
\end{equation*}
which follows from the first part of assumption~\eqref{eq:formal_cond} in
Theorem~\ref{thm:formal}. Therefore under the same assumption, $|Q_n - T_n|$ converges to $0$
in probability as $n \to \infty$. Similarly, by~\eqref{eq:idealclt3} and~\eqref{eq:idealclt4}, the assumption
implies $\frac{T_n}{\|\mvv\|^2}$ converges to $1$ in probability. Together they imply condition~(ii). We
also need to check that $m$, as defined above, is at most $n$ for all large enough $n$ which has
already been verified under assumption~\eqref{eq:formal_cond} in the previous part.

\begin{proof}[Proof of Lemma~\ref{lem:idealclt}]
\noindent{\em Proof of~\eqref{eq:idealclt1}}. To this end, let us recall from~\eqref{eq:identity4}:
\begin{equation}\label{eq:Q_t}
\begin{split}%\Delta_t^2	= 
\norm{B\mvu_{t}}^{-2}\left \langle B\mv u_t, \mat{\mv v}{0} \right \rangle^2 &= \norm{B\mvu_{t}}^{2} \left(\mvv[\cA_t]^{\t}(I_{a_t} - Y_tD_tY_t^{\t})\mve_p[\cA_t]\right)^2\\
%\mve_p[\cA_t]^{\t} (I_{a_t} - 
%	Y_tD_tY_t^{\t})\mvv[\cA_t]\\
&=\norm{B\mvu_{t}}^{2} \cQ_t = \cQ_t + (\norm{B\mvu_{t}}^{2} - 1)\cQ_t
\end{split}
\end{equation}
with  $\cQ_t = \left(\mvv[\cA_t]^{\t}(I_{a_t} - Y_tD_tY_t^{\t})\mve_p[\cA_t]\right)^2$. Let us deal with the second term first.  %bound on $\norm{B\mvu_{t}}^{2} - 1$ given by~\eqref{eq:Butnormbnd} in the first and the second line and the bound
%\begin{equation*}%\label{eq:bndypDyp}
%\mvy_p^{\t}D_{t}\mvy_p \le  \|D_t\|_{{\rm op}} \|\mvy_p^{\t}\|^2  \stackrel{\eqref{def:zeta} + 
%\eqref{def:YD}}{\le} \zeta^2
%\end{equation*}
%in the second line along with the 
%Using the expression~\eqref{eq:identity5} for $\cQ_t$, we can write
%further 
%\begin{equation*}
%	\begin{split}
%		&(\norm{B\mvu_{t}}^{2} - 1)\cQ_t	\le \Cr{C:xi} \,\mvy_p^{\t}D_{t}\mvy_p \mvv[\cA_t]^{\t}(I_{a_t} - Y_tD_tY_t^{\t})\mve_p[\cA_t]\mve_p[\cA_t]^{\t} (I_{a_t} - Y_tD_tY_t^{\t})\mvv[\cA_t]\\
%		&\stackrel{\eqref{eq:qformineq} +~\eqref{eq:Butnormbnd}}{\le} C \, \mvv[\cA_t]^{\t}\mve_p[\cA_t]\mvy_p^{\t}C_{t}\mvy_p\mve_p[\cA_t]^{\t}\mvv[\cA_t] + C \, \mvv[\mathcal 
%		A_t]^{\t}Y_tD_tY_t^{\t}\mve_p[\cA_t]\mve_p[\cA_t]^{\t}Y_tD_tY_t^{\t}\mvv[\mathcal 
%		A_t].
%	\end{split}
%\end{equation*}
\begin{equation*}
\begin{split}
&(\|B\mv u_t\|^{2} - 1)\cQ_t \stackrel{\eqref{eq:identity2}}{\le}  \Cr{C:xi} \mvy_p^{\t} D_t\mvy_p \mvv[\cA_t]^{\t}(I_{a_t} - Y_tD_tY_t^{\t})\mve_p[\cA_t]\mve_p[\cA_t]^{\t} (I_{a_t} - Y_tD_tY_t^{\t})\mvv[\cA_t]\\
&\stackrel{\eqref{eq:qformineq} +~\eqref{eq:Butnormbnd}}{\le} 2 \Cr{C:xi} \, \mvv[\cA_t]^{\t}\mve_p[\cA_t]\mvy_p^{\t}D_{t}\mvy_p\mve_p[\cA_t]^{\t}\mvv[\cA_t] +
2\Cr{C:xi}^2 \, \mvv[\cA_t]^{\t}Y_tD_tY_t^{\t}\mve_p[\cA_t]\mve_p[\cA_t]^{\t}Y_tD_tY_t^{\t}\mvv[\cA_t].
\end{split}
\end{equation*}
We now take conditional expectations on both sides with respect to $\cF_{t-1}$ for each of the two
terms separately (recall that $|\cA_t| = a_t =  n - t + 1$).
\begin{equation*}%\label{eq:cond_expecbnd}
\begin{split}
\bE&\left[\mvv[\cA_t]^{\t}\mve_p[\cA_t]\mvy_p^{\t}D_{t}\mvy_p\mve_p[\cA_t]^{\t}\mvv[\cA_t] \, | \, \mathcal
F_{t-1}\right] = \frac1{n - t + 1}\,\sum_{j \in \cA_t} v_j^2 \mvy_j^{\t} D_t \mvy_j	, \text{ whereas}\\
\bE&\left[\mvv[\cA_t]^{\t}Y_tD_tY_t^{\t}\mve_p[\cA_t]\mve_p[\cA_t]^{\t}Y_tD_tY_t^{\t}\mvv[\cA_t] \, | \, \mathcal
F_{t-1}\right] = \frac1{n - t + 1}\,\mvv[\cA_t]^{\t}Y_tD_tY_t^{\t}Y_tD_tY_t^{\t}  \mvv[\cA_t]\\
&\stackrel{\eqref{def:XX_tBC}}{=} \frac1{n - t + 1}\,\mvv[\cA_t]^{\t}Y_tD_t(D_t^{-1} - I_d)D_tY_t^{\t}\mvv[\cA_t] \le \frac1{n - t + 1}\,\mvv[\cA_t]^{\t}Y_tD_tY_t^{\t}\mvv[\cA_t].
\end{split}
\end{equation*}
where in the last step we used the fact that $D_t - D_t^2$ is at most $D_t$ in Loewner order. %Let us further inspect the first bound. To this end notice that
%\begin{equation}\label{eq:vycyv_bnd}
%\begin{split}
%\mvv[\cA_t]^{\t}Y_tD_tY_t^{\t}\mvv[\cA_t] &=  \sum_{k \in [d]}v[\cA_t]^{\t}Y_t[:, k]  \sum_{i, j \in \cA_t} v_i v_j \mvy_i^{\t} 
%D_t \mvy_j \le \sum_{i, j \in \cA_t} v_i v_j \mvy_i^{\t} D_t \mvy_j\\ 
%&\le \sum_{i, j \in \cA_t} (v_i^2\mvy_i^{\t} D_t \mvy_i + v_j^2\mvy_j^{\t} D_t \mvy_j)= d\sum_{j \in 
%\cA_t}v_j^2\mvy_j^{\t} D_t \mvy_j.
%\end{split}
%\end{equation}
Putting together the last %three 
two displays, we obtain
\begin{equation}\label{eq:bnd_cond_expect1}
\bE \left[(\norm{B\mvu_{t}}^{2} - 1)\cQ_t \, | \, \cF_{t-1}\right] \le \frac{2\Cr{C:xi}}{n - t + 1}\, \sum_{j \in \cA_t} v_j^2\mvy_j^{\t} D_t \mvy_j + \frac {2\Cr{C:xi}^2}{n - t + 1}\,\mvv[\cA_t]^{\t}Y_tD_tY_t^{\t}\mvv[\cA_t].
\end{equation}

As to the first term on the right hand side of~\eqref{eq:Q_t}, we can write
\begin{equation*}%\label{eq:Q_tcond_expect2}
\begin{split}
\bE\left[ \cQ_t \, | \, \cF_{t-1}\right] &= \frac1{n - t + 1} \mvv[\cA_t]^{\t}(I_{a_t} - Y_tD_tY_t^{\t})^2\mvv[\cA_t]\\
&= \frac1{n - t + 1} \mvv[\cA_t]^{\t}(I_{a_t} - 2Y_tD_tY_t^{\t} + Y_tD_tY_t^{\t}Y_tD_tY_t^{\t})\mvv[\cA_t]\\
&\stackrel{\eqref{def:XX_tBC}}{=} \frac1{n - t + 1} \mvv[\cA_t]^{\t}(I_{a_t} - 2Y_tD_tY_t^{\t} + Y_tD_t(D_t^{-1} - I_d)D_tY_t^{\t})\mvv[\cA_t]\\
&= \frac1{n - t + 1} \|\mvv[\cA_t]\|^2 - \frac1{n - t + 1}\mvv[\cA_t]^{\t}Y_t (D_t + D_t^2)Y_t^{\t}\mvv[\cA_t].
\end{split}
\end{equation*}
We can bound the second term in the last step above as follows:
%\begin{equation*}
%\widetilde{\cQ}_t \stackrel{{\rm def.}}{=} \frac1{a_t}\mvv[\cA_t]^{\t}Y_t (D_t + D_t^2)Y_t^{\t}\mvv[\cA_t] = \frac2{a_t}\mvv[\cA_t]^{\t}Y_t 
%D_tY_t^{\t}\mvv[\cA_t] \stackrel{\eqref{eq:vycyv_bnd}}{\le} \frac{2d}{|\cA_t|}\sum_{j \in 
%	\cA_t}v_j^2\mvy_j^{\t} D_t \mvy_j.
%\end{equation*}
\begin{equation*}
\frac1{n - t + 1}\mvv[\cA_t]^{\t}Y_t (D_t + D_t^2)Y_t^{\t}\mvv[\cA_t] \stackrel{\eqref{def:XX_tBC}}{\le} \frac2{n - t + 1}\mvv[\cA_t]^{\t}Y_t
D_tY_t^{\t}\mvv[\cA_t]. %\stackrel{\eqref{eq:vycyv_bnd}}{\le} \frac{2d}{|\cA_t|}\sum_{j \in \cA_t}v_j^2\mvy_j^{\t} D_t \mvy_j.
\end{equation*}
where we used $D_t^2$ is at most $D_t$ in Loewner order.

Now taking the conditional expectations on both sides of~\eqref{eq:Q_t} and plugging the results of
previous two displays as well as~\eqref{eq:bnd_cond_expect1} into the resulting expression we obtain
%	\begin{equation*}
%		\begin{split}
%			\mathcal R_t &\coloneqq \left|\,\bE \left[ \norm{B\mvu_{t}}^{-2}\left \langle B\mv u_t, %\mat{\mv v}{0} \right \rangle^2  \Big| \, \cF_{t-1}\right] - \frac1{a_t} \|\mvv[\cA_t]\|^2\,\right| \\
%			&\le \frac{C}{n - t + 1}\sum_{j \in \cA_t}v_j^2\mvy_j^{\t} D_t \mvy_j + \frac C{n - t + 1}\,\mvv[\cA_t]^{\t}Y_tD_tY_t^{\t}\mvv[\cA_t].
%\end{split}
%	\end{equation*}
\begin{equation*}
\begin{split}
\mathcal R_t &\coloneqq \left|\,\bE \left[ \Delta_t^2  \Big| \, \cF_{t-1}\right] - \frac1{a_t} \|\mvv[\cA_t]\|^2\,\right| \\
&\le \frac{C\cdot\Cr{C:xi}}{n - t + 1}\sum_{j \in \cA_t}v_j^2\mvy_j^{\t} D_t \mvy_j + \frac {C\cdot\Cr{C:xi}^2}{n - t + 1}\,\mvv[\cA_t]^{\t}Y_tD_tY_t^{\t}\mvv[\cA_t]
\end{split}
\end{equation*}
which yields~\eqref{eq:idealclt1}.

\begin{comment}
{\color{red}
Hence, in order to show that $Q'_n = 1 + o_p(1)$ as $n \to \infty$, it suffices to prove that
{\color{red} under the
assumptions of Proposition~\ref{prop:cltideal}:}
\begin{enumerate}[label = {(iii)-{\bf \Alph*.}}]
\item\label{conc:B} $\sum_{t \in [n]}\frac{1}{n - t + 1}\sum_{j \in \cA_t}v_j^2\mvy_j^{\t}
D_t \mvy_j  + \sum_{1 \le t \le n}\frac1{n - t + 1}\,\mvv[\cA_t]^{\t}Y_tD_tY_t^{\t}\mvv[\cA_t] = o_p(1)$ as $n \to \infty$.
\item \label{conc:A} $\sum_{t \in [n]}\frac1{n - t + 1} \|\mvv[\cA_t]\|^2 = 1 + o_p(1)$ as $n \to \infty$.
\end{enumerate}
Let us start with the first term in item~\ref{conc:B} which is easiest of all.}%In view of the 
\end{comment}

\smallskip

\noindent{\em Proof of~\eqref{eq:idealclt2}}. Let us start with the first term in $\mathcal{QF}$ which is the easiest of the two. %In view of the 
To this end, let us evaluate its expectation as follows: %(recall that $a_t = n - t + 1$):
\begin{equation}\label{eq:QF1_bnd}
\begin{split}
\bE \left[\sum_{t \in [n]}\frac{1}{n - t + 1}\sum_{j \in \cA_t}v_j^2\mvy_j^{\t}
D_t \mvy_j\right] &\le \|\mvv\|_{\infty}^2 \bE \left[\sum_{t \in [n]}\frac{1}{n - t + 1}\,{\rm Tr}(Y_t D_tY_t^{\t})\right]\\
&\stackrel{\eqref{def:YD}}{\le} d\|\mvv\|_{\infty}^2 \sum_{t \in [n]}\frac{1}{n - t + 1} \le d\|\mvv\|_{\infty}^2 \log en.
\end{split}
\end{equation}
%where in the second last inequality we used the fact that $\|D_tY_t^{\t}Y_t\|_{\rm op} \leq 1$.

Tackling the second term in $\mathcal{QF}$, which we henceforth refer to as $S_n$, requires a careful
analysis due to a similar reason as that for the second term in~\eqref{eq:Qt2} while verifying
condition~(ii). For some $m \in {[n]}$, let us write
\begin{equation}\label{eq:Sn}
\begin{split}
S_n &= S_{n - m} + \sum_{t = n - m + 1}^n \frac 1{n - t +1} \,\mvv[\cA_t]^{\t}Y_tD_tY_t^{\t}\mvv[\cA_t]\\
&\le S_{n - m} + \sum_{t = n - m + 1}^n \frac 1{n - t +1} \,\|\mvv[\cA_t]^{\t}Y_t\|^2,%\\
%&\le S_{n - m} + \sum_{j \in [d]}\,\sum_{n - m \le t \le n} \frac 1{n - t +1} \, \left( \sum_{i \in \cA_t} v_i y_{ij}\right)^2\\&\eqqcolon S_{n - m}  + \sum_{j \in [d]} \, {\rm Rem}_{2, j}  \eqqcolon S_{n - m}  + {\rm Rem}_{2},
\end{split}
\end{equation}
where in the %second 
final step we used the observation that $\|D_t\|_{{\rm op}} \le 1$. First we analyze the
expected values of %${\rm Rem}_{2, j}$'s. 
the second summand. Since $\mvv^{\t}Y = 0$, we obtain in view of the second moment bound in Lemma~\ref{lem:srsworappli},
\begin{equation}
\label{eq:rem1j}
\begin{split}
%\bE {\rm Rem}_{2, j} 
%&\bE \sum_{n - m \le t \le n} \frac 1{n - t +1} \,\|\mvv[\cA_t]^{\t}Y_t\|^2 \\
%&= 	\sum_{n - m + 1 \le t \le n}\, \frac 1{n - t + 1} \frac{(n - t + 1)(t - 1)}{n(n - 1)} 
%\sum_{j \in [d]}\sum_{i \in [n]} v_i^2 y_{ij}^2 \le \frac{\|v\|_\infty^2 \cdot {\rm SS}_j}{n(n - 1)}\,nm \\
%&= \frac{\|v\|_\infty^2 \cdot {\rm SS}_j \cdot m}{n - 1}
&\bE \left[\sum_{t = n - m + 1}^n \frac 1{n - t +1} \,\|\mvv[\cA_t]^{\t}Y_t\|^2 \right]\\
&= 	\sum_{t = n - m + 1}^n\, \frac 1{n - t + 1} \frac{(n - t + 1)(t - 1)}{n(n - 1)}
\|\mvv\|_\infty^2 \cdot \|Y\|_{\frob}^2 \le \frac{\|\mvv\|_\infty^2 \cdot \|Y\|_{\frob}^2} {n(n - 1)}\,nm \\
&= \frac{\|\mvv\|_\infty^2 \cdot \|Y\|_{\frob}^2 \cdot m}{n - 1} \leq C \zeta^2\|\mvv\|_\infty^2 m
\end{split}
\end{equation}
where in the last step we used the fact that $\|Y\|_{\frob}^2 \leq n \zeta^2.$

%where ${\rm SS}_j \coloneqq \sum_{i \in [n]} y_{ij}^2$ for $j \in [d]$. %is as defined in~\eqref{def:SSj}. 
%In the sequel, we will use the shorthand
%\begin{equation}\label{def:SSj}
%	{\rm SS}_j \coloneqq \sum_{i \in [n]} y_{ij}^2 \mbox{ for } j \in [d].
%\end{equation}
%Next we deal with the sum over $t \in [n - m]$ which requires additional work.  Summing over $j$, we get
%\begin{equation}\label{eq:rembnd}
%\bE {\rm Rem}_{2} \le \frac{\|v\|_\infty^2 \cdot m}{n - 1} \sum_{j \in [d]} {\rm SS}_j.
%\end{equation}
Next we bound the expectation of $S_{n - m}$ towards which we will once again
use the event $\cG_{2, t}$ from~\eqref{def:G1sG2s} and its consequences as in~\eqref{eq:Ctopbnd}. We can write
\begin{equation}\label{eq:Sn-ln}
\begin{split}
S_{n - m} &\le \sum_{t = 1}^{n - m} \frac1{n - t + 1}\,\|D_t\|_{{\rm op}} \|\mvv[\cA_t]^{\t}Y_t\|^2 \\ %= \sum_{t = 1}^{n - m} \frac1{n - t + 1}\,\lambda_{{\rm max}}(D_t)\|\mvv[\cA_t]^{\t}Y_t\|^2\\
&\stackrel{\eqref{eq:Ctopbnd}}{\le} \sum_{t = 1}^{n - m} \frac1{n - t + 1} \, \frac{2\kappa}{n - t + 1}\,\|\mvv[\cA_t]^{\t}Y_t\|^2 + \sum_{t = 1}^{n - m} \frac1{n - t + 1} \, \|\mvv[\cA_t]^{\t}Y_t\|^2 \1_{\cG_{2, t}^c}\\
&\eqqcolon S_{n - m}^\star + {\rm Rem}.
\end{split}
\end{equation}
Using the second moment bound in Lemma~\ref{lem:srsworappli}, we then obtain
\begin{equation}\label{eq:ESn-elln}
\begin{split}
\bE\left[S_{n - m}^\star\right] &\le \frac{C\|\mvv\|_\infty^2\kappa}{n(n - 1)} \cdot \|Y\|_{{\rm
Frob}}^2 \sum_{t = 1}^{n - m} \frac{t - 1}{n - t + 1} \le \frac{C \|\mvv\|_\infty^2 \kappa\cdot \log en}{n -
1} \|Y\|_{\frob}^2 \\ & \leq C \zeta^2\|\mvv\|_\infty^2 \kappa \log en.
\end{split}
\end{equation}
where again in the last step we used the fact that $\|Y\|_{\frob}^2 \leq n \zeta^2.$

As to $\bE \left[{\rm Rem}\right]$, first notice that for {\em any possible choice} of $\cA_t \subset [n]$,
\begin{equation*}
\begin{split}
\|\mvv[\cA_t]^{\t}Y_t\|^2 &= \sum_{j \in [d]}\|\mvv[\cA_t]^{\t}Y_t[ \,\, : j]\|^2 \le \sum_{j \in
[d]}\|\mvv[\cA_t]\|^2 \|Y_t[ \,\, : j]\|^2 \le \|\mvv\|^2 \sum_{j \in [d]}\|Y_t[ \,\, : j]\|^2 \\
&= \|\mvv\|^2 \|Y_t\|_{\frob}^2 \stackrel{\eqref{def:zeta}}{\le} (n - t + 1)\zeta^2 \|\mvv\|^2,
\end{split}
\end{equation*}
where in the second step we used the Cauchy-Schwarz inequality. Plugging this bound into the
expression of ${\rm Rem}$, we get
\begin{equation*}
\begin{split}
\bE \left[{\rm Rem}\right] \le \zeta^2 \|\mvv\|^2 \sum_{t = 1}^{n - m} \P[\cG_{2, t}^c]  \stackrel{\eqref{eq:smallevent2}}{\le} C d n \zeta^2 \|\mvv\|^2 \exp\left(-c\:(\kappa\zeta^2)^{-2}m\right).
\end{split}
\end{equation*}
Combining this with~\eqref{eq:ESn-elln} and plugging the resulting bound into~\eqref{eq:Sn-ln} we
obtain after taking expectations on both sides,
\begin{equation*}
\begin{split}
%\bE S_{n - m} \le 	\frac{4\|v\|_\infty^2 \cdot \log n}{\lambda_{{\rm min}}(Y^{\t} Y)} \sum_{j \in [d]}{\rm 
%	SS}_j + 2d n\,\zeta^2 \exp\left(-\frac{m \lambda_{\min}^2\big(Y^{\t}Y\big)}{8 n^2}\right).
\bE \left[S_{n - m}\right] \le C \zeta^2\|\mvv\|_\infty^2 \kappa \log en + C d n \zeta^2 \|\mvv\|^2 \exp\left(-c\:(\kappa\zeta^2)^{-2}m\right).
\end{split}
\end{equation*}
Together with~\eqref{eq:rem1j} and subsequently~\eqref{eq:Sn}, this implies
\begin{align*}
\bE \left[S_{n}\right] \le C \zeta^2\|\mvv\|_\infty^2 (m + \kappa \log en) + C d n \zeta^2 \|\mvv\|^2 \exp\left(-c\:(\kappa\zeta^2)^{-2}m\right).
\end{align*}
Along with~\eqref{eq:QF1_bnd}, the above bound yields~\eqref{eq:idealclt2}.
%following bound on $\bE\left[\mathcal{QF}\right]$. 
%\begin{equation}\label{eq:goodbd}
%d\|\mvv\|_{\infty}^2 \log en + C \|v\|_\infty^2 \kappa \log n + C d n\ \exp\left(-c\:\frac{m}{\kappa^2}\right) + C \|v\|_\infty^2 m.
%\end{equation} 
%This shows~.

\smallskip

\noindent{\em Proof of~\eqref{eq:idealclt3}}. Finally it remains to show the facts about $T_n$. Let us
denote the random variable in question by $T_n$. %It is enough to show that
%\[
%T_{n}=\sum_{t=1}^{n}\frac1{\abs{\cA_t}}\norm{\mveps_{\cA_t}}^{2} \text{ is concentrated at } \norm{\mveps}^{2}
%\]
%and
%\[
%\sum_{t=1}^{n} \frac1{\abs{\cA_t}} \bE \mveps_{\cA_t}^{\t}X_{t}C_{t}X_{t}^{\t}\mveps_{\cA_t}\ll \norm{\mveps}^{2}.
%\]
Clearly,
\begin{equation*}
T_{n}=\sum_{t \in [n]} v_{p_t}^{2}\cdot \sum_{s \in [t]}\frac1{n-s+1} = \sum_{t \in [n]}v_{p_t}^{2}\cdot H(t)
\end{equation*}
where
\begin{equation*}
H(t) = \sum_{s \in [t]}\frac1{n-s+1} = \sum_{s \in [n]}\frac1{s}\cdot \1_{t + s > n}\le H(n) \text{ for all } t.
\end{equation*}
Hence,
\begin{equation}\label{redef:Tn}
T_{n} = \sum_{j \in [n]} v_{j}^{2}\cdot H(\theta_{j})
\end{equation}
where $\theta=\pi^{-1}$ is the inverse permutation of $\pi$ defined as $\pi(t) = p_t$. Since $\theta$ is
also a uniformly random permutation, note that,
\begin{equation}\label{eq:expectHn}
\bE \left[H(\theta_1)\right] = \frac1n\sum_{t \in [n]}\sum_{s \in [n]} \frac1{s}\cdot \1_{t + s>n} = \frac1n \, \sum_{s \in [n]}\frac1{s}\cdot \sum_{t \in [n]}\1_{t + s > n} = 1.
\end{equation}
\begin{comment}
{\color{red}
Also,
\begin{equation*}
\bE H(\theta_1)^2 = \frac1n\sum_{t=1}^{n}H(t)^{2} \le H(n)
\end{equation*}
was this a typo? remove this..}
\end{comment}
Also,
\begin{equation}\label{eq:var_expr}
\bE \left[H(\theta_1)^2\right] = \frac1n\sum_{t \in [n]}H(t)^{2} \le H(n)^2
\end{equation}
and
\begin{equation}\label{eq:cov_expr}
\begin{split}
\bE \left[H(\theta_1)H(\theta_{2})\right] &= \frac1{n(n-1)}\sum_{t \in [n]} H(t)\cdot \sum_{s\neq
t}H(s) = \frac1{n(n-1)}\sum_{t \in n}H(t) (n - H(t)) \\
&= \frac1{n-1}(n - \bE \left[H(\theta_1)^2\right]) \le \frac n{n - 1}.
\end{split}
\end{equation}
Thus $\bE \left[T_{n}\right] = \norm{\mvv}^{2}$ in view of~\eqref{redef:Tn} and
\eqref{eq:expectHn}. Also,
\begin{align*}
\var[T_{n}] & = \bE \left[T_n^2\right] - \|\mvv\|^4 = \bE \Big[\sum_{j \in [n]} v_j^2 H(\theta_j)\Big]^2 - \|\mvv\|^4                                                                                                    \\
& = \Big(\sum_{t \in [n]} v_t^4\Big) \bE \left[H(\theta_1)^2\right] + \Big(\sum_{1 \le s \neq t \le n} v_s^2 v_t^2 \Big) \cdot \bE \left[H(\theta_1)H(\theta_{2})\right]- \|\mvv\|^4                         \\
& \stackrel{\eqref{eq:var_expr} +~\eqref{eq:cov_expr}}{\leq}  \|\mvv\|_{\infty}^2 \|\mvv\|^2 H(n)^2 + \Big(\frac{n}{n - 1} - 1 \Big)\|\mvv\|^4 \\
& \leq \|\mvv\|^4 \Big(\frac{\|\mvv\|_{\infty}^2}{\|\mvv\|^2} (\log en)^2 +
\frac{1}{n - 1}\Big).
%	(\bE H(\theta_1)^2-H(\theta_1)H(\theta_{2}))\sum_{i=1}^{n}\mvv_j^4 + (\bE H(\theta_1)H(\theta_{2})-1)\sum_{i,j}\mvv_{i}^{2}\mvv_{j}^{2}\\
%	&\le \left(\frac{H(n) \cdot \|\mvv\|_{\infty}^2}{\norm{\mvv}^2} + \frac{\log n}{n}\right)\cdot \norm{\mvv}^4 \ll (\bE T_{n})^{2}
\end{align*}
This finishes the proof of this lemma. 
\end{proof}

\section{Outlook}
Our analysis in this paper raises a few questions, which we mention here as possible directions to pursue for the future. 
\begin{enumerate}\setlength\itemsep{1ex}
\item It will be interesting to get a sub-Gaussian tail probability bound with the improved limiting variance that we get here.

\item Incorporating the simplification from the skeletal process in the algorithm itself might reduce the running time.

\item Incorporating randomization in an online version of the algorithm might help analyze the online process under {\em milder} assumptions, but it is not immediately apparent. In this article, we studied the offline version of the GSW design algorithm, and pivot randomization played a crucial role in the analysis. 

\item The GSW design algorithm can be thought of as a discrete localization process where we want to sample a random assignment vector from a distribution ``minimizing'' a given cost function. In the present 
case, the GSW design satisfies
\[
\max\left\{\phi\cdot \norm{\cov(\mvz)},(1-\phi)\cdot \norm{\cov(\xi^{-1}X^{\t}\mvz)}\right\}  \le 1,
\]
and the algorithm behaves as a discrete coordinate-wise localization process.
In general, the distribution of
the random assignment vector minimizing the cost function will be  unknown; it will depend on the covariate matrix $X$ and the robustness parameter
$\phi$. It will be interesting to incorporate the techniques from localization literature to develop an algorithmic construction for a general cost function.

\item  Our approach to finding and analyzing the skeletal process is general. This can be used to analyze  the Horvitz--Thompson estimator based on the GSW design in non-uniform probability assignment 
cases (see Remark~\ref{rem:mainres}). %(see also~\cite[Section~A3.1]{harshaw2019balancing}). 
It is, therefore, natural to ask if one can 
generalize the strategy in different directions, e.g., in handling non-linear constraints.
\end{enumerate}

\vspace*{1ex}
\noindent{\bf Acknowledgments.}
SC was supported by the NSF Grant DMS-1916375. PD was supported partially by the Campus  Research Board Grant RB23016. SG's research was supported by the SERB grant SRG/2021/000032, a grant from the Department of Atomic Energy, Government of India, under project 12--R\&D--TFR--5.01--0500 and in part by a grant from the Infosys Foundation as a member of the Infosys-Chandrasekharan virtual center for Random Geometry. A significant part of this research was accomplished when SC visited the School of Mathematics at the Tata Institute of Fundamental Research (TIFR), Mumbai. SG and SC are also grateful to the International Centre for Theoretical Sciences (ICTS), Bengaluru, for their kind hospitality during the early phase of the project. The authors thank  Dan Spielman, Fredrik Sävje,
Christopher Harshaw and Peng Zhang for their encouraging and valuable comments on the first version of the manuscript. 

\begin{comment}

{\color{red} Now use this lemma to show the concentration..}

{\color{red} integrate the $\delta_t$ close to $1$ in this part...}

{\color{red}
We will define $$\kappa \coloneqq \frac{n}{\lambda_{{\rm min}}(Y^{\t}Y)}.$$

All the conditions are satisfied if
\begin{itemize}
\item $$d \|v\|_{\infty}^2 \kappa^3 (\log n)^2 \rightarrow 0.$$

\item $$n \geq d/\varepsilon_n^2$$

\item $$\frac{d}{\kappa \log n \:\varepsilon_n^2} \rightarrow 0.$$
\end{itemize}

We can replace $\log n$ by $n^{\delta}$ for some small $\delta$...
}

\end{comment}

\bibliographystyle{chicago}
%\bibliography{references}

\end{document}